\documentclass{amsart}
\usepackage[english]{babel}

\usepackage[latin1]{inputenc}
\usepackage[T1]{fontenc}

\usepackage{amssymb}
\usepackage[arrow,matrix,cmtip,curve]{xy}
\usepackage[colorlinks=true, linkcolor=blue, citecolor=magenta,
urlcolor=green]{hyperref}
\usepackage{paralist}

\title[Anosov representations: Domains of discontinuity]{Anosov representations: Domains of
    discontinuity and applications}
\author{Olivier Guichard}
\address{CNRS, Laboratoire de Math\'ematiques d'Orsay, Orsay cedex, F-91405\\
Universit\'e Paris-Sud, Orsay cedex, F-91405}
\author{Anna Wienhard}
\address{Department of Mathematics, Princeton University\\
Fine Hall, Washington Road, Princeton, NJ 08540, USA}
\thanks{A.W. was partially supported by the National Science
  Foundation under agreement No. DMS-0803216 and DMS-0846408.
  O.G. was partially supported by the Agence Nationale de la
  Recherche under ANR's projects Repsurf (ANR-06-BLAN-0311) and
  ETTT (ANR-09-BLAN-0116-01) and by the National Science Foundation
  under agreement No. DMS-0635607.}
\date{\today}
\keywords{hyperbolic groups, surface groups, Hitchin component, maximal representations, Anosov representations, higher Teichm\"uller spaces, compact Clifford-Klein forms, discrete subgroups of Lie groups, convex cocompact subgroups, quasi-isometric embedding} 
\numberwithin{equation}{section}

\theoremstyle{plain}

\newtheorem{thm}[equation]{Theorem}
\newtheorem{prop}[equation]{Proposition}
\newtheorem{lem}[equation]{Lemma}

\newtheorem{cor}[equation]{Corollary}
\newtheorem*{thm*}{Theorem}
\newtheorem*{prob*}{Problem}
\theoremstyle{definition}

\newtheorem{defi}[equation]{Definition}

\newtheorem{nota}[equation]{Notation}

\theoremstyle{remark}

\newtheorem{remark}[equation]{Remark}
\newtheorem{remarks}[equation]{Remarks}
\newtheorem{examples}[equation]{Examples}
\newtheorem{example}[equation]{Example}

\newcommand{\NN}{\mathbf{N}}
\newcommand{\ZZ}{\mathbf{Z}}
\newcommand{\RR}{\mathbf{R}}
\newcommand{\CC}{\mathbf{C}}
\newcommand{\HH}{\mathbf{H}}
\newcommand{\KK}{\mathbf{K}}

\newcommand{\PP}{\mathbb{P}}
\newcommand{\Sphr}{\mathbb{S}}

\newcommand{\Xx}{\mathcal{X}}

\newcommand{\Ff}{\mathcal{F}}

\newcommand{\G}{\Gamma}
\newcommand{\g}{\gamma}

\newcommand{\vcd}{\mathrm{vcd}}
\newcommand{\wdg}[1]{{\textstyle{ \bigwedge^{#1}}}}

\newcommand{\Ll}{\mathcal{L}}
\newcommand{\Gr}{\mathrm{Gr}}
\newcommand{\dev}{\mathrm{dev}}
\newcommand{\hol}{\mathrm{hol}}
\newcommand{\PSL}{\mathrm{PSL}}
\newcommand{\SL}{\mathrm{SL}}
\newcommand{\Sp}{\mathrm{Sp}}
\newcommand{\SO}{\mathrm{SO}}

\newcommand{\diag}{\mathrm{diag}}

\renewcommand{\O}{\mathrm{O}}
\newcommand{\PGL}{\mathrm{PGL}}

\newcommand{\SU}{\mathrm{SU}}

\newcommand{\U}{\mathrm{U}}
\newcommand{\PSp}{\mathrm{PSp}}
\newcommand{\PSO}{\mathrm{PSO}}
\newcommand{\GL}{\mathrm{GL}}

\newcommand{\tr}{\mathrm{tr}}
\newcommand{\id}{\mathrm{Id}}

\newcommand{\pr}{\pi}

\newcommand{\Span}{\mathrm{Span}}

\renewcommand{\hom}{\mathrm{Hom}}
\newcommand{\stab}{\mathrm{Stab}}

\newcommand{\moins}{\smallsetminus}

\newcommand{\sep}{,\,}
\newcommand{\longto}{\longrightarrow}

\newcommand{\bqn}{\begin{equation*}}
\newcommand{\eqn}{\end{equation*}}

\newcommand{\bq}{\begin{equation}}
\newcommand{\eq}{\end{equation}}

\def\bal#1\eal{\begin{align}#1\end{align}}
\def\baln#1\ealn{\begin{align*}#1\end{align*}}

\newcommand{\homPano}{\hom_{P\textup{-Anosov}}}

\newcommand{\flow}{\widehat{\Gamma}}
\newcommand{\rank}{\mathrm{rk}}

\begin{document}

\begin{abstract}
The notion of Anosov representations has been introduced by Labourie in 
his study of the Hitchin component for $\SL(n,\RR)$. 
Subsequently, Anosov representations have been studied mainly for
surface groups, in particular in the context of higher Teichm\"uller
spaces, and for lattices in $\SO(1,n)$.
In this article we extend the notion of Anosov representations to representations of arbitrary word hyperbolic groups and start the systematic study of their geometric properties. 
In particular, given an Anosov representation $\Gamma \to G$ we explicitly construct open subsets  of compact $G$-spaces, on which $\Gamma$ acts properly discontinuously and with compact quotient. 

As a consequence we show that higher Teichm\"uller spaces 
parametrize 
locally
homogeneous geometric structures on compact manifolds. 
We also obtain applications regarding
(non-standard) compact Clifford-Klein forms and compactifications of locally symmetric spaces of infinite volume.  
\end{abstract}

\maketitle

\tableofcontents{}

\section{Introduction}
\label{sec_intro}

The concept of Anosov structures has been introduced by Labourie \cite{Labourie_anosov}, and is, in some sense, a dynamical analogue of locally homogeneous geometric structures for manifolds %%% commented this: $M$
 endowed with an Anosov flow. Holonomy representations of Anosov structures are called Anosov representations. They are defined for representations into any real semisimple Lie group $G$. Anosov representations have several important properties. For example, the set of Anosov representations is open in the representation variety, and any Anosov representation is a quasi-isometric embedding. 
The notion of Anosov representations has turned out to be very useful
in the study of surface group representations, in particular in the
study of higher Teichm\"uller spaces
\cite{Barbot_anosov, Burger_Iozzi_Labourie_Wienhard, Burger_Iozzi_Wienhard_survey,
  Guichard_Wienhard_InvaMaxi, Labourie_anosov}.  
  Here we put the concept of Anosov representations into the broader
  context of finitely generated word hyperbolic groups. In this
  framework Anosov representations are generalizations of convex
cocompact subgroups of rank one Lie groups to the context of discrete subgroups of Lie groups of higher rank.

 Let us recall that given a real semisimple Lie group $G$, a discrete subgroup $\Lambda < G$ is said to be convex cocompact if there exists a $\Lambda$-invariant convex subset in the symmetric space $G/K$, on which $\Lambda$ acts properly discontinuously with compact quotient. 
A representation $\rho:\Gamma \to G$ of a 
 group $\Gamma$ into $G$ is said to be convex cocompact if $\rho$ has finite kernel, and $\rho(\Gamma)$ is a convex cocompact subgroup of $G$. 

When $G$ is of real rank one, i.e.\ when $G/K$ has negative
sectional curvature, the set of convex cocompact representations is an
open subset of the representation variety $\hom(\Gamma, G)/G$. In that
case, being convex cocompact is equivalent to being a
quasi-isometric embedding \cite{Bourdon_conforme} (see Theorem~\ref{thm_intro:Ano_rk_one}). 
This implies in particular that $\Gamma$ is a word hyperbolic group. 
Important examples of convex cocompact representations are discrete embeddings of surface groups into $\PSL(2,\RR)$, quasi-Fuchsian representations of surface groups into $\PSL(2,\CC)$, embeddings of free groups as Schottky groups, or embeddings of uniform lattices. 
 Anosov
representations into Lie groups of rank one are exactly convex cocompact representations (see Theorem~\ref{thm_intro:Ano_rk_one})

For Lie groups $G$ of real rank $\geq 2$, a rigidity
result of Kleiner-Leeb \cite{Kleiner_Leeb} and Quint \cite{Quint_cc}
says that the class of convex cocompact subgroups of $G$ reduces to
products of uniform lattices and of convex cocompact subgroups in 
rank one Lie groups. On the other hand, there is an abundance of examples of Anosov representations of surface groups, uniform lattices in $\SO(1,n)$ or hyperbolic groups into Lie groups of higher rank. 
 
\medskip 

 In the first part of the article we extend the definition of Anosov
 representations to representations of arbitrary finitely generated
 word hyperbolic groups $\Gamma$ (into semisimple Lie groups) and
 establish their basic 
 properties. In the second part we develop a
 quite intriguing 
geometric picture for Anosov representations. In
 particular, given an Anosov representation $\rho: \Gamma \to G$ we
 construct a $\Gamma$-invariant open subset of a compact $G$-space, on
 which  $\Gamma$ acts properly discontinuously with compact quotient.  
In the third part we discuss several applications of this
construction, in particular to higher Teichm\"uller spaces, to the
existence of non-standard Clifford-Klein forms of homogeneous spaces,
and  to compactifications of locally symmetric spaces of infinite volume.

We now describe the contents of the three parts 
in more detail. 

\subsection{Anosov representations}
We quickly mention the definition of Anosov representations  in
the context of manifolds given in \cite{Labourie_anosov} (see Section~\ref{sec:defiAR} below for
details). 

Let $G$ be a semisimple Lie group and $(P^+,P^-)$ a pair of opposite parabolic subgroups of $G$. 
Let $M$ be a compact manifold equipped with an Anosov flow. A
representation $\rho:\pi_1(M) \to G$ is said to be a $(P^+,P^-)$-Anosov
representation if the associated  $G/(P^+ \cap P^-)$-bundle over $M$ admits
a section  $\sigma$ which is constant along the flow and with certain contraction
properties. 

When $N$ is a negatively curved
compact Riemannian manifold, and $M = T^1 N$ is its unit tangent
bundle equipped with the geodesic flow $\phi_t$, there
is a natural projection $\pi: \pi_1(M) \to \pi_1(N)$; 
a representation $\rho: \pi_1(N) \to G$ is said to be $(P^+, P^-)$-Anosov,
if $\rho \circ \pi$ is $(P^+, P^-)$-Anosov.

In order to extend the notion of Anosov representations to arbitrary finitely generated 
word hyperbolic groups $\G$, we replace $T^1 \widetilde{N}$ by the
flow space $\flow$, introduced by Gromov \cite{Gromov_hyp}
and developed by Champetier \cite{Champetier}, Mineyev
\cite{Mineyev_flow} and others. The flow space $\flow$ is a  proper
hyperbolic metric space with an action of $\Gamma \times \RR \rtimes
\ZZ/2\ZZ$, where the $\RR$-action (the flow) is free and such that $\flow/\RR$ is
naturally homeomorphic to $\partial_\infty \Gamma^{(2)}$. Here
$\partial_\infty \Gamma^{(2)}$ denotes the space of distinct points in
the boundary at infinity  $\partial_\infty \Gamma$ of $\Gamma$. 
In analogy to the above, a representation $\rho:\Gamma \to G$ is said to be $(P^+,P^-)$-Anosov
 if the associated  $G/(P^+ \cap P^-)$-bundle over $\G \backslash \flow$ admits
a section $\sigma$ which is constant along $\RR$-orbits and with certain contraction
properties. 

The conjugacy class of a pair of opposite parabolic subgroups $(P^+, P^-)$ is completely
determined by $P^+$. We thus say that a representation is $P^+$-Anosov
if it is 
$(P^+, P^-)$-Anosov. We say that a representation is Anosov if it is $P^+$-Anosov for some proper parabolic subgroup $P^+ < G$. 

\begin{examples}
\noindent
  \begin{asparaitem}
    \item Let $G$ be a split real simple Lie group and let $\Sigma$ be a closed connected orientable surface of genus $\geq 2$. Representations $\rho: \pi_1(\Sigma) \to G$ in the Hitchin component are $B$-Anosov, where $B< G$ is a Borel subgroup.

\item  The holonomy representation $\rho: \pi_1(M) \to \PGL(n+1, \RR)$ of a
  convex real projective structure on an $n$-dimensional orbifold  $M$ is
  $P$-Anosov, where $P<  \PGL(n+1, \RR)$  is the stabilizer of a line.
  
  For more details, see Section~\ref{sec:examples}. 
  \end{asparaitem}
\end{examples}

We extend Labourie's result on the stability of Anosov
representations to the more general framework of
word hyperbolic groups:

\begin{thm} (Theorem~\ref{thm:Ano_open})
Let $\Gamma$ be a finitely generated word hyperbolic group. Let $G$ be a semisimple Lie group and $P^+ < G$ be a parabolic subgroup. 
 The set of $P^+$-Anosov representations is open in $\hom(\Gamma, G)$. 
\end{thm}

An immediate consequence of a representation $\rho:\Gamma \to G$ being
$P^+$-Anosov is the existence of continuous $\rho$-equivariant maps
$\xi^+: \partial_\infty \Gamma \to G/P^+$ and  $\xi^-: \partial_\infty
\Gamma \to G/P^-$. We call these maps \emph{Anosov maps}. They have the following properties:  
\begin{enumerate}
\item for all $(t,t')\in  \partial_\infty \Gamma^{(2)}$, the pair
  $(\xi^+(t), \xi^-(t'))$ is in the (unique) open $G$-orbit in $G/P^+ \times G/P^-$. 
\item for all $t \in \partial_\infty \Gamma$, the pair $(\xi^+(t),
  \xi^-(t))$ is contained in the (unique) closed $G$-orbit in  $G/P^+ \times G/P^-$.
\item they satisfy some contraction property with respect to the flow. 
\end{enumerate}

A representation admitting continuous equivariant maps $\xi^+, \xi^-$
with the above properties is easily seen to be $P^+$-Anosov. 
A consequence of this characterization of Anosov representations in terms of equivariant maps is 
\begin{cor}
Let $\Gamma$ be a finitely generated word hyperbolic group, and let $\Gamma'<\Gamma$ be of finite index. Then a representation $\rho: \Gamma \to G$ is $P^+$-Anosov if and only if the restriction of $\rho$ to $\Gamma'$ is $P^+$-Anosov. 
\end{cor}
A corollary of this and \cite{Burger_Iozzi_Labourie_Wienhard} is (see
also Remark~\ref{rem:redmaxtosymp}) 
\begin{example}[\cite{Burger_Iozzi_Wienhard_anosov}] 
Let $G$ be a classical Lie group of Hermitian type, and let $\rho:\pi_1(\Sigma) \to G$ be a maximal representation, then $\rho$ is $P$-Anosov, where $P$ is the stabilizer of a point in the Shilov boundary of the bounded symmetric domain associated to $G$.
\end{example}

A pair of maps $\xi^+, \xi^-$ satisfying the first two of the above
properties is said to be \emph{compatible}.  
We show that for a generic representation the contraction property involving the flow
is satisfied by any pair of equivariant continuous compatible 
maps.

\begin{thm}\label{thm_intro:maps} (Theorem~\ref{thm:Ano_Zd})
  Let $\Gamma$ be a finitely generated word hyperbolic group. Let $G$ be a semisimple Lie
  group and $P^+<G$ a parabolic subgroup. 
  Let $\rho: \G \to G$ be a Zariski dense  representation. Suppose that
  $\rho$ admits a pair  of equivariant continuous compatible maps
  $\xi^+: \partial_\infty \G \to G/ P^+$, $\xi^-: \partial_\infty \G \to G/ P^-$.
   
  Then the representation $\rho$ is $P^+$-Anosov and $(\xi^+,
  \xi^-)$ are the associated Anosov maps.
\end{thm}

Note that the statement of Theorem~\ref{thm_intro:maps} does not
involve the flow space of the hyperbolic group, but only its boundary
at infinity. 

A consequence of Theorem~\ref{thm_intro:maps} is that any representation admitting a pair  of equivariant continuous compatible maps is Anosov as a representation into 
the Zariski closure of its image. 
This leads to the problem of determining when the composition of an
Anosov representation with a Lie group homomorphism is still Anosov.
 
More precisely,  if $\phi:G \to G'$ is an embedding and $\rho:\Gamma \to G$ is a $P^+$-Anosov
representation, when is the composition $\phi\circ \rho: \Gamma \to
G'$ an Anosov representation and with respect to 
which  parabolic subgroup $P^{\prime +}<G'$?

When $G$ is a Lie group of rank one,
an answer has been given by Labourie \cite{Labourie_anosov} (see also
Proposition~\ref{prop:inj_rk_one} below).
We provide an answer to this question for general semisimple Lie groups $G$ in Section~\ref{sec:homo}. 
 Here we just note one 
 consequence, which plays an important role for the construction of domains of discontinuity.

\begin{prop}\label{prop_intro:ano_Qzero}
  A representation $\rho: \G \to G$ is Anosov if and only if there
  exists a real vector space $V$ with a non-degenerate indefinite quadratic form
  $F$ and a homomorphism $\phi: G
  \to \O(F)$ such that $\phi \circ \rho$ is $Q_0$-Anosov, where
  $Q_0<\O(F)$ is the stabilizer of an $F$-isotropic line in
  $V$.
\end{prop}

We introduce so called $L$-Cartan projections
as a new tool 
to study Anosov representations (see
Section~\ref{sec:lifting-section-gm} for details). These are
continuous maps from $\G \backslash \flow \times \RR$ with values in
the closure of 
a Weyl chamber  of  
$ L = P^+ \cap P^-$, which are defined whenever a representation
$\rho$ admits a continuous section $\sigma$ of the $G/L$-bundle over
$\G \backslash \flow$ that is flat along
$\RR$-orbits. 
The $L$-Cartan projections control the contraction properties of the section $\sigma$. 
 They play a central role 
 in the discussion  in Section~\ref{sec:homo}. 

We now turn our attention to the geometric properties of Anosov representations. 
The reader is referred to Sections~\ref{sec:discreteness} and \ref{sec:proximality} for definitions and details. 

\begin{thm}\label{thm_intro:geom_prop} (Theorems~\ref{thm:Ano_QIE} and \ref{thm:Ano_AMS})
Let $\Gamma$ be a finitely generated word hyperbolic group, $G$ a real semisimple Lie group
and $\rho:\Gamma \to G$ an Anosov representation.
Then 
\begin{enumerate}
\item the kernel of $\rho$ is finite and the image of $\rho$ is discrete. 
\item the map $\rho: \Gamma \to G$ is a quasi-isometric embedding,
  with respect to the word-metric on $\Gamma$ and any (left)
  $G$-invariant Riemannian metric on $G$.
\item the representation $\rho$ is well-displacing. 
\item the representation $\rho$ is (AMS)-proximal (Definition~\ref{defi:AMS}).
\end{enumerate}
\end{thm}

As a consequence of this and \cite{Bourdon_conforme} we obtain the following characterization of Anosov representations in Lie groups of rank one: 
\begin{thm}\label{thm_intro:Ano_rk_one} (Theorem~\ref{thm:Ano_rk_one})
  Let $\Gamma$ be a  finitely generated word hyperbolic group and $G$ a real semisimple Lie group with  $\rank_\RR G= 1$. For a representation $\rho: \G \to G$  the following statements are equivalent:
  \begin{enumerate}
  \item $\rho$ is Anosov.
  \item There exists a continuous, $\rho$-equivariant, and 
    injective map $\xi: \partial_\infty \G \to G/P$. 
      \item $\rho: \Gamma \to G$ is a quasi-isometric embedding,  with
    respect to the word metric on $\Gamma$ and any (left)
    $G$-invariant Riemannian metric on $G$. 
  \item $\rho$ is convex cocompact.
  \end{enumerate}
\end{thm}

\subsection{Domains of discontinuity}
The heart of this article is to construct, given an Anosov
representation $\rho:\Gamma \to G$, an open subset $\Omega$ of a
compact $G$-space, which  is $\Gamma$-invariant, and on which
$\Gamma$ acts properly discontinuously and with compact
quotient.

Let us recall that any real semisimple Lie group admits an Iwasa
decomposition $G= KAN$, where $K$ is a maximal compact subgroup and
$A$ the $\RR$-split part of a  
Cartan subgroup 
 and $N$ is the unipotent radical of a minimal parabolic subgroup $B$ containing $A$.

\begin{thm}\label{thm_intro:general_case}
  (Theorem~\ref{thm:dod_general_case})
Let $\rho: \G \to G$ be a $P$-Anosov representation for some proper parabolic subgroup $P<G$. Then there exists a $\G$-invariant open set $\Omega \subset G/AN$ such that 
\begin{enumerate}
\item the action of $\G$ on $\Omega$ is properly discontinuous, and 
\item the quotient $\G\backslash \Omega$ is compact. 
\end{enumerate}
\end{thm}

\begin{remarks}\label{rem:dod}
\noindent
\begin{asparaenum}
\item \label{empty}
The domain of discontinuity $\Omega$ is constructed explicitly using the Anosov maps associated to $\rho$, and depends on some additional combinatorial data. 
For specific examples the domain of discontinuity $\Omega$ might be
empty. 
 One such example is an Anosov representation $\rho: \pi_1(\Sigma) \to \PSL(2,\RR)$, where $\Sigma$ is a closed connected oriented surface of genus $\geq 2$. Then $\pi_1(\Sigma)$ acts minimally on $\PSL(2,\RR)/AN =\PP^1(\RR)$. 
See Remark~\ref{rem:examples_orth} for other examples.

We describe conditions for $\Omega$ to be nonempty in Section~\ref{sec:doma-disc-gh}, see also 
Corollary~\ref{cor:nonemptyVCD}, Theorem~\ref{thm_intro:dod_freegrp}
and Theorem~\ref{thm:dod_surf} below. The domain of discontinuity can
always be ensured to be nonempty by embedding $G$ into a larger Lie
group $G'$.
\item The domain of discontinuity $\Omega$ is not unique. For Anosov
  representations into $\SL(n,\KK)$, $\KK = \RR$ or $\CC$, we describe in Section~\ref{sec:desclim} several different domains of discontinuities in $G/AN$.

\item Recall that a minimal parabolic subgroup $B<G$ admits a decomposition as $B=
MAN$ with $M$ being the centralizer of $A$ in $K$. 
In particular, the compact
$G$-space $G/AN$ is a compact fiber bundle over $G/B$ and hence over
$G/Q$ for any parabolic subgroup $Q<G$. In many cases the domain of
discontinuity $\Omega$ is indeed the pull back of a domain of
discontinuity in some $G/Q$.
\end{asparaenum}
\end{remarks}

For free groups and for surface
groups we obtain 
\begin{thm}\label{thm_intro:free}\label{thm_intro:dod_freegrp}
Let $F_n$ be the free group on $n$ letters, and let $G$ be
a real semisimple Lie group. 
Assume that $\rho:F_n \to G$ is an Anosov representation. Then there
exists a nonempty $F_n$-invariant open subset $\Omega \subset
G/AN$ such that the action of $F_n$ on $\Omega$ is
properly discontinuous and cocompact. 
\end{thm}

\begin{thm}\label{thm_intro:surface}\label{thm:dod_surf}
Let $\pi_1(\Sigma)$ be the fundamental group of a closed connected orientable surface of genus $\geq 2$, and let $G$ be a real semisimple Lie group with no (almost) factor being locally isomorphic to $\SL(2,\RR)$. 
Assume that $\rho: \pi_1(\Sigma) \to G$ is an Anosov representation. Then
there exists a nonempty $\pi_1(\Sigma)$-invariant open subset $\Omega
\subset G/AN$ such that the action of $\pi_1(\Sigma)$ on $\Omega$ is
properly discontinuous and cocompact. 
\end{thm}
\begin{remarks}
\noindent
\begin{enumerate} 
\item In view of Remark~\ref{rem:dod}.\eqref{empty} the condition on $G$ in Theorem~\ref{thm_intro:surface} is optimal. 

\item Theorem~\ref{thm:dod_surf} was announced in \cite{Guichard_Wienhard_CRAS} in the form that $\Omega \subset G/B$. This is true in many cases, but in general our construction provides only a domain of discontinuity in $G/AN$ (which is a compact fiber bundle over $G/B$).

\item Theorem~\ref{thm_intro:dod_freegrp} holds more generally for
hyperbolic groups whose virtual cohomological dimension is one,
respectively Theorem~\ref{thm:dod_surf} holds for hyperbolic groups 
whose virtual cohomological dimension is two, see
Theorem~\ref{thm:dofFnGammag}.
\end{enumerate}
\end{remarks}

\medskip

We shortly describe the general strategy for the construction of the domain of discontinuity $\Omega$.

A basic example is when $\rho: \pi_1(\Sigma) \to \PSL(2,\CC)$
is a quasi-Fuchsian representation. Then, there is an equivariant
continuous curve $\xi: \partial_\infty \pi_1(\Sigma) \cong S^1 \to
\PP^1(\CC)$, whose image is a Jordan curve. The action of
$\pi_1(\Sigma)$ on the complement $\Omega = \PP^1(\CC) \moins
\xi(S^1)$  is free and properly discontinuous. The quotient
$\pi_1(\Sigma)\backslash\Omega$ has two connected components, both
of which are homeomorphic to the surface $\Sigma$.

When $\rho: \Gamma \to G$ is an Anosov representation into a Lie group
$G$ of rank one, the construction is a
straightforward generalization of this procedure. We consider the
equivariant continuous Anosov map $\xi: \partial_\infty (\Gamma) \to
G/P$, where $P$ is (up to conjugation) the unique proper parabolic
subgroup of $G$. In that case $\Omega=G/P \moins
\xi(\partial_\infty \Gamma)$.

In general, when $G$ is a Lie group of higher rank, the situation
becomes more complicated. 
We will have to consider the Anosov
map $\xi: \partial_\infty (\Gamma) \to G/P$ in order to construct a
domain of discontinuity in $G/P'$ with $P'<G$ being a different
parabolic subgroup.
 
  We start first with the case when $\rho: \G \to \O(F)$ is a
 $Q_0$-Anosov representation, where 
 $Q_0$ is the
stabilizer of an $F$-isotropic line in $V$. 
Let $\Ff_0 = G/Q_0$ be the space
of $F$-isotropic lines in $V$, let $\Ff_1 = G/Q_1$ be the space of
maximal $F$-isotropic subspaces of $V$, and let $\Ff_{01}$ be the space of
pairs consisting of an $F$-isotropic line and an incident maximal
$F$-isotropic space. There are two projections $\pr_i: \Ff_{01} \to
\Ff_i$, $i=0,1$.
Starting from the Anosov map $\xi: \partial_\infty \Gamma \to \Ff_0$, we consider the subset $K_\xi := \pr_1( \pr_0^{-1}(\xi(\partial_\infty \Gamma))) \subset \Ff_1$ and define 
$\Omega_\rho = \Ff_1 \moins K_\xi$. 

That the action of $\Gamma$ on $\Omega_\rho$ is properly discontinuous follows from the fact that $\rho(\G)$ has special dynamical properties. A more general construction of domains of discontinuity for discrete subgroups of $\O(F)$ with special dynamical behavior is described in Theorem~\ref{thm:dod_prox}. 
We deduce the compactness of the quotient $\Gamma\backslash \Omega_\rho$ from computations in homology.

Given an Anosov representation $\rho:\Gamma \to G$ into an arbitrary
semisimple Lie group, the first step is to choose an appropriate
finite dimensional (irreducible) representation of $G$ on a real
vector space $V$ preserving a non-degenerate indefinite symmetric
bilinear form $F$, such that the composition of $\rho$ with the
representation $\phi: G \to \O(F)$ is a $Q_0$-Anosov representation
$\phi\circ \rho: \Gamma \to \O(F)$; this is made possible by
Proposition~\ref{prop_intro:ano_Qzero}. Let $\Omega_{\phi\circ \rho}
\subset \Ff_1$ denote the domain of discontinuity whose construction
we just described. 
We show that there is always a maximal isotropic subspace $W \in \Omega_{\phi\circ \rho} \subset \Ff_1$ 
which is invariant by $AN$. The intersection of 
$\Omega_{\phi\circ \rho}$ with the $G$-orbit of $W$ is a domain of discontinuity $\Omega' \subset G/\stab_G(W)$, which can be lifted to obtain a domain of discontinuity $\Omega \subset G/AN$. 

\subsection{Applications} 

\subsubsection{Hitchin components, maximal representations and deformation spaces of geometric structures}

By the uniformization theorem, the Teichm\"uller space of a surface
$\Sigma$ can be identified with the 
moduli space of marked hyperbolic structures 
on $\Sigma$, and consequently with  a connected component of the space 
 $\hom(\pi_1(\Sigma), \PSL(2,\RR))/\PSL(2,\RR)$. In 1992
Hitchin defined the Teichm\"uller component, now called the
Hitchin component, a connected component of the space
$\hom(\pi_1(\Sigma), G)/G$ of representations into a split
real adjoint simple Lie group, which he proved to be homeomorphic to a ball
\cite{Hitchin}. Choi and Goldman \cite{Goldman_Choi} showed that the Hitchin component for
$\PSL(3,\RR)$ can be realized as the moduli space of convex real projective structures on $\Sigma$. 

Hitchin's work, using methods from the theory of Higgs bundles, does
not provide much information on the geometric significance of
representations in the Hitchin component. In recent years, due to
work of Labourie \cite{Labourie_anosov,
  Labourie_crossratio, Labourie_energy} 
 and of Fock and Goncharov \cite{Fock_Goncharov}, it has been shown that
 representations in Hitchin components have beautiful geometric
 properties, which generalize properties of  
classical Teichm\"uller space. 
Parallel to this, the study of maximal representations of $\pi_1(\Sigma)$
into Lie groups of Hermitian type
\cite{Bradlow_GarciaPrada_Gothen,Bradlow_GarciaPrada_Gothen_survey,
  Burger_Iozzi_Labourie_Wienhard, Burger_Iozzi_Wienhard_toledo,
  GarciaPrada_Gothen_Mundet, Gothen, Hernandez, Toledo_89}
showed that spaces of maximal representations also share several
properties with 
classical Teichm\"uller space, which itself is identified with
the space of maximal representations into $\PSL(2,\RR)$ \cite{Goldman_88}. Therefore, Hitchin components and spaces of maximal representations are also referred to as higher Teichm\"uller spaces.

Using the construction of domains of discontinuity for Anosov representations we can show that  Hitchin components parametrize connected components of 
deformation spaces of locally homogeneous $(G,X)$-structures on
compact manifolds. 
\begin{thm}\label{thm_intro:hitchin_other} (Theorem~\ref{thm:hitchin_other})
  Let $\Sigma$ be a closed connected orientable surface of negative Euler characteristic. Assume that $G$ is 
$\PSL(2n, \RR)$ ($n \geq 2$), 
 $\PSp(2n, \RR)$ ($n \geq
  2$), or 
 $\PSO(n,n)$ ($n\geq 3$), 
  and $X = \PP^{2n-1}(\RR)$; 
 or that $G$ is 
  $\PSL(2n+1, \RR)$ ($n \geq 1$), or 
 $\PSO(n,n+1)$ ($n\geq
  2$), and  $X = \mathcal{F}_{1,2n}(\RR^{2n+1}) = \{ (D,H) \in
  \PP^{2n}(\RR) \times \PP^{2n}(\RR)^* \mid D \subset H \}$.

  Then there exist a compact manifold $M$ and a connected component
  $\mathcal{D}$ of the deformation space of $(G,X)$-structures on $M$
  which is parametrized by the Hitchin component in $ \hom( \pi_1(\Sigma), G)/G$.
  \end{thm}
\begin{remarks}
 The Lie groups given in Theorem~\ref{thm_intro:hitchin_other} comprise (up to local isomorphism) all classical split real simple Lie groups. An analogous statement holds also for exceptional split real simple  Lie groups with $X = G/B$, where $B<G$ is the Borel subgroup. 

 We expect $M$ to have the homeomorphism type of a bundle over
 $\Sigma$ with  compact fibers. In the case when $G = \PSL(2n, \RR)$ ($n \geq 2$) or  
 $\PSp(2n, \RR)$ ($n \geq  2$) we prove in \cite{Guichard_Wienhard_DoDSymp} that $M$ is homeomorphic to the total space of an $\O(n)/\O(n-2)$-bundle over $\Sigma$.  
 For $\PSL(4,\RR)$ we refer the reader to
 \cite{Guichard_Wienhard_Duke}. The known topological relation 
 between $M$ and $\Sigma$ is the existence of a homomorphism $\pi_1(M)
 \to \pi_1( \Sigma)$; this homomorphism is in fact used to relate the
 component $\mathcal{D}$ to the Hitchin component, we refer to
 Theorem~\ref{thm:hitchin_other} for a precise statement.
 \end{remarks}

We also associate locally homogeneous $(G,X)$-
structures on compact manifolds to all maximal representations. We state the result in the case of the symplectic group.

\begin{thm}\label{thm_intro:max_sympl} (Theorem~\ref{thm:compspn})
  Let $\Sigma$ be a closed connected orientable surface of genus $\geq
  2$.
  
  Then for any connected component component $\mathcal{C}$ of the space of maximal representations, there exists a compact manifold $M$ and a connected component $\mathcal{D}$ of the deformation space of 
  $(\Sp(2n,\RR), \PP^{2n-1}(\RR))$-structures on $M$, which is parametrized by a Galois cover of $\mathcal{C}$. 
\end{thm}

\begin{remarks}
  The components of the space of maximal representations can have
  nontrivial fundamental group (see Section~\ref{sec:maxim-comp-sympl}
  for details).
  
 There are similar statements for components of the space of  maximal
 representations of $\pi_1(\Sigma)$ into other Lie groups $G$ of
 Hermitian type, with  
\begin{itemize}
\item $G= \SO(2,n)$, $X= \mathcal{F}_1(\RR^{2+n})$ the space of
  isotropic $2$-planes.
\item $G= \SU(p,q)$, $X\subset \PP^{p+q-1}(\CC)$ is the null cone for
  the Hermitian form.
\item $G= \SO^*(2n)$, $X$ is the null cone for the skew-Hermitian form.
\end{itemize}
Recall that $\SO^*(2n)$ can be realized as the automorphism group of a
skew-Hermitian form $V\times V \to \HH$ on an $n$-dimensional right
$\HH$-vector space $V$. 
  \end{remarks}
 
\subsubsection{Compactifications of actions on symmetric spaces}
The construction of domains of discontinuity we give is very flexible
as it applies to Anosov representations into arbitrary semisimple
Lie groups, in particular into complex groups. Using this
flexibility one can apply the construction to obtain natural
compactifications of non-compact quotients of symmetric spaces or
of other homogeneous spaces. We illustrate this in the case when $\rho:
\G \to \Sp(2n,\RR)$, $n\geq 2$, is a $Q_0$-Anosov representation, where
$Q_0$ is the stabilizer of a line in $\RR^{2n}$. 

Let $\mathcal{H}_{\Sp(2n,\RR)}$ denote the bounded symmetric domain associated to $\Sp(2n,\RR)$. 
Since $\rho(\G)$ is discrete, the action of $\G$ on $\mathcal{H}_{\Sp(2n,\RR)}$ is properly discontinuous. The quotient $M = \G\backslash \mathcal{H}_{\Sp(2n,\RR)} $ is not compact. 

\begin{thm}\label{thm_intro:compactify} (Theorem~\ref{thm:compactify})
Let $\rho:\G \to \Sp(2n,\RR)$ be a $Q_0$-Anosov representation. 
Then there exists a compactification $\overline{M}$ of $M = \G
\backslash \mathcal{H}_{\Sp(2n,\RR)} $ such that $\overline{M}$
carries a $(\Sp(2n,\RR), \overline{\mathcal{H}}_{\Sp(2n,\RR)})$-structure, where
$\overline{\mathcal{H}}_{\Sp(2n,\RR)}$ is the bounded symmetric domain
compactification, and the inclusion $M \subset \overline{M}$ is an 
$\Sp(2n,\RR)$-map.  
\end{thm}

The proof of this theorem relies on the fact that $\G \to \Sp(2n,\RR)
\to \Sp(2n,\CC)$ is a $Q_0$-Anosov representation, where $Q_0$ is the
stabilizer of a line in $\CC^{2n}$. With this, one constructs a
domain of discontinuity $\Omega$ in the space of complex Lagrangians
$\mathcal{L}(\CC^{2n})$,  which contains the image of
$\mathcal{H}_{\Sp(2n,\RR)}$ under the Borel embedding into
$\mathcal{L}(\CC^{2n})$, then $\overline{M} = \G \backslash (\Omega
\cap \overline{\mathcal{H}}_{\Sp(2n,\RR)})$.

\subsubsection{Compact Clifford-Klein forms} 
As we already mentioned, for some Anosov representations $\rho: \G \to
G$ the domain of discontinuity  turns out to be empty. Nevertheless
embedding $G$ into a bigger group $G'$ one can obtain a nonempty
domain of discontinuity $\Omega$ for $\rho': \G \to G \to G'$. In some
cases this nonempty domain of discontinuity coincides with a
$G$-orbit,  $\Omega = G/H$. 
Then $\Gamma \backslash \Omega = \Gamma \backslash G/H$ is a compact Clifford-Klein form. 

With this we recover several examples of compact Clifford-Klein forms (see Proposition~\ref{prop:CK_one}), including non-standard ones. In particular, using results of Barbot \cite{Barbot_component} we prove the following. 

\begin{thm}\label{thm_intro:CK_two} (Theorem~\ref{thm:CK_two})
Let $\G < \SO(1,2n)$ be a cocompact lattice, let $\SO(1,2n)< \SO(2,2n)$
be the standard embedding, and consider $\rho: \G \to \SO(1,2n) <
\SO(2,2n)$. Then any representation $\rho'$ in the connected component
of $\rho$ in $\hom(\G , \SO(2,2n))$ 
gives rise to a properly discontinuous and cocompact action on
the homogeneous space $\SO(2,2n)/\U(1,n)$.
\end{thm}

This extends a recent result of Kassel \cite[Theorem~1.1]{Kassel_deformation}.

\smallskip

\thanks{{\bf Acknowledgments.} A lot of the work for this article has
  been done during a visit of the first author to  the Mathematics
  Department at Princeton University, and when the first author was a member, and the second author was a visitor at the
  School of Mathematics at the Institute for Advanced Study. We thank both institutions for their hospitality and support.

We thank Marc Burger, Bill Goldman, Fran\c{c}ois Guéritaud, Misha Kapovich, and Fanny Kassel for interesting discussions and helpful remarks.}
 
\part{Anosov representations}
\label{part:AR}

\section{Definition}
\label{sec:defiAR}
In this section we recall the notion of Anosov representations, a
concept introduced by  Labourie \cite[Section~2]{Labourie_anosov}, and
we generalize it to representations of arbitrary finitely generated word hyperbolic
groups. 
  A variety of examples of Anosov representations are 
discussed in Section~\ref{sec:examples} 
(see also Remark~\ref{rem:examples_orth}).

\subsection{For Riemannian manifolds}
\label{sec:defiRiem}

Let $(N,g)$ be a closed negatively curved Riemannian manifold and let $M =
T^1 N$ be its unit tangent bundle, 
equipped with
the geodesic flow $\phi_t$ for the metric $g$. The geodesic flow is an Anosov flow. 
We denote by $\widehat{M} =
T^1 \tilde{N}$ the $\pi_1(N)$-cover of $M$ and by $\phi_t$ again the geodesic flow on
$\widehat{M}$.

Let $G$ be a semisimple Lie group and $(P^+, P^-)$ be a pair of
opposite parabolic subgroups\footnote{We review the structure theory of parabolic subgroups in Section~\ref{sec:parab-subgr-lie}. 
Pairs of opposite parabolic subgroups arise as follows: 
choose a semisimple element $g\in G$ and set
$\mathrm{Lie}( P^+)$ (resp.\ $\mathrm{Lie}( P^-)$) to be the sum of
eigenspaces of $\mathrm{Ad}(g)$ associated with eigenvalues of modulus $\geq 1$ (resp.\ $\leq 1$).}
 of $G$  and set $\Ff^\pm = G/P^{\pm}$. 
The subgroup $L = P^+ \cap P^-$ is the Levi subgroup of both $P^+$ and
$P^-$. The homogeneous space $\mathcal{X} =  G/L$ is the unique open
$G$-orbit in the product $\mathcal{F}^+ \times \mathcal{F}^- =
G/P^+ \times G/P^-$. From this product structure $\mathcal{X}$
inherits two $G$-invariant distributions $E^+$ 
and $E^-$: $(E^\pm)_{(x_+, x_-)} = T_{x_\pm}
\mathcal{F}^\pm$. 
As a consequence any
$\mathcal{X}$-bundle is equipped with two distributions which are denoted also 
by $E^+$ and $E^-$.

\begin{nota}
  Let $M$ be a topological space, $\G$ be a group, and let
  $\widehat{M}$ be the $\Gamma$-cover of $M$.  Let $\rho: \G \to G$ be
  a representation and let $\mathcal{S}$ be a $G$-space. We set
  \[ \mathcal{S}_\rho = \widehat{M} \times_\rho \mathcal{S} = \G\backslash (\widehat{M}
  \times \mathcal{S}),\] where $\G$ acts diagonally, as deck
  transformations on $\widehat{M}$ and via the representation $\rho$
  on $\mathcal{S}$.  The projection onto the first factor gives
  $\widehat{M} \times_\rho \mathcal{S}$ the structure of a flat $\mathcal{S}$-bundle
  over $M$. 
\end{nota}

\begin{defi}\label{defi:anosov_mfd}
A representation $\rho: \pi_1( N) \to G$ is said to be $(P^+, P^-)$-Anosov if
  \begin{enumerate}
  \item the flat bundle $\mathcal{X}_\rho$ admits a section $\sigma: M \to
    \mathcal{X}_\rho$ which is flat along flow lines (i.e.\ the restriction of
    $\sigma$ to any geodesic leaf is flat).
  \item The (lifted) action of $\phi_t$ on $\sigma^* E^+$
    (resp.\ $\sigma^* E^-$) is dilating (resp.\ contracting).
  \end{enumerate}
   The section $\sigma$ will be called \emph{Anosov section}.
\end{defi}

\begin{remarks}
\noindent
  \begin{asparaenum}[(a)]
   \item The second condition means more precisely that there exists a continuous family
    of norms $(\|\cdot\|_m)_{m \in M}$ on the fibers of the vector
    bundle $\sigma^* E^+ \to M$ (resp.\ $\sigma^* E^- \to M$) and
    positive constants $A,a$ such that, for any $t \in \RR_{\geq 0}$,
    $e \in \sigma^* E^+$ (resp.\ $e \in \sigma^* E^-$), with $\pi(e)=m$, one
    has $\| \phi_{-t} e\|_{\phi_{-t} m} \leq A e^{-a t} \|e\|_m$ (resp.\ $\|
    \phi_{t} e\|_{\phi_{t} m} \leq A e^{-a t} \|e\|_m$).
  \item Due to the compactness of $M$, the definition does not depend on the particular choice
    of $\|\cdot\|_m$ or the 
    particular parametrization of the flow on $M$.
     \item The Anosov section is uniquely determined
    (see Lemma~\ref{lem:uniqu}).
  \end{asparaenum}
\end{remarks}

\subsection{In terms of equivariant maps}
\label{sec:RiemEqu}

Let $N$, $M = T^1 N$ and $\widehat{M}$ be as before.

We denote by $\partial_\infty \tilde{N}$ the boundary at infinity of the universal cover $\tilde{N}$ of $N$ and by $\partial_\infty \pi_1(N)$ the boundary at infinity of $\pi_1(N)$. We can identify  $\partial_\infty \tilde{N}$ and  $\partial_\infty \pi_1(N)$. 
Since $N$ is negatively
curved, $\partial_\infty \tilde{N}$ is homeomorphic to a sphere. 

The space of geodesic leaves $\widehat{M} / \{\phi_t\}$ in $\widehat{M}$ is canonically
identified with 
\[\partial_\infty \tilde{N}^{(2)} = \partial_\infty
\tilde{N} \times \partial_\infty \tilde{N} \moins \{(t,t) \mid t
\in \partial_\infty \tilde{N} \};\] 
the
identification associates to a geodesic its endpoints in
$\partial_\infty \tilde{N}$.

Given a representation $\rho: \pi_1(N) \to G$, a section $\sigma$ of
$\mathcal{X}_\rho = \widehat{M} \times_\rho
\mathcal{X}$ is completely determined by its pullback 
$\hat{ \sigma}$ to $\widehat{M}$,
which is a $\rho$-equivariant map:
\[ \hat{ \sigma} : \widehat{M} \longto \mathcal{X}.\]

Conversely, such a $\rho$-equivariant map $\hat{\sigma}$ descends to a
section $\sigma$.
The section $\sigma$ is flat along flow lines if and only if
$\hat{\sigma}$ is $\phi_t$-invariant. In this case, one can consider
$\hat{ \sigma}$ as being defined on $\widehat{M} / \{\phi_t\} \cong
\partial_\infty \tilde{N}^{(2)}$:
\[\hat{ \sigma} = (\xi^+, \xi^-): \partial_\infty \tilde{N}^{(2)}
\longto \mathcal{X} \subset \mathcal{F}^+ \times \mathcal{F}^-.\]

The contraction property implies immediately (see \cite[Proposition~3.2]{Labourie_anosov} and \cite[Proposition~2.5]{Guichard_Wienhard_InvaMaxi}) that $\xi^+ :
\partial_\infty \tilde{N}^{(2)} \to \mathcal{F}^+$ factors through the
 projection to the first factor $\pi_1 : \partial_\infty \tilde{N}^{(2)} \to
\partial_\infty \tilde{N}$, i.e.\ $\xi^+$ is a map from
$\partial_\infty \tilde{N}$ to $\mathcal{F}^+$. Similarly $\xi^-$
factors through the projection to the second factor. Thus we get a pair of maps 
$\xi^+: \partial_\infty \tilde{N} \to \Ff^+$ and $\xi^-: \partial_\infty \tilde{N} \to \Ff^-$. 

\begin{remark} 
  In definition~\ref{defi:anosov_mfd} the dilatation property and
  contraction property on $E^+$ and $E^-$ have been exchanged compared
  to the original definition
  \cite[Section~2.0.1]{Labourie_anosov}. The convention chosen
  here are such that $\xi^+$ factors through the projection onto the first factor and
  also such that $\xi^+(t^{+}_{\g})$ is the attracting fixed point of
  $\g$ (see Lemma~\ref{lem:ano_proxi}). These two properties seem natural to us.
\end{remark}

\begin{defi}
 The maps $\xi^\pm:\partial_\infty \tilde{N} \to \Ff^\pm$ are said to be the \emph{Anosov maps} associated to the
 Anosov representation $\rho: \pi_1(N) \to G$. 
\end{defi}

Let $\tau^+$, $\tau^- : \widehat{M} \to \partial_\infty
\tilde{N}$ be the maps associating to a tangent vector the
endpoints at $+\infty$ and $-\infty$ of the corresponding
geodesic. The dilatation/contraction property in Definition~\ref{defi:anosov_mfd} 
translates into a dilatation property for $\phi_t$ on the 
family of tangent spaces  $(T_{\xi^+(\tau^+( \hat{m}))}
\mathcal{F}^+)_{\hat{ m} \in \widehat{ M}}$ (resp.\ contraction
on $(T_{\xi^-(\tau^-( \hat{m}))}
\mathcal{F}^-)_{\hat{ m} \in \widehat{ M}}$).

Conversely, one
can use the maps $\xi^+$, $\xi^-$ to express the Anosov property. 

\begin{defi}\label{defi:transverse}
A pair of points $(x^+, x^-) \in \Ff^+ \times \Ff^-$ is said to be \emph{transverse}, if 
$(x^+, x^-) \in \Xx \subset  \Ff^+ \times \Ff^-$. 

Given $x\in \Ff^+$ and $y \in \Ff^-$ we say that $y$ is \emph{transverse} to $x$ (and $x$ is transverse to $y$)  if $(x,y)$ is transverse. 
\end{defi}

\begin{prop}\label{prop:maps_Anosov}
  Let $\rho: \pi_1(N) \to G$ be a representation. Suppose that there
  exist maps  $\xi^+: \partial_\infty \tilde{N} \to \mathcal{F}^+$ and 
$\xi^-:\partial_\infty \tilde{N} \to   \mathcal{F}^-$ such that:
  \begin{enumerate}
  \item \label{item:cont} $\xi^+$ and $\xi^-$ are continuous and $\rho$-equivariant.
  \item \label{item:trans} For all $(t^+, t^-) \in \partial_\infty \tilde{N}^{(2)}$ the pair $(\xi^+( t^+), \xi^-(t^-))$ is
    transverse.
  \item For one (and hence any) continuous and equivariant family of norms
    $(\|\cdot\|_{\hat{m}})_{\hat{ m} \in \widehat{ M}}$ on 
 $(T_{\xi^+(\tau^+( \hat{m}))}
\mathcal{F}^+)_{\hat{ m} \in \widehat{ M}}$ (resp.\ 
 $(T_{\xi^-(\tau^-( \hat{m}))}
\mathcal{F}^-)_{\hat{ m} \in \widehat{ M}}$), the following
property holds: 
\begin{itemize}
\item there exist positive constants $A,a$ such that for all
$t\in \RR_{\geq 0}$,  $\hat{m} \in \widehat{M}$ and  $e \in
T_{\xi^+(\tau^+( \hat{m}))} \mathcal{F}^+$ (resp.\ $e \in T_{\xi^-(\tau^-( \hat{m}))} \mathcal{F}^-$):
\[ \| e\|_{\phi_{-t} \hat{m}} \leq A e^{-a t} \|e\|_{\hat{m}} \quad
(\text{resp.} \ \| e\|_{\phi_{t} \hat{m}} \leq A e^{-a t} \|e\|_{\hat{m}}). \]
  \end{itemize}
  \end{enumerate}
  Then the representation $\rho$ is $(P^+, P^-)$-Anosov and the pull-back
  to $\widehat{M}$ of the Anosov section $\sigma$  of $\mathcal{X}_\rho$ is the
  map $\hat{\sigma}: \widehat{M} \to \mathcal{X} \sep \hat{m} \mapsto
  (\xi^+(\tau^+( \hat{m})), \xi^-(\tau^-( \hat{m})))$.
\end{prop}

\begin{remark}
We will see later (Theorem~\ref{thm:Ano_Zd}) that in the case of Zariski dense representations the existence of $\xi^+$ and $\xi^-$ satisfying \eqref{item:cont} and
  \eqref{item:trans} is sufficient to ensure the Anosov property. 
\end{remark}

\subsection{For hyperbolic groups}
\label{sec:hyperbolic-groups}
In order to define the notion of Anosov representations $\rho: \G \to G$ for an arbitrary finitely generated word hyperbolic group $\G$ we need a replacement for the space $( \widehat{M}, \phi_t)$. 
Geodesic flows for hyperbolic groups were introduced by Gromov
\cite{Gromov_hyp} and later developed by Champetier \cite{Champetier},
Mineyev \cite{Mineyev_flow} and others. We recall the results
necessary for our purpose.

Let $\G$ be a finitely generated word hyperbolic group and let $\partial_\infty \Gamma$ denote its
boundary at infinity
(\cite[Chapitre~2]{Coornaert_Delzant_Papadopoulos}).
We set $  \partial_\infty \G^{(2)} =  \partial_\infty \G
\times  \partial_\infty \G\moins \{(t,t) \mid t 
\in \partial_\infty \G \}.$ 

\begin{thm}
  \cite[Theorem~8.3.C]{Gromov_hyp}, \cite[Theorem~60]{Mineyev_flow}

  Let $\G$ be a finitely generated word hyperbolic group. Then there exists a proper hyperbolic
  metric space $\flow$ such that:
  \begin{enumerate}
  \item $\G \times \RR \rtimes \ZZ/2\ZZ$ acts on $\flow$.
  \item The $\G \times \ZZ/2\ZZ$-action is isometric.
  \item Every orbit $\G \to \flow$ is a quasi-isometry. In particular, $\partial_\infty \flow \cong \partial_\infty \G$.
  \item The $\RR$-action is free,  and every orbit $\RR \to \flow$ is a
    quasi-isometric embedding. The induced 
map $\flow/ \RR
    \to \partial_\infty \flow^{(2)}$ is a homeomorphism.
  \end{enumerate}
\end{thm}
In fact $\flow$ is unique up to a $\G \times \ZZ/2\ZZ$-equivariant quasi-isometry sending $\RR$-orbits to $\RR$-orbits. We shall denote by $\phi_t$ the $\RR$-action on $\flow$ and by $(\tau^+, \tau^-): \flow \to \flow/ \RR \cong \partial_\infty \G^{(2)}$ the maps associating to a point the endpoints of its $\RR$-orbit.

\begin{defi}\label{defi:ARhyp}
  A representation $\rho: \G \to G$ is said to be  \emph{$(P^+, P^-)$-Anosov} if there exist continuous $\rho$-equivariant maps $\xi^+: \partial_\infty \G \to \mathcal{F}^+$, $\xi^-: \partial_\infty \G \to 
  \mathcal{F}^-$ such that:
  \begin{enumerate}
 \item For all $(t^+ , t^-) \in \partial_\infty \G^{(2)} $ the pair $(\xi^+( t^+), \xi^-(t^-))$ is transverse.
  \item\label{contraction} For one (and hence any) continuous and equivariant family of norms
    $(\|\cdot\|_{\hat{m}})_{\hat{ m} \in \flow}$ on 
 $(T_{\xi^+(\tau^+( \hat{m}))}
\mathcal{F}^+)_{\hat{ m} \in \flow}$ (resp.\ 
 $(T_{\xi^-(\tau^-( \hat{m}))}
\mathcal{F}^-)_{\hat{ m} \in \flow}$), there exist $A,a>0$ such that for all
$t\geq 0$, $\hat{m} \in \flow$ and $e \in
T_{\xi^+(\tau^+( \hat{m}))} \mathcal{F}^+$ (resp.\ $e \in
T_{\xi^-(\tau^-( \hat{m}))} \mathcal{F}^-$):
\[ \| e\|_{\phi_{-t} \hat{m}} \leq A e^{-a t} \|e\|_{\hat{m}} \quad
(\text{resp.} \ \| e\|_{\phi_{t} \hat{m}} \leq A e^{-a t} \|e\|_{\hat{m}}). \]
  \end{enumerate}

  The maps $\xi^\pm$ are said to be  the \emph{Anosov maps} associated to $\rho$.
\end{defi}

\begin{remark}
  As explained in Section~\ref{sec:RiemEqu}, the definition here is equivalent to the existence of a section
  $\sigma$ of the $\mathcal{X}$-bundle $\mathcal{X}_\rho =  \flow \times_\rho
  \mathcal{X}$ over $\G \backslash \flow$ that is flat along $\RR$-orbits and such that the action of $\phi_t$ on the vector bundle
  $\sigma^* E^+$ (resp.\ $\sigma^* E^-$) is
  dilating (resp.\ contracting).
\end{remark}

\section{Controlling the Anosov section}
In this section we first recall some well known properties of Anosov
representations. Then we introduce $L$-Cartan projections which are
$\Gamma$-invariant continuous maps from  $\flow \times
\RR$ with values in a closed Weyl chamber  of  $ L = P^+ \cap P^-$. These $L$-Cartan projections 
provide a simple criterion 
for a section to satisfy the contraction property (see Definition~\ref{defi:ARhyp}.\eqref{contraction}). 

\subsection{Holonomy and uniqueness}
\label{sec:holon-uniq}
Any non-torsion element $\g\in\G$ has two fixed points in $\partial_\infty \G$. We denote the attracting fixed point by $t^{+}_{\g}$ 
and the repelling fixed point by $t^{-}_{\g}$.
From the definition of Anosov representation,
one deduces easily.

\begin{lem}\label{lem:ano_proxi}
  Let $\rho: \G \to G$ be a $(P^+, P^-)$-Anosov representation and 
  let $\xi^\pm: \partial_\infty \G \to \Ff^\pm$ 
  be the associated Anosov maps. Let $\g \in \G$ be a non-torsion
  element.

  Then $\xi^+( t^{+}_{\g})$ is the unique attracting fixed point of
  $\rho( \g)\in \mathcal{F}^+$. The basin of attraction is the set of
  all points in $\mathcal{F}^+$ that are transverse to
  $\xi^-(t^{-}_{\g})$.
   In particular the eigenvalues of $\rho( \g)$
  acting on $T_{\xi^+(t^{+}_{\g})} \mathcal{F}^+$ are all of modulus
  less than $1$.
\end{lem}
An analogous statement holds for the action on $\mathcal{F}^-$.

\begin{cor}\label{cor:holonomy}
Let $\rho: \G \to G$ be a $(P^+, P^-)$-Anosov representation and $\g \in \Gamma$ a non-torsion element. 
Then $\rho(\g)$ is conjugate to an element of $L$, whose action is dilating on $T_{P^-} \mathcal{F}^-$ and 
contracting on
$T_{P^+} \mathcal{F}^+$.
\end{cor}
A refinement of this corollary  will be given in Lemma~\ref{lem:mu_contract}, providing a quantitative statement for the contraction. 

Lemma~\ref{lem:ano_proxi}, together with the density of the fixed
points $\{t^{+}_{\g}\}_{\g \in \G}$ in $\partial_\infty \Gamma$, has
the following consequence:

\begin{lem}\label{lem:uniqu}\cite[Proposition~2.5]{Guichard_Wienhard_InvaMaxi}
   Let $\rho: \G \to G$ be a $(P^+, P^-)$-Anosov representation. Then
   the maps $\xi^+: \partial_\infty \G \to \mathcal{F}^+$ and $\xi^-: \partial_\infty \G \to 
   \mathcal{F}^-$ satisfying the properties of 
   Definition~\ref{defi:ARhyp} are unique.
\end{lem}

The uniqueness of the Anosov maps $(\xi^+, \xi^-)$ gives the following corollaries:

\begin{cor}\cite[Proposition~2.8]{Guichard_Wienhard_InvaMaxi}\label{cor:finite_index}
  Let $\G' < \G$ be a finite index subgroup. A representation $\rho:
  \G \to G$ is $(P^+, P^-)$-Anosov if and only if $\rho|_{\G'}$ is
  $(P^+, P^-)$-Anosov. Furthermore with the canonical identification 
 $\partial_\infty \G \cong \partial_\infty \G'$ the Anosov maps
 $\xi^+$ and $\xi^-$ are the same for $\rho$ and for $\rho|_{\G'}$.
\end{cor}

\begin{cor}\label{cor:xi_cent}
  Let $\rho: \G \to G$ be a $(P^+, P^-)$-Anosov representation and let 
  $\xi^+: \partial_\infty \G \to \mathcal{F}^+$ and $\xi^-: \partial_\infty \G \to 
   \mathcal{F}^-$ be the corresponding Anosov maps.

Then any element $z\in Z_G(\rho(\G))$ in the centralizer of $\rho(\G)$ fixes $\xi^\pm (\partial_\infty \G) $ pointwise, i.e.\
for any $t\in \partial_\infty \Gamma$, $z \cdot
  \xi^\pm(t) = \xi^\pm(t)$.
\end{cor}

\begin{cor}
  Let $\pi: \widehat{G} \to G$ be a covering of Lie groups,
  $(P^+, P^-)$ a pair of opposite parabolic subgroups of $G$ so that
  $(\widehat{P}^+, \widehat{P}^-)= (\pi^{-1}(P^+), \pi^{-1}(P^-))$
  is a pair of opposite parabolic subgroups of $\widehat{G}$.

  Then a representation $\rho: \G \to \widehat{G}$ is
  $(\widehat{P}^+, \widehat{P}^-)$-Anosov if and only if $\pi \circ
  \rho$ is $(P^+, P^-)$-Anosov.
\end{cor}

\subsubsection{Lifting}
Even though we do not use it in this article, for future reference we describe 
when the maps $\xi^\pm: \partial_\infty \Gamma \to G/P^\pm$ 
can be lifted to maps $\xi^{\prime\pm}:\partial_\infty \G \to G/ P^{\prime\pm}$ where
$P^{\prime\pm} \subset P^\pm$ are finite index subgroups.

Recall that $P^+$ (and $P^-$) is the semi-direct product of its
unipotent radical by $L$. Hence $\pi_0( P^+) \cong \pi_0(L) \cong
\pi_0 (P^-)$, and a finite index subgroup $P^{\prime+} \subset
P^+$ corresponds to a finite index subgroup $\pi_0(P^{\prime+}) \subset
\pi_0( P^+) $, and hence to a finite index subgroup
$L^\prime \subset L$ as well as to a 
finite index subgroup $P^{\prime-} \subset P^-$. Using that the
unipotent radical of $P^+$ is contractible (and the classical fact that the
space of sections of a bundle with contractible fibers is nonempty and
contractible) one deduces:

\begin{lem}
  Let  $\rho: \G \to G$ be a $(P^+, P^-)$-Anosov representation with Anosov maps 
  $\xi^+, \xi^- : \partial_\infty \G \to \mathcal{F}^+, \mathcal{F}^-$. Let $\sigma$ be the corresponding section of
  $\mathcal{X}_\rho =  \flow \times_\rho G/L$. Then the following are equivalent:
  \begin{enumerate}
  \item \label{item:lemXilifts} $\xi^+$ lifts to a continuous $\rho$-equivariant map 
    $\partial_\infty \G \to G/ P^{\prime+}$.
  \item \label{item:lemSigmalifts} $\sigma$ lifts to a continuous section $\sigma'$ of $\mathcal{X}^{\prime}_{\rho} =  \flow \times_\rho G/L^\prime$.
  \item $\xi^-$ lifts to a continuous $\rho$-equivariant map 
    $\partial_\infty \G \to G/ P^{\prime-}$.
  \end{enumerate}
  In that case 
  the $\rho$-equivariant lift $\hat{\sigma}': \flow \to G/L'$ of the section 
  $\sigma'$  is of
  the form $(\xi^{\prime+} \circ \tau^+, \xi^{\prime-} \circ \tau^-)$, with maps  
  $\xi^{\prime+}, \xi^{\prime-}: \partial_\infty \G \to G/
  P^{\prime+}, G/ P^{\prime-}$.
\end{lem}
 Note that the maps  $\xi^{\prime+}, \xi^{\prime-}$ in the above lemma
 are not necessarily unique. 
 
\begin{proof}
  Note that $\xi^+$ naturally defines a section $\zeta^+$ of
  $\mathcal{F}^{+}_{\rho} = \flow \times_\rho \mathcal{F}^+$. This
  section is the image of $\sigma$ under the map $\mathcal{X}_\rho \to
  \mathcal{F}^{+}_{\rho}$ induced by the projection to the first
  factor. The spaces $\mathcal{X}^{\prime}_{\rho}$ and $
  \mathcal{F}^{\prime+}_{\rho}$ are defined similarly, and it is easy
  to see that $\xi^+$ lifts if and only if $\zeta^+$
  lifts. Furthermore $\mathcal{X}^{\prime}_{\rho}$ is the fiber
  product $\mathcal{X}_{\rho} \times_{\mathcal{F}^{+}_{\rho}}
  \mathcal{F}^{\prime+}_{\rho}$ so that $\sigma$ lifts if and only if
  $\zeta^+$ does. This proves the equivalence of
  \eqref{item:lemXilifts} and \eqref{item:lemSigmalifts}; the
  remaining statements follow.
\end{proof}

\subsection{Structure of parabolic subgroups}
\label{sec:parab-subgr-lie}
In the following sections we will use the finer structure of parabolic subgroups. In order to fix notation we recall here the classical structure theory of parabolic subgroups. 

\begin{asparaitem} 
\item Let $G$ be a semisimple Lie groups and let $\mathfrak{g}$ be its Lie algebra. Let $K$ be a maximal compact subgroup of $G$ and $\mathfrak{k}$ its Lie algebra;  the decomposition  $\mathfrak{g} = \mathfrak{k} \oplus
  \mathfrak{k}^\perp$ is orthogonal with respect to the
  Killing form.
\item Let $\mathfrak{a}$ be a maximal abelian subalgebra contained
    in
  $\mathfrak{k}^\perp$; its action on $\mathfrak{g}$  gives rise to a
  decomposition into eigenspaces:
  \[ \mathfrak{g} =  \!\!\!\! \bigoplus_{\alpha \in \Sigma \cup \{0\}} \!\!\!\!
  \mathfrak{g}_\alpha, \text{ where }
  \mathfrak{g}_\alpha = \bigcap_{ a \in
  \mathfrak{a}} \ker(\mathrm{ad}(a) - \alpha(a)), \]
  where $
  \Sigma
  = \{ \alpha \in \mathfrak{a}^* \moins \{0 \} \mid
  \mathfrak{g}_\alpha \neq 0 \}$
   is the system of restricted roots of $\mathfrak{g}$.
\item Let $N_K(\mathfrak{a})$ and $Z_K( \mathfrak{a})$ be the
  normalizer and the centralizer of $\mathfrak{a}$ in $K$. The 
  \emph{Weyl  group} $W = N_K(\mathfrak{a}) / Z_K( \mathfrak{a})$ acts on
  $\mathfrak{a}$ and also  on $\Sigma$. 
\item Let $<_{\mathfrak{a}^*}$ (or simply $<$) 
  be a total ordering on the 
  group $\mathfrak{a}^*$. The sets $\Sigma^+ = \{ \alpha \in \Sigma \mid 
  \alpha > 0 \}$ and $\Sigma^- = \{ \alpha \in \Sigma \mid \alpha < 0 \}$
  are the \emph{positive} roots and the \emph{negative} roots, $\Sigma^- = -
  \Sigma^+$.
\item A positive root is \emph{decomposable} if it is
  the sum of two positive roots; it is called \emph{simple}
  otherwise. The set $\Delta \subset \Sigma^+$ is the set of simple roots.
\item The unique element $w_0\in W$ sending
  $\Sigma^-$ to $\Sigma^+$ induces an involution $\iota: \Sigma^+ \to
  \Sigma^+ \sep \alpha \mapsto -w_0 (\alpha)$, called the \emph{opposition
  involution}; $\iota( \Delta)= \Delta$. The opposition
  involution is also defined on $\mathfrak{a}$.
\item A \emph{Weyl chamber} is $\mathfrak{a}^+ = \{ a \in \mathfrak{a} \mid \alpha(a)
  > 0, \forall \alpha \in \Sigma^+\} = \{ a \in \mathfrak{a} \mid
  \alpha(a) > 0, \forall \alpha \in \Delta\}$. The involution $\iota$ sends $\mathfrak{a}^+$
  into itself. 
  \item Any element $g\in G$ can be written as a product
  \[ g = k \exp( \mu(g)) l, \quad \text{with } k,l \in K \text{ and } \mu(g) \in
  \overline{\mathfrak{a}}^+;\] 
  $\mu(g)$ is uniquely determined by $g$, it is called the
  \emph{Cartan projection} of $g$. The map $\mu: G \to
  \overline{\mathfrak{a}}^+$ is continuous.
  Another continuous 
  projection $\lambda: G \to  \overline{\mathfrak{a}}^+$ comes from
  the Jordan decomposition. The two projections are related by $\lambda(g) = \lim_{n\to \infty }
  \frac{\mu(g^n)}{n}$ \cite[Corollaire in Paragraph 2.5.]{Benoist}. 
\item The subalgebra $\mathfrak{n}^+ = \bigoplus_{\alpha \in \Sigma^+}
  \mathfrak{g}_\alpha$ is 
  nilpotent, $N = \exp( \mathfrak{n}^+)$ is unipotent.

The subgroup $A \ltimes N = \exp(
\mathfrak{a}) \ltimes \exp( \mathfrak{n}^+) < G$ will be denoted by
$AN$. 
\item The set $B^+ = Z_K( \mathfrak{a}) A N$ is a subgroup
  of $G$ called the \emph{minimal parabolic subgroup}. Its Lie algebra is
  $\mathfrak{b}^+ = \mathfrak{g}_0 \oplus \mathfrak{n}^+$. 
\item Similarly one defines $\mathfrak{n}^-$, $N^-$, $B^-$, $\mathfrak{b}^-$. The group $B^-$ is conjugate to $B^+$.
\end{asparaitem}

\emph{Parabolic subgroups} of $G$ are conjugate to subgroups containing
$B^+$. A pair of parabolic subgroups is said to be 
\emph{opposite} if their intersection is a reductive
group.

Conjugacy classes of parabolic subgroups are in one to one correspondence with subsets $\Theta \subset \Delta$. Given such a subset we set 
$\mathfrak{a}_\Theta = \bigcap_{\alpha \in \Theta} \ker \alpha$ and denote by $M_\Theta = Z_K( \mathfrak{ a}_\Theta)$ its centralizer in $K$.
  Then $P^{+}_{\Theta} = M_\Theta \exp( \mathfrak{a}) N$ and 
  $P^{-}_{\Theta} = M_\Theta \exp( \mathfrak{a}) N^-$ are parabolic
  subgroups, which are opposite. 
   A parabolic
  subgroup containing $B^+$ is of the form $P^{+}_{\Theta}$ for a
  uniquely determined $\Theta$. 
Any pair of opposite parabolic subgroups is conjugate to
  $(P^{+}_{\Theta}, P^{-}_{\Theta})$ for some $\Theta\subset \Delta$.  
  
  \begin{remark} 
 The conjugacy class of $(P^+, P^-)$ is 
  determined by the conjugacy class of $P^+$. In view of this, we will sometimes
  say that a representation is $P^+$-Anosov (or
  $L$-Anosov). 
  
  Note also that $P^{-}_{\Theta}$ is conjugate to $P^{+}_{\iota(\Theta)}$. 
In particular $P^{+}_{\Theta}$ is conjugate
to its opposite if and only if $\Theta = \iota( \Theta)$.
\end{remark}

The intersection $L_\Theta = P^{+}_{\Theta} \cap P^{-}_{\Theta}$ is
the common Levi component of $P^{+}_{\Theta} $ and $P^{-}_{\Theta}$. The group $M_\Theta$ is a maximal compact
subgroup of $L_\Theta$.  We denote the Weyl chamber of $L_\Theta$ by
$\mathfrak{a}_{L_\Theta}^{+} = \{a \in \mathfrak{a} \mid \alpha(a)
>0,\ \text{for all } \alpha \in \Theta \}$. The Cartan decomposition
for $L_\Theta$ is $L_\Theta = M_\Theta
\exp(\overline{\mathfrak{a}}^{+}_{L_\Theta} ) M_\Theta$.

 We denote by
$\Sigma_\Theta$ the roots in the span of $\Theta$: $\Sigma_\Theta =
\Span_\RR( \Theta) \cap \Sigma$. The Lie algebras of $P^{+}_{\Theta}$,
$P^{-}_{\Theta}$ and $L_\Theta$ are:
\[ \mathfrak{p}^{+}_{\Theta} = \hspace{-1.5em} \bigoplus_{\alpha \in \Sigma^+ \cup
  \Sigma_\Theta \cup \{ 0 \}} \hspace{-1.5em} \mathfrak{g}_\alpha, \
\mathfrak{p}^{-}_{\Theta} = \hspace{-1.5em} {\bigoplus_{\alpha \in \Sigma^- \cup
  \Sigma_\Theta \cup \{ 0 \}} \hspace{-1.5em} \mathfrak{g}_\alpha}, \
\mathfrak{l}_\Theta = \hspace{-1em} {\bigoplus_{\alpha \in \Sigma_\Theta \cup \{ 0
  \}} \hspace{-1em} \mathfrak{g}_\alpha}.\]
 The nilpotent radicals of
$\mathfrak{p}^{+}_{\Theta}$ and $\mathfrak{p}^{-}_{\Theta}$ are hence:
\[ \mathfrak{n}^{+}_{\Theta} = \underset{\alpha \in \Sigma^+ \moins
  \Sigma_\Theta }{\bigoplus\ \mathfrak{g}_\alpha}, \quad
\mathfrak{n}^{-}_{\Theta} = \underset{\alpha \in \Sigma^- \moins
  \Sigma_\Theta }{\bigoplus\ \mathfrak{g}_\alpha}.\]

 There are $L_\Theta$-equivariant identifications of the tangent space
of $G/ P^{+}_{\Theta}$  at
$P^{+}_{\Theta}$  with $\mathfrak{
  n}^{-}_{\Theta}$, and respectively of the tangent space
of $G/ P^{-}_{\Theta}$  at
$P^{-}_{\Theta}$  with $\mathfrak{
  n}^{+}_{\Theta}$.

An element $\exp(a)$ with $a\in \overline{
  \mathfrak{a}}^{+}_{L_\Theta}$ contracts on $\mathfrak{ n}^{-}_{\Theta}$
(resp.\ dilates 
  on $\mathfrak{n}^{+}_{\Theta}$) if and only if
$\alpha(a)>0$ for all $\alpha \in \Delta \moins \Theta$. Moreover one
has the quantitative statement:

\begin{lem}\label{lem:mu_contract}
  There is a positive constant $C$ such that:
  \begin{itemize}
  \item For any $k$ and $a \in \overline{ \mathfrak{a}}^{+}_{L_\Theta}$, if
    $\exp(a)$ is $k$-Lipschitz on $T_{P^{+}_{\Theta}} G/
    P^{+}_{\Theta}$ then, for all $\alpha \in \Delta \moins \Theta$,
    one has $\alpha(a) \geq - \log k$.
  \item For any $M \geq 0$ and $a\in \overline{ \mathfrak{a}}^{+}_{L_\Theta}$, if
    for all $\alpha \in \Delta \moins \Theta$, $\alpha(a) \geq M$ then
    $\exp(a)$ is $Ce^{-M}$-Lipschitz on $T_{P^{+}_{\Theta}}
    G/ P^{+}_{\Theta}$. 
  \end{itemize}
\end{lem}

In particular this lemma implies that an element of $\overline{
  \mathfrak{a}}^{+}_{L_\Theta}$ whose action is contracting on $T_{P^{+}_{\Theta}} G/
    P^{+}_{\Theta}$ is contained in the closed Weyl chamber $\overline{
      \mathfrak{a}}^{+}$ of $G$.

\subsection{Lifting sections and $L$-Cartan projections}
\label{sec:lifting-section-gm}

Let $(P^+,P^-) = (P^{+}_{\Theta}, P^{-}_{\Theta})$ be a pair of
opposite parabolic subgroups, and let $L_\Theta$, $M_\Theta$, etc.\ be
as in the preceding section.
We occasionally drop the subscript $\Theta$. 
We set $\mathcal{Y} = G/M$, this is an $L/M$-bundle over $\Xx = G/L$.

Let $\rho: \G \to G$ be a representation and let 
$\mathcal{X}_\rho = \flow \times_\rho \mathcal{X}$ and 
$\mathcal{Y}_\rho = \flow \times_\rho \mathcal{Y}$ be the associated flat bundles over $\G \backslash \flow$. Then $\pi :
\mathcal{Y}_\rho \to \mathcal{X}_\rho$ is an $L/M$-bundle and hence has
contractible (even convex) fibers. 
This implies 
\begin{lem}\label{lem:exi_lift}
  Let $\sigma$ be a section of $\mathcal{X}_\rho$, then there exists a
  section $\beta$ of $\mathcal{Y}_\rho$ such that $\pi \circ \beta = \sigma$.

\end{lem}

\begin{proof}
  Indeed it is equivalent to finding a section of the $L/M$-bundle
  $\sigma^* \mathcal{Y}_\rho$ over $\G \backslash \flow$ (the pull
  back by
  $\sigma$ of the $L/M$-bundle $\mathcal{Y}_\rho \to \mathcal{X}_\rho$). This is a
  locally trivial bundle, thus local sections exist. Moreover $\G
  \backslash \flow$ is a compact metric space and admits partitions
  of unity, so that local sections can be glued together to a global section.
\end{proof}

Two different lifts are equal up to finite distance:

\begin{lem}\label{lem:bdd_distance}
Let $\beta$ and $\beta'$ be two lifts of $\sigma$, and denote by $\hat{\beta}, \hat{\beta}': \flow \to \mathcal{Y}$ their $\rho$-equivariant pull-backs. Then there exists $R>0$ such that $\forall m \in \flow$, $d_{G/M}(\hat{\beta}(m), \hat{\beta}'(m) ) \leq R$, where $d_{G/M}$ 
is a $G$-invariant Riemannian distance  $\mathcal{Y} = G/M$. 
\end{lem}

Suppose now that $\sigma$ is flat along $\RR$-orbits. The section
$\beta$ corresponds to an equivariant continuous map $\hat{\beta}:
\flow \to \mathcal{Y}$ lifting $\hat{ \sigma}: \flow \to \mathcal{X}$,
the equivariant map corresponding to $\sigma$. As $\hat{ \sigma}$ is
$\RR$-invariant, for any $\hat{m}\in \flow$ and any $t \in \RR$,
$\hat{\beta}( \hat{m})$ and $\hat{\beta}( \phi_t \hat{m})$ project to
the same point in $\mathcal{X} = G/L$. Thus the pair $(\hat{\beta}( \hat{m}),
\hat{\beta}( \phi_t \hat{m}))$ is in the $G$-orbit of a unique pair of the
form $(\exp(\mu_{+,\Theta}(\hat{m},t)) M, M)$ with $\mu_{+,\Theta}(\hat{m},t)\in \overline{ \mathfrak{a}}^{+}_{L_\Theta}$; similarly it is in the $G$-orbit of a unique pair of the
form $( M, \exp(\mu_{-,\Theta}(\hat{m},t)) M)$ with $\mu_{-,\Theta}(\hat{m},t)\in \overline{ \mathfrak{a}}^+_{L_\Theta}$. 

\begin{defi}
Let $\sigma: \G\backslash \flow \to \mathcal{X}_\rho$ be a section which is flat along $\RR$-orbits, and let $\beta$ be a lift of $\sigma$ to a section of $\mathcal{Y}_\rho$.  The  maps 
\[\mu_{+,\Theta}, \mu_{-,\Theta}: \flow \times \RR
\to \overline{ \mathfrak{a}}^{+}_{L_\Theta}\] 
are called \emph{$L_\Theta$-Cartan projections}. 
They are continuous and
$\G$-invariant, and hence well defined on $\G \backslash \flow \times
\RR$.
\end{defi}

\begin{remarks}
\noindent
\begin{asparaenum}[(a)]

\item  The $L_\Theta$-Cartan projections take values in the closed Weyl chamber
  $\overline{ \mathfrak{a}}^{+}_{L_\Theta} = \{ a \in \mathfrak{a} \mid \alpha(a)
\geq 0,\ \text{for all } \alpha \in \Theta  \}$ which is a closed cone
of $\mathfrak{a}$ with nonempty interior.
\item 
It is possible to define maps into the closed Weyl chamber $\overline{
  \mathfrak{a}}^+$ of $\mathfrak{g}$, however doing this will lead to
a loss of information: we need to understand $G$-orbits of pairs of
points in $G/M$ that project to the same element in $G/L$; this is the
same as understanding $L$-orbits of pairs of points in $L/M$ which are
ultimately completely classified by the closed Weyl chamber of
$L$. The maps into $\overline{
  \mathfrak{a}}^+$ would amount to classifying $G$-orbits of pairs in
$G/K$ and cannot keep track of the action on
$\mathfrak{n}^{\pm}_{\Theta}$. 

\item The idea of lifting the section $\sigma$ of $\mathcal{X}_\rho$
  to a section of $\mathcal{Y}_\rho$ is already implicit in
  \cite{Labourie_anosov} and \cite{Burger_Iozzi_Labourie_Wienhard},
  where a specific metric on the bundle $\sigma^* E^\pm$ (see
  Section~\ref{sec:defiRiem}) is chosen in order to prove the
  contraction property. The choice of this specific metric corresponds
  to a choice of a lift $\beta$.

\item 
The classical Cartan projection $\mu: G \to \mathfrak{a}^+$ can be
used to define a refined Weyl chamber valued distance
function: if $\delta$ denotes the Weyl chamber valued distance on the
symmetric space $G/K$, then $\mu(g) = \delta(K, gK)$, see for example 
\cite{KapovichLeebMillson_triangle,Parreau_distance} for an account on this
subject and references therein. 
\end{asparaenum}
\end{remarks}

The dependence of the $L$-Cartan projection on the section $\sigma$ is crucial. 
However, their asymptotic behavior does not depend on the choice of the lift $\beta$:

\begin{lem}\label{lem:mu_beta}
Let $\sigma: \G\backslash \flow \to \mathcal{X}_\rho$ be a section
which is flat along $\RR$-orbits, and let $\beta$ and $\beta'$ be two
lifts. Let $\mu_{\pm , \Theta}$ and $\mu_{\pm , \Theta}^{\prime}$ be
the $L$-Cartan projections associated with $\beta$ and $\beta'$
respectively.

Then there exists a constant $C>0$ such that for all $(m,t) \in \flow \times \RR$ 
\[ d(\mu_{\pm , \Theta} (m,t),  \mu_{\pm , \Theta}^{\prime} (m,t))
\leq C, \]
where $d$ is the distance for some norm on the vector space $ \mathfrak{a}_{L_\Theta}$.
\end{lem}

\begin{proof}
Denote by $\hat{\sigma}$, $\hat{\beta}$, and $ \hat{\beta}'$ the
$\rho$-equivariant lifts of $\sigma$, $\beta$ and $\beta'$. Let $(m,t) \in \flow \times \RR$,
since $\pi(\hat{\beta} (m))  = \pi(\hat{\beta} (\phi_t m))  =
\hat{\sigma} (m) =  \pi(\hat{\beta}' (\phi_t m) ) = \pi(\hat{\beta}'
(m))$,  the four points  $ \hat{\beta} (m)$,  $\hat{\beta}(\phi_t m)$,
$\hat{\beta}'(\phi_t m)$ and $\hat{\beta}'(m)$ lie in one fiber, which
we can assume to be $L/M$.  

Moreover by Lemma~\ref{lem:bdd_distance}  there exists $R>0$ such that $\forall m \in \flow$:  $d( \hat{\beta} (m), \hat{\beta'} (m)) \leq R$ and  $d( \hat{\beta} (\phi_t m), \hat{\beta'} (\phi_t m)) \leq R$. 

The statement now follows from the triangle inequality for the Weyl
chamber valued distance function on the symmetric space $L/M$
\cite{KapovichLeebMillson_triangle, Parreau_distance}. 
\end{proof}

Recall that $\lambda(g)$ denotes the hyperbolic part of the Jordan
decomposition of $g$ (Section~\ref{sec:parab-subgr-lie}):
\begin{lem}
Let  $\g \in \G$ be a  non-torsion element and  denote by $\hat{\g}
\subset \flow$ the corresponding $\g$-invariant $\RR$-orbit in $\flow$. 
Let $T$ be the period of this orbit (i.e.\ $\g$ acts as $\phi_T$ on $\hat{\g}$) and $\hat{m} \in \hat{\g}$. 
Assume that $\rho: \G \to G$ is an Anosov representation. 
Then 
\[\mu_{+, \Theta}(\hat{m},T) = \lambda (\rho(\g)),
\text{ and } \mu_{-, \Theta} ( \hat{m},T) = \lambda(\rho(\g)^{-1}).\]
\end{lem}

For $(m, t)\in \G \backslash \flow \times \RR$ we set 
\[ A_+(m,t) = \min_{\alpha \in \Delta \moins \Theta}
\alpha(\mu_{+,\Theta}(\hat{m},t)), \quad A_-(m,t) = \min_{\alpha \in \Delta \moins \Theta}
\alpha(\mu_{-,\Theta}(\hat{m},t)).\]

Note that  $\mu_{-,\Theta}(\hat{m},t) = \iota_\Theta(
\mu_{+,\Theta}(\hat{m},t))$ where $\iota_\Theta$ is the opposition 
involution for $L_\Theta$, and also that $\mu_{+,\Theta}(\hat{m},t) = 
\mu_{-,\Theta}(\phi_t \hat{m},-t)$. This means that dilatation on
$\sigma^* E^+$, which is governed by $\mu_{+,\Theta}$, is equivalent
to contraction on $\sigma^* E^-$, which is governed by
$\mu_{-,\Theta}$. Hence, in the next proposition, only one of the
$L$-Cartan projections needs to be considered:

\begin{prop}\label{prop:mu_contrac}
  Let $\rho: \G \to G$ be a representation, let $\sigma$ be a section
  of $\mathcal{X}_\rho$ which is flat along $\RR$-orbits.  Let $A_+,
  A_-$ be as above. The following are equivalent:
  \begin{enumerate}
  \item \label{item1:propMU} $\sigma$ is an Anosov section (and hence $\rho$ is $(P^+, P^-)$-Anosov).
  \item \label{item2:propMU} There exist positive constants $C$ and $c$ such that for all
    $t \geq 0$ and all $m\in \G \backslash \flow$, one has
    $A_+(m,t) \geq c t -C$ and $A_-(m,-t) \geq c t -C$.
  \item \label{item3:propMU} There exist positive constants $C$ and $c$ such that for all
    $t \geq 0$ and all $m\in \G \backslash \flow$, one has
    $A_+(m,t) \geq c t -C$. 
  \item \label{item4:propMU} $\lim_{t\to +\infty} \inf_{m \in \G \backslash \flow}
    A_+(m,t)= +\infty$.
  \end{enumerate}
\end{prop}

\begin{remark}\label{rem:in_a_plus}
Note that this implies that $\mu_{+,\Theta}(m,t)$ belongs to
$\overline{ \mathfrak{a}}^+$ for $t$ big enough.
\end{remark}
\begin{proof}
  Indeed \eqref{item1:propMU} $\Rightarrow$ \eqref{item2:propMU} follows
  from Lemma~\ref{lem:mu_contract}; the implications
  \eqref{item2:propMU} $\Rightarrow$ \eqref{item3:propMU} $\Rightarrow$
  \eqref{item4:propMU} are immediate. By Lemma~\ref{lem:mu_contract},
  condition \eqref{item4:propMU} implies weak dilatation of the flow on
  the bundle $\sigma^* E^+$, and thus also weak contraction for the flow
  on the bundle $\sigma^* E^-$, from the relations between $\mu_{+,
    \Theta}$ and $\mu_{-, \Theta}$ mentioned above. Uniform estimates follow from
  the compactness of $\G \backslash \flow$. For details on this last
  argument see \cite[Section~6.1]{Labourie_anosov}.
\end{proof}

\subsection{Consequences}
\label{sec:conseq}

Proposition~\ref{prop:mu_contrac} reduces the property of being Anosov
to the control of a few eigenvalues (or rather principal values). From
this one immediately deduces the following

\begin{lem}\label{lem:ThetaAn}
  Let $\rho: \G \to G$ be a representation.
  \begin{enumerate}
  \item If $\rho$ is $P_\Theta$-Anosov, then it is $P_{\iota( \Theta)}$-Anosov.
  \item If $\rho$ is $P_\Theta$-Anosov, then it is $P_{\Theta'}$-Anosov for
    any $\Theta' \supset \Theta$.
  \item If $\rho$ is $P_{\Theta_1}$-Anosov and $P_{\Theta_2}$-Anosov, then it
    is $P_{\Theta_1 \cap \Theta_2}$-Anosov.
  \item \label{itemF:lemThetAn} If $\rho$ is Anosov, then it is $P_\Theta$-Anosov for some
    $\Theta$ satisfying $\iota( \Theta) = \Theta$.
  \end{enumerate}
\end{lem}
The first statement is a consequence of the fact that the
$\ZZ/2\ZZ$-action on $\flow$ anti-commutes with the action of
$\RR$. The other consequences follow immediately from
Proposition~\ref{prop:mu_contrac}.

\section{Lie group homomorphisms and equivariant maps}
In this section we first describe how the property of being Anosov behaves with respect to compositions with Lie group homomorphisms. Then we show that for irreducible or Zariski dense representations the existence of equivariant maps readily implies the contraction property. 

  \subsection{Lie group homomorphisms}\label{sec:homo}
Let $\phi: G \to G'$ be a homomorphism of semisimple Lie groups and let $\rho: \G \to G$ be a $P_\Theta$-Anosov representation. For several
arguments we give later it will be essential to determine the
parabolic subgroup $P' < G'$ such that the composition $\phi\circ
\rho: \G \to G'$ is $P'$-Anosov. Whereas this has a rather simple
answer when $G$ is of rank one (see
Proposition~\ref{prop:inj_rk_one}), it is more complicated when $G$ is
of higher rank (see Section~\ref{sec:ex_inj} for
examples). In this section we give an explicit construction of $P'$.  

Let $\phi: G \to G'$ be a homomorphism between semisimple Lie groups. We can assume that the
maximal compact subgroup $K'< G'$ and the Cartan algebra
$\mathfrak{a}'$ 
are chosen
to be compatible with $\phi$, i.e.\ $\phi( K) \subset K'$,
$\phi_*(\mathfrak{a}) \subset \mathfrak{a}'$ (\cite{Karpelevic53}, \cite[Theorem~6]{Mostow_dec}). 
The set of simple roots of $G'$ relative to $\mathfrak{a}'$ is
denoted by $\Delta'$. We shall usually denote with primes the objects
associated with $G'$.

\begin{prop}\label{prop:embedding}
  Let $\rho: \G \to G$ be a
  representation. Suppose that $\phi \circ \rho$ is $P^{\prime}_{\Theta'}$-Anosov
  then $\rho$ is $P_\Theta$-Anosov where $\Theta = \{ \alpha \in \Delta \mid
    \alpha |_{\phi^{-1}(\mathfrak{a}^{\prime}_{ \Theta'})} \text{ is
      zero}\}$.
\end{prop}

For the other direction we first describe in detail the case when $G' = \SL(V)$
for some irreducible $G$-module $V$. 
\begin{lem}\label{lem:basisGmod}
  Let $V$ be an irreducible $G$-module and denote by $\phi: G \to
  \SL(V)$ the corresponding homomorphism and by $\phi_*: \mathfrak{g}
  \to \mathfrak{sl}(V)$ the Lie algebra homomorphism.

  Suppose that there exists a line $D \subset V$ that is
  $P_\Theta$-invariant for some $\Theta \subset \Delta$.

  Then the following holds: 
   \begin{asparaenum}
  \item there exists a basis $(e_1, \dots, e_n)$ of $V$ with $e_1 \in D$
    and consisting of eigenvectors for $\mathfrak{a}$:
    \[ \phi_*(a) \cdot e_i = \lambda_i(a) e_i, \ \forall a \in
    \mathfrak{a}, \text{ with } \lambda_i \in \mathfrak{a}^*.\]
  \item For all $i>1$, $\lambda_1 - \lambda_i$ is a linear combination
    with integer coefficients of the elements of $\Delta$:
    \[ \lambda_1 - \lambda_i = \sum_{\alpha \in \Delta} n_{\alpha,i}
    \alpha, \ n_{\alpha,i} \in \NN,\]
  \item with the property that $(n_{\alpha, i})_{\alpha \in \Delta \moins
      \Theta}$ are not simultanously zero.
  \end{asparaenum}
\end{lem}

\begin{proof}
  For that proof we use standard tools from representation theory (see
  \cite{Fulton_Harris} for example).
  By irreducibility $V = \mathfrak{n}^- \cdot D$. Also if $v \in V$ is
  an eigenvector for $\mathfrak{a}$ with eigenvalue $\lambda$ and if
  $n \in \mathfrak{n}_{\beta}$, for $\beta \in \Sigma$, then the vector
  $w = n \cdot v$, if nonzero, is an eigenvector with
  eigenvalue $\lambda + \beta$. These two remarks give the two first
  conclusions of the lemma (starting with $e_1 = v \in D$). To prove
  the third conclusion, one only has to note that the hypothesis implies $\mathfrak{n}_{-\alpha} \cdot e_1 = 0$ for all $\alpha
  \in \Theta$.
\end{proof}

\begin{prop}\label{prop:Anosov_Rep}
  Let $\phi: G \to G' =  \SL(V)$ be an irreducible finite dimensional linear
  representation of $G$. Let $V = D \oplus H$ be a decomposition of $V$  into a
  line and a hyperplane, and set  $Q^{+}_{0} = \stab(D)$ and
  $Q^{-}_{0} = \stab(H)$. Suppose that $(P^+, P^-)= ( \stab_G(D), \stab_G(H))$ is a pair of opposite parabolic
  subgroups.
  
  Then a representation $\rho: \G \to G$ is $(P^+, P^-)$-Anosov if and
  only if $\phi \circ \rho: \G \to \SL(V)$ is $(Q^{+}_{0},
  Q^{-}_{0})$-Anosov. Furthermore the Anosov maps satisfy $\phi^\pm
  \circ \xi_\rho^\pm = \xi^{\pm}_{\phi \circ \rho}$, where $\phi^+:
  \Ff^+ \to \PP(V)$ and $\phi^-: \mathcal{F}^- \to \PP(V^*)$ are the
  maps induced by $\phi$.
\end{prop}
  
\begin{proof}
  The if part is a consequence of Proposition~\ref{prop:embedding}.

  For the converse statement, we can assume (up to conjugating in $G$)
  that $P^+ = P^{+}_{\Theta}$. Let $(e_1, \dots, e_n)$ denote the basis
  obtained in Lemma~\ref{lem:basisGmod} and
  $\lambda_i$ the corresponding eigenvalues. The Cartan
  subalgebra $\mathfrak{a}' \subset \mathfrak{sl}(V)$ is chosen to be the
  set of matrices that are diagonal with respect to the basis $(e_i)$ and the Weyl
  chamber $\mathfrak{a}^{\prime +}$ is the set of diagonal matrices
  $\diag(t_1, \dots, t_n )$ with $t_1 > \cdots > t_n$.
  
   With these choices the root $\alpha'\in \Delta'$ such
  that $Q^{+}_{0} = P^{\prime +}_{\Delta' \moins \alpha'}$ is $\diag(t_1, \dots, t_n )
  \mapsto t_1 -t_2$. The Weyl chamber $\mathfrak{a}^{ \prime +}_{
    L^{\prime}_{\Delta' \moins \alpha'}}$ of the Levi component
  $L^{\prime}_{\Delta' \moins \alpha'}$ is the set of diagonal matrices $\diag(t_1,
  \dots, t_n )$ with $t_2 > \cdots > t_n$, the Weyl group $W'_{\Delta' \moins \alpha'}$ of
  $L^{\prime}_{\Delta' \moins \alpha'}$ acts on $\mathfrak{a}^{ \prime}_{
    L^{\prime}_{\Delta' \moins \alpha'}} \subset  \mathfrak{a}^{
    \prime}$  as the group of
  permutations on the last $(n-1)$-th diagonal coefficients.

  The maps $\phi^+:
  \Ff \to \PP(V)$ and $\phi^-: \mathcal{F}^- \to \PP(V^*)$ induced by
  $\phi$, give rise to a map $\phi^\mathcal{X} : \mathcal{X} \to
  \mathcal{X}'$ with 
  \[\mathcal{X}' = \{ (L,P) \in \PP(V) \times
  \PP(V^*) \mid V = L \oplus P \} = \SL(V) / L^{\prime}_{\Delta' \moins  \alpha'}.
  \] 
  Let $\hat{ \sigma} : \flow \to
  \mathcal{X}$ be the lift of the Anosov section for $\rho$. We want
  to prove that $\phi^\mathcal{X} \circ \hat{\sigma}$ is the lift of
  an Anosov section for $\phi \circ \rho$.

  For this let $\hat{\beta}': \flow \to \mathcal{Y}'$ be an
  equivariant lift of $\hat{\sigma}' = \phi^\mathcal{X} \circ
  \hat{\sigma}$ where $\mathcal{Y}' = \SL(V) / M^{\prime}_{\Delta' \moins  \alpha'}$,
  with $M^{\prime}_{\Delta' \moins  \alpha'} = L^{\prime}_{\Delta' \moins  \alpha'} \cap \SO(V)$ the
  maximal compact subgroup of $L^{\prime}_{\Delta' \moins  \alpha'}$. Let
  $\mu^{\prime}_{+} : \G \backslash \flow \times \RR \to
  \mathfrak{a}^{\prime+}_{L^{\prime}_{\Delta' \moins \alpha'}}$ the Cartan projection
  associated to $\hat{ \beta}'$
  (Section~\ref{sec:lifting-section-gm}). By
  Proposition~\ref{prop:mu_contrac} it is enough to prove that $\lim_{t\to +\infty} 
  \alpha'( \mu^{\prime}_{+}(m,t))= +\infty$.

  There is also a natural map $\phi^\mathcal{Y} : \mathcal{Y} \to
  \mathcal{Y}'$ induced by $\phi$. If $\hat{\beta}$ is an equivariant
  lift $\hat{\sigma}$, one can suppose that $\hat{\beta}' =
  \phi^\mathcal{Y} \circ \hat{\beta}$. Let $\mu_+$ the Cartan
  projection associated with $\hat{ \beta}$.  By
  Proposition~\ref{prop:mu_contrac} we have that, for all
  $\alpha \in \Delta \moins \Theta$, $\lim_{t\to +\infty} 
  \alpha( \mu_{+}(m,t))= +\infty$. Furthermore
  $\mu^{\prime}_{+}(m,t)$ is the unique element of the orbit of
  $\phi_*( \mu_+(m,t))$ under the Weyl group $W'_{\Delta' \moins \alpha'}$ of
  $L^{\prime}_{\Delta' \moins \alpha'}$ that belongs to the Weyl chamber $\mathfrak{a}^{ \prime +}_{
    L^{\prime}_{\Delta' \moins \alpha'}}$. From this fact and the above description
  of the Weyl chamber and the Weyl group we get that
  \[ \alpha'( \mu^{\prime}_{+}(m,t)) = \min_{i=2, \dots, n} (\lambda_1
  - \lambda_i) ( \mu_+(m,t)). \]
  The third conclusion of the previous lemma implies then:
  \[ \alpha'( \mu^{\prime}_{+}(m,t)) \geq \min_{\alpha \in \Delta
    \moins \Theta} \alpha( \mu_+(m,t)).\]
  This last inequality proves: $\lim_{t\to +\infty} 
  \alpha'( \mu^{\prime}_{+}(m,t))= +\infty$.
\end{proof}

  The above proof relies on the fact that we are able to estimate $\mu^{\prime}_{+}$ in terms of $\mu_+$. More precisely, we used the fact that $\mu_+(m,t)$ belongs to $\overline{ \mathfrak{a}}^+ \moins \bigcup_{ \alpha \in \Theta} \ker( \alpha)$ and that the image of this last cone by $\phi_*$ is contained in $W^{\prime}_{ \Theta'} \cdot \big ( \overline{ \mathfrak{a}}^{\prime+} \moins \bigcup_{ \alpha' \in \Theta'} \ker( \alpha') \big)$, which is a reformulation of the properties deduced in Lemma~\ref{lem:basisGmod}. This is precisely the condition in the general statement: 
    
\begin{prop}\label{prop:injLieGr}
  Let $\phi: G \to G'$ be a Lie group homomorphism as above. 
Let $\Theta \subset \Delta$ and suppose that there exist $w'$ in $W'$ and $\Theta' \subset \Delta'$ such that
  \[ \phi_* \big ( \overline{ \mathfrak{a}}^+ \moins \bigcup_{ \alpha \in \Theta} \ker( \alpha) \big ) \subset w' \cdot W^{\prime}_{ \Theta'} \cdot \big ( \overline{ \mathfrak{a}}^{\prime+} \moins \bigcup_{ \alpha' \in \Theta'} \ker( \alpha') \big). \]
  
  Then, for any $P^{+}_{\Theta}$-Anosov representation $\rho: \G \to G$, the representation $\phi \circ \rho $ is $P^{\prime +}_{\Theta'}$-Anosov. Furthermore $\phi(P^{\pm}_{\Theta}) \subset w' P^{\prime \pm}_{\Theta'} w^{\prime -1}$ and hence there are maps $\phi^+: \mathcal{F}^{+}_{\Theta} \to \mathcal{F}^{\prime +}_{\Theta'} $ and $\phi^-: \mathcal{F}^{-}_{\Theta} \to \mathcal{F}^{\prime -}_{\Theta'} $. If $\xi^\pm$ are the Ansosov maps associated to $\rho$, the Anosov maps for $\phi \circ \rho$ are $\phi^\pm \circ \xi^\pm$.
  \end{prop}

\begin{remark}
  It can happen that $ \Theta' = \Delta' $, i.e.\ that $P^{\prime +}_{\Theta'} = G'$.
\end{remark}

\begin{proof}
  Up to changing the Weyl chamber of $G'$ one can suppose that $w'=1$.
  
  We first prove that $\phi( P^{\pm}_{\Theta}) \subset P^{\prime \pm}_{\Theta'}$. For this note that $P^{+}_{\Theta}$ contains the stabilizer of any point of the visual compactification of the symmetric space $G/K$ that is the endpoint at infinity of a geodesic ray $(\exp(t a) \cdot K)_{t \in \RR_{\geq 0}}$ with $a\in \overline{ \mathfrak{a}}^+ \moins \bigcup_{ \alpha \in \Theta} \ker( \alpha)$. This geometric characterization and the hypothesis on $\phi_*$ imply that $\phi( P^{+}_{\Theta})$ is contained in $\omega P^{\prime+}_{\Theta'} \omega^{-1}$ for some $\omega\in W^{\prime}_{\Theta'}$. Here $\omega\in W^{\prime}_{\Theta'}$ is  such that $\omega \cdot \big ( \overline{ \mathfrak{a}}^{\prime+} \moins \bigcup_{ \alpha' \in \Theta'} \ker( \alpha') \big)$ contains $\phi_*(a)$ with $a\in \overline{ \mathfrak{a}}$ satisfying $\alpha(a)= 0$ for $\alpha \in \Delta \moins \Theta$ and $\alpha(a)> 0$ for $\alpha \in  \Theta$. In conclusion $\phi( P^{+}_{\Theta}) \subset \omega P^{\prime+}_{\Theta'} \omega^{-1} = P^{\prime+}_{\Theta'}$. Similarly $\phi( P^{-}_{\Theta}) \subset P^{\prime-}_{\Theta'}$.

Now the proof follows the same lines as the proof of the Proposition~\ref{prop:Anosov_Rep}, where the hypothesis on $\phi_*$ replaces the use of Lemma~\ref{lem:basisGmod}.
\end{proof}

Notice that in view of establishing Proposition~\ref{prop:mu_contrac}
one could equally work with a continuous function $\mu_{+,\Theta} : \G
\backslash \flow \times \RR \to C + \overline{ \mathfrak{a}}^{+}_{L_\Theta}$, where 
$C \subset \mathfrak{a}$ is a compact subset, satisfying that, for all
$(\hat{m}, t)$, the pair $(\hat{ \beta}( \hat{m}), \hat{ \beta}( \phi_t \hat{m}))$
is in the $G$-orbit of $(\exp(\mu_{+,\Theta}(\hat{m},t))M,M )$. 
This gives a refined version of Proposition~\ref{prop:injLieGr}:

\begin{prop}
 Let $\Theta \subset \Delta$ and let $\rho: \G \to G$ be a $P_\Theta$-Anosov representation; let
  $\mu_{+,\Theta}: \G \backslash \flow \times \RR \to
  \overline{\mathfrak{a}}^{+}_{L_\Theta}$ be the $L_\Theta$-Cartan projection defined in
  Section~\ref{sec:lifting-section-gm}. Suppose that there exist $w'$ in $W'$ and $\Theta' \subset \Delta'$ and a compact $C' \subset \mathfrak{a}'$ such that
  
\[ \phi( \mathrm{Im}(\mu_{+,\Theta})) \subset C' + w' \cdot W^{\prime}_{ \Theta'} \cdot \big ( \overline{ \mathfrak{a}}^{\prime+} \moins \bigcup_{ \alpha' \in \Theta'} \ker( \alpha') \big). \]
  
  Then $\phi \circ \rho$ is $\Theta'$-Anosov.
\end{prop}

  \subsection{Groups of rank one}
\label{sec:grprkone}
When $G$ is of rank one, for any homomorphism $\phi: G \to G'$ one can
arrange that the closed Weyl chamber $\overline{ \mathfrak{a}}^{
  \prime +}$ contains $\phi_*( \mathfrak{a}^+)$. Thus Proposition~\ref{prop:injLieGr} implies the following (see also \cite[Proposition~3.1]{Labourie_anosov})  
\begin{prop}\label{prop:inj_rk_one}
Let $G$ be a Lie group of real rank one.  Let $\rho: \G \to G$ be an Anosov representation and $\phi: G \to
  G'$ a homomorphism of Lie groups. Assume that the Weyl
  chambers of $G$ and $G'$ are arranged so that  $\phi(
  \mathfrak{a}^+) \subset \overline{ \mathfrak{a}}^{
  \prime +}$.

  Then $\phi \circ \rho$ is $P_{\Theta'}$-Anosov where $\Theta' = \{
  \alpha' \in \Delta' \mid \phi^* \alpha' = 0 \}$, where $\phi^*:
  \mathfrak{a}^{\prime *} \to \mathfrak{a}^*$ is the map induced by $\phi$..
\end{prop}

\subsection{Injection of Lie groups: examples and counterexamples}
\label{sec:ex_inj}
We describe an example that shows that composing an Anosov
representation $\rho: \Gamma \to G$  with an injective Lie group homomorphism $\phi: G \to G'$ does not always give rise to a (nontrivial) 
Anosov representation.

Let $G_1$ and $G_2$ be two copies of $\SL(2, \RR)$, $G=G_1 \times G_2$ and $G'= \SL(4,
\RR)$ with the natural injection $\phi:G_1 \times G_2 \to G'$. 
Let $P_1$ and $P_2$ be parabolic subgroups of $G_1$ and $G_2$. Up
to conjugation the proper parabolic subgroups of $G_1 \times G_2$
are $P_1 \times G_2$, $P_1 \times P_2$ and $G_1 \times P_2$.

Let $Q_0$ be the stabilizer in $G'$ of a line in $\RR^4$ and let $Q_2$ be 
the stabilizer of a $2$-plane. 
Let $\iota_1, \iota_2: \G \to G_1, G_2$ be non-conjugate discrete
and faithful 
representations of a surface group $\G$; $\iota_1$ is $P_1$-Anosov,
$\iota_2$ is $P_2$-Anosov. Define the two representations
\[ \rho= (\iota_1, 1) : \G \to G_1 \times G_2, \quad \rho'= (\iota_1,
\iota_2) : \G \to G_1 \times G_2; \]
then $\rho$ and $\rho'$ are $P_1 \times G_2$-Anosov. The
representation $\phi \circ \rho: \G \to G$ is $Q_0$-Anosov. However
the representation $\phi \circ \rho'$ is not $Q_0$-Anosov. If it were,
this would imply that for all $\g \in \G$ one has $|\tr( \iota_1( \g))|
\geq |\tr( \iota_2( \g))|$. This is impossible unless $\iota_1$ and
$\iota_2$ are conjugate (see \cite[Theorem~3.1]{Thurston_stretch}).

  The representation $\rho'$ is $P_1 \times P_2$-Anosov and, as a
  consequence, the composition $\phi \circ \rho'$ is
  $Q_2$-Anosov. However, by choosing  $\iota_2$ appropriately (not discrete and faithful),  one can ensure that 
  the
  composition $\phi \circ \rho'$ is also not $Q_2$-Anosov, and hence not Anosov with respect to any proper parabolic subgroup of $G'$.

\subsection{Equivariant maps}
\label{sec:exist_maps}
In this section we consider representations $\rho:\Gamma \to G$ that
admit a pair of continuous transverse equivariant
maps $(\xi^+,\xi^-)$, $\xi^\pm: \partial_\infty \Gamma \to \Ff^\pm$ without requiring any contraction property. 
We conclude that such representations are Anosov, at least up to
considering them into a subgroup of $G$.

Recall that a pair $(x^+, x^-)\in \mathcal{F}^+ \times
\mathcal{F}^-$ is transverse if it belongs to $\mathcal{X}
\subset \mathcal{F}^+ \times \mathcal{F}^-$. 
\begin{defi}
A pair  $(x^+, x^-)\in \mathcal{F}^+ \times
\mathcal{F}^-$ is said to be \emph{singular} if $\stab(x^+) \cap \stab(x^-)$ is a parabolic subgroup. 
There is one $G$-orbit of singular pairs, namely the orbit of $(
P^{+}_{\Theta}, P^{+}_{\iota( \Theta)})$; we denote the set of singular pairs by 
$\mathcal{S} \subset \mathcal{F}^+ \times
\mathcal{F}^-$. 
\end{defi}
\begin{defi}
    \label{defi:compat}
A pair of maps $(\xi^+, \xi^-)$, $\xi^\pm : \partial_\infty \G \to
\mathcal{F}^\pm$, is said to be  \emph{compatible} if
\[
  \forall t \in \partial_\infty \G,\, (\xi^+(t), \xi^-(t)) \in
  \mathcal{S}, \quad \forall t^+\neq t^- \in \partial_\infty \G,\,
  (\xi^+(t^+), \xi^-(t^-)) \in \mathcal{X}.\]
\end{defi}

Due to Proposition~\ref{prop:Anosov_Rep}
the main case we have to consider is when $G =\SL(V)$.

\begin{prop}\label{prop:Comp_Maps}
  Let $V=D \oplus H$ be a decomposition of a vector space $V$ into a
  line and a hyperplane. $Q^{+}_{0} = \stab(D)$ and
  $Q^{-}_{0} = \stab(H)$, and denote by $\mathcal{F}^+  = G/Q^{+}_0 = \PP(V)$ and $\mathcal{F}^-  = G/Q^{-}_0 = \PP(V^*)$ be the corresponding 
 homogeneous spaces. Let $\rho: \G \to \SL(V)$ be a representation.

Suppose that:
\begin{itemize}
\item $\rho$ is irreducible, and 
\item $\rho$ admits a compatible pair $(\xi^+, \xi^-)$ of continuous
  equivariant maps.
\end{itemize}

  Then $\rho$ is $(Q^{+}_{0} ,Q^{-}_{0})$-Anosov and $\xi^\pm$ are its
  Anosov maps.
\end{prop}

\begin{proof}
  Only the contraction property needs to be proved. The basic
  observation is that the action of the group $\G$ on its boundary at
  infinity $\partial_\infty \G$  already exhibits some contraction property and hence
 one gets contraction along the image of $\partial_\infty \G$
  by $\xi^+$. We will use the maps $\xi^\pm$ as much
  as possible to define an equivariant family of norms $(\| \cdot
  \|_{\hat{m}})_{\hat{m} \in \flow}$ and prove the contraction
  property. 
  
  The projection $\flow \to \G \backslash \flow$ is denoted by $\pi$.

  \medskip

  \noindent \emph{Definition of $\| \cdot
    \|_{\hat{m}}$}.

  We already observed earlier, that it is enough to prove dilatation on $(\sigma^*
  E^+)_{\hat{m}} = T_{\xi^+( \tau^+( \hat{m}))} \PP(V)$, thus we
  define norms only on these spaces.

The irreducibility of $\rho$ implies that for
  any $\hat{p}\in \flow$:
  \begin{itemize}
  \item there exist an open neighborhood $V_{\hat{p}}$ of $\hat{p}$
    and $t^{1}_{\hat{p}}, \dots, t^{n-1}_{\hat{p}} \in \partial_\infty
    \G$, such that
  \item for all $\hat{m} \in V_{\hat{p}}$, the sum $\xi^+( \tau^+(
    \hat{m})) + \xi^+( t^{1}_{\hat{p}}) + \cdots + \xi^+(
    t^{n-1}_{\hat{p}})$ is direct,
  \item for all $\hat{m} \in V_{\hat{p}}$ and all $i$, the sum $\xi^-( \tau^-(
    \hat{m})) + \xi^+( t^{i}_{\hat{p}})$ is direct,
  \item $\pi$ is injective in restriction to $V_{ \hat{p}}$.
  \end{itemize}

  In particular, for all $\hat{m} \in V_{\hat{p}}$ and all $i$,
  $\tau^\pm( \hat{m}) \neq t^{i}_{\hat{p}}$. Furthermore,
  the set $\{(\tau^+( \hat{m}), t^{i}_{ \hat{p}}, \tau^-(
  \hat{m}))\mid  \hat{m} \in V_{ \hat{p}}\}$ is contained in a compact
  subset of  the set of pairwise distinct triples of $\partial_\infty \G$, which we denote by $\partial_\infty \G^{(3)}$.
  
We first construct a basis
  $(e^{i}_{\hat{p}}(\hat{m}))_{i = 1, \dots, n-1}$ of $T_{\xi^+(
    \tau^+( \hat{m}))} \PP(V)$, for any  $\hat{m} \in
  V_{\hat{p}}$, that varies continuously with $\hat{m}$. 
  The vector 
  $e^{i}_{\hat{p}}(\hat{m})$ is defined by the property that the
  corresponding  map $\phi: \xi^+(
  \tau^+( \hat{m})) \to \xi^-( \tau^-( \hat{m}))$ 
  (under the
   isomorphism $T_{\xi^+(
   \tau^+( \hat{m}))} \PP(V) \cong \hom( \xi^+(
  \tau^+( \hat{m})), \xi^-( \tau^-( \hat{m})))$) 
  is such that
  $\xi^+( t^{i}_{\hat{p}}) = \{ v + \phi(v) \mid v \in \xi^+(
  \tau^+( \hat{m})) \}$. We say that $e^{i}_{\hat{p}}(\hat{m})$
  corresponds to the line $\xi^+( t^{i}_{\hat{p}})$.

  In turn, when $\hat{m}$ is in $V_{ \hat{p}}$, we can define a norm on $T_{\xi^+( \tau^+( \hat{m}))}
  \PP(V)$:
  \[ \| v \|_{\hat{p}} = \sum | \lambda^{i} | \text{ if } v =
  \sum \lambda^{i} e^{i}_{\hat{p}}(\hat{m}).\]

  By compactness, there exist $\hat{p}_1, \dots, \hat{p}_K$ such that
  $\G \backslash \flow = \bigcup \pi( V_{\hat{p}_k})$. For ease of
  notation, we will write $V_k = V_{ \hat{p}_k}$,
  $e^{i}_{k}(\hat{m}) = e^{i}_{\hat{p}_k}(\hat{m})$, $\| \cdot
  \|_{k} = \| \cdot \|_{\hat{p}_k}$. There exist continuous
  functions $f_1, \dots, f_K : \G \backslash \flow \to \RR_{\geq 0}$
  such that $\sum f_k = 1$ and $\mathrm{Supp}(f_k) \subset \pi( V_k)$
  for all $k$.

  \medskip

  For all $\hat{m} \in \flow$, we now define a norm $\| \cdot \|_{
    \hat{m}}$ on $T_{\xi^+( \tau^+( \hat{m}))}
  \PP(V)$ in a $\G$-equivariant way using the $\| \cdot \|_k$ and the
  $f_k$.
  For all $k$, if $\hat{m} \in \G \cdot V_k = \pi^{-1}( \pi(V_k))$,
  there exists a unique $\g^{ \hat{m}}_{k} \in \G$ such that $ \g^{
    \hat{m}}_{k} \cdot \hat{m}$ belongs to $V_k$. We then set
  \[\| v \|_{ \hat{m}} = \sum f_k( \pi( \hat{m})) \| \rho(
  \g^{ \hat{m}}_{k} ) v \|_k = \sum f_k( \pi( \hat{m})) \| v \|_{
    \hat{m}, k}, \text{ for  } v \in T_{\xi^+( \tau^+( \hat{m}))}
  \PP(V),\]
  this is well defined and continuous in $\hat{m}$.

  The relation $\g^{ \hat{m}}_{k} = \g^{ \g \hat{m}}_{k} \g$ is easy
  to check and implies the equivariance:
  \[\| \rho(\g) v \|_{ \g \cdot \hat{m}} = \sum f_k( \pi( \g \hat{m})) \| \rho(
  \g^{ \g \hat{m}}_{k} ) \rho(\g) v \|_k = \sum f_k( \pi( \hat{m})) \| \rho(
  \g^{ \hat{m}}_{k} ) v \|_k = \| v \|_{ \hat{m}}.\]
  
  \medskip

  \noindent \emph{The contraction property.} 

  To check the contraction property, as in the proof of
  Proposition~\ref{prop:mu_contrac}, we only need to prove weak
  dilatation, that is
  \begin{itemize}
  \item $\forall \hat{m} \in \flow$, $v \in T_{\xi^+( \tau^+( \hat{m}))}
  \PP(V)$, and sequence $(x_l)_{l \in \NN}$ in $\RR$ with $x_l \mapsto +\infty$, one has
  $\lim \| v \|_{\phi_{x_l} \hat{m}} =+\infty$.
  \end{itemize}
  Note that it is enough to have $\lim \| v \|_{\phi_{x_l} \hat{m},k}
  =+\infty$ for some $k$. Hence we can assume (up to passing to a subsequence) that, for
  all $l$, $\phi_{x_l} \hat{m} \in \G \cdot V_k$. Let $\g_l$ be such
  that $\g_l \cdot \phi_{x_l} \hat{m} \in V_k$. Thus, by definition,
  \[
    \| v \|_{ \phi_{x_l} \hat{m},k} = \sum | \lambda^{i}_{l} | \text{
      with }  \rho( \g_l) v  = \sum \lambda^{i}_{l} e^{i}_{k} ( \g_l
    \cdot \phi_{x_l} \hat{m});
  \]
  hence $v  = \sum \lambda^{i}_{l} \epsilon^{i}_{l}$ where
  $\epsilon^{i}_{l} = \rho( \g^{-1}_l) e^{i}_{k} ( \g_l 
    \cdot \phi_{x_l} \hat{m})$ is the vector of 
 $T_{\xi^+( \tau^+(
    \hat{m}))} \PP(V) \cong \hom( \xi^+( \tau^+(
  \hat{m})), \xi^-( \tau^-( \hat{m})) )$ which corresponds to the
  line $\xi^+( \g^{-1}_l \cdot t^{i}_{k})$. It is therefore enough to
  prove that $\epsilon^{i}_{l} \mapsto 0$ which is equivalent to
  $\xi^+( \g^{-1}_l \cdot t^{i}_{k}) \mapsto \xi^{+}( \tau^+ ( 
  \hat{m}))$. From the continuity of $\xi^+$, it suffices to prove
  $\g^{-1}_l \cdot t^{i}_{k} 
  \mapsto \tau^+ ( \hat{m})$.

  For this, note first that, since the sequence $(\g_l \cdot \phi_{x_l}\hat{m})_{l \in
    \NN}$ is contained in $V_k$, the triples 
  $(\tau^+( \g_l \cdot \phi_{x_l}\hat{m}), t^{i}_{k}, \tau^-( \g_l \cdot
  \phi_{x_l}\hat{m}))$ belong to a compact subset of $\partial_\infty \G^{(3)}$.
  The action of $\G$ on
  $\partial_\infty \G^{(3)}$ is proper and cocompact (see \cite{Bowditch_convergence}), thus 
    $\g_l \mapsto \infty$ implies that
  the sequence $(\tau^+( \phi_{x_l}\hat{m}), \g^{-1}_l \cdot t^{i}_{k}, \tau^-(
  \phi_{x_l}\hat{m})) = (\tau^+( \hat{m}), \g^{-1}_l \cdot t^{i}_{k}, \tau^-(
  \hat{m}))$ diverges in $\partial_\infty \G^{(3)}$. This means either
  that $\g^{-1}_l \cdot t^{i}_{k} \mapsto \tau^+( \hat{m})$ or that
  $\g^{-1}_l \cdot t^{i}_{k} \mapsto \tau^-( \hat{m})$. The second 
  possibility is easily eliminated (as it would contradict that
  $\phi_{x_l} \hat{m} = \g^{-1}_l \cdot \hat{m}_l$ tends
  to 
  $\tau^+( \hat{m})$ with $\hat{m}_l = \g_l \cdot \phi_{x_l} \hat{m}$ being bounded). Thus 
  we conclude that $\g^{-1}_l \cdot t^{i}_{k} \mapsto \tau^+( \hat{m})$.
\end{proof}

From Proposition~\ref{prop:Comp_Maps} and Proposition~\ref{prop:Anosov_Rep},  we deduce the following 
\begin{thm}\label{thm:Ano_Zd}

  Let $\rho: \G \to G$ be a Zariski dense representation and $P^+, P^- < G$ opposite parabolic
  subgroups of $G$. Suppose 
  that $\rho$ admits a pair  of equivariant continuous compatible maps (Definition~\ref{defi:compat}) $\xi^+: \partial_\infty \G \to \Ff^+$, $\xi^-: \partial_\infty \Gamma \to \Ff^-$.
    
  Then the representation $\rho$ is $(P^+, P^-)$-Anosov and $(\xi^+,
  \xi^-)$ are the associated Anosov maps.
\end{thm}

\begin{proof}
  By classical representation theory there exists an
  irreducible $G$-module $V$ admitting a decomposition $V=D\oplus H$
  into a line and a hyperplane and such that $P^+ = \stab_G(D)$ and
  $P^- = \stab_G(H)$. The result then follows from
  Proposition~\ref{prop:Comp_Maps} and 
  Proposition~\ref{prop:Anosov_Rep}. 
\end{proof}

\begin{remark}
A $G$-module satisfying the hypothesis of
Proposition~\ref{prop:Anosov_Rep} is easy to find; e.g.\ $\wdg{p}
\mathfrak{g}$ where $p = \dim \mathfrak{p}^+$. Taking the 
irreducible factor containing the line $\wdg{p}
\mathfrak{p}^+$
 gives an irreducible module $V$ satisfying the requirements of the above proof.
\end{remark}

\begin{remark}\label{rem:Ano_in_Zcl}
  As a conclusion of Theorem~\ref{thm:Ano_Zd}, a representation $\rho: \G \to G$ admitting a pair of compatible $\rho$-equivariant
  and continuous maps $(\xi^+, \xi^-)$ is Anosov when considered as a representation into
  its \emph{Zariski closure} $H$ (or more precisely into the quotient
  of $H$ by its radical, since we define Anosov representations only
  into semisimple Lie groups). Proposition~\ref{prop:injLieGr} then gives sufficient conditions for the representation $\rho: \G \to G$ to be Anosov. 
\end{remark}

\begin{remark}
Recently Sambarino \cite{Sambarino} established counting theorems for representations of fundamental groups of negatively curved manifolds into $\SL( V)$ satisfying the assumptions of Proposition~\ref{prop:Comp_Maps}. Using Proposition~\ref{prop:mu_contrac} his results should extend to all $(Q^{+}_{0} ,Q^{-}_{0})$-Anosov representations of fundamental groups of negatively curved manifolds into $\SL( V)$.
\end{remark}

\subsection{Parabolic subgroups conjugate to their opposite}
\label{sec:para_opp}

A parabolic subgroup $P^{+}_{\Theta}$ is conjugate to
$P^{-}_{\Theta}$ if and only if $\Theta = \iota (\Theta)$. 
Lemma~\ref{lem:ThetaAn} states that any Anosov representation is
$P_\Theta$-Anosov for some $\Theta$ satisfying $\Theta = \iota
(\Theta)$. In that case the two homogeneous spaces
$\mathcal{F}^{+}_{\Theta} = G/ P^{+}_{\Theta}$ and
$\mathcal{F}^{-}_{\Theta} = G/ P^{-}_{\Theta}$ are canonically identified. Hence,
by  uniqueness (Lemma~\ref{lem:uniqu}),  there is a single Anosov map 
\[\xi  = \xi^+ = \xi^-: \partial_\infty \G \to \mathcal{F}^{+}_{\Theta} =
\mathcal{F}^{-}_{\Theta} = \mathcal{F}^{}_{\Theta}.\] 

\begin{defi}
  \label{defi:transmap}
  A map $\xi: \partial_\infty \Gamma \to \mathcal{F}_{\Theta}$ is said to be  \emph{transverse} if for all $t^+ \neq t^-\in \partial_\infty \Gamma$,  the pair  $(\xi( t^+),\xi(t^-)) \subset \Ff_\Theta \times \Ff_\Theta$ is transverse.
\end{defi}

A special case of Theorem~\ref{thm:Ano_Zd} is the following 

\begin{cor}\label{cor:Map_Ano}
Let  $\rho: \G \to G$ be a
  Zariski dense representation.   Suppose that $\Theta= \iota(\Theta)$ and assume that there exists  a continuous
  $\rho$-equivariant map $\xi:
  \partial_\infty \G \to \mathcal{F}_{\Theta}$ such that for all $t^+ \neq t^-\in \partial_\infty
  \G$,  the pair  $(\xi( t^+),\xi(t^-))$ is transverse.

  Then the representation $\rho$ is $P_\Theta$-Anosov.
\end{cor}

When $\Theta = \iota(\Theta)$, there exists an
irreducible $G$-module $V$ with a $G$-invariant non-degenerate bilinear form $F$,  and
an isotropic line $D$ in $V$ such that $P^{+}_{\Theta} =
\stab_G(D)$. 
We denote by $G_F$ the automorphism group of $(V,F)$. The irreducibility implies that $F$ can be
supposed to be either symmetric or skew-symmetric, i.e.\  $G_F
= \O(V,F)$ or $\Sp(V,F)$. We denote by $Q_0$ the stabilizer in $G_F$ of the line $D$. This discussion together
with Proposition~\ref{prop:Anosov_Rep} implies

\begin{prop}\label{prop:ano_Qzero}
  A representation $\rho: \G \to G$ is Anosov if and only if there is
  a self-dual $G$-module $(V,F)$ with $\phi: G\to G_F$ the
  corresponding homomorphism, such that $\phi \circ \rho$ is $Q_0$-Anosov.
\end{prop}

Applying the construction of Lemma~\ref{lem:red_spe}.(\ref
{item3:lemredspe}) below we deduce the following 
\begin{cor}\label{cor:allinO}
A representation $\rho: \G \to G$ is Anosov if and only if there is a Lie group homomorphism  $\phi: G \to \O(V,F)$  such that  $\phi \circ\rho: \G \to \O(V,F)$ is a $Q_0$-Anosov representation. 
\end{cor}

\section{Discreteness, metric properties and openness}

\subsection{Quasi-isometric embeddings and well displacing}
\label{sec:discreteness}
The group $\G$ is endowed with the left invariant distance $d_\G$
coming from a word length $\ell_\G$. The group $G$ is endowed with the distance
  $d_G$ comming from a left invariant Riemannian metric. 
With this
distance $G$ is quasi-isometric to any homogeneous space $G/M$ where
$M$ is a compact subgroup, endowed with a left invariant Riemannian
metric. The translation length of an element $\g\in \G$ (resp.\ $g$
in $G$) is
\[ t_\G(\g) = \inf_{ x \in \G} d_\G( x, \g x) \quad (\textrm{resp. }
t_G(g) = \inf_{ x \in G} d_G( x, g x)).\] 

\begin{defi}
A representation $\rho: \Gamma \to G$ is a \emph{quasi-isometric embedding} if there
exist positive constants $K,C$ such that, for every 
$\g \in \G$,
\[K^{-1} \ell_\Gamma( \g) -C \leq d_G( 1, \rho( \g)) \leq K \ell_\G(
\g) +C\]
 (for generalities on quasi-isometries, quasi-geodesics, etc.\
see \cite[Chapitre~3]{Coornaert_Delzant_Papadopoulos}). A
representation $\rho: \Gamma \to G$ is said to be \emph{well
  displacing} \cite{Delzant_Guichard_Labourie_Mozes, Labourie_energy} if, for all $\g\in \Gamma$, 
\[K^{-1}
t_\Gamma( \g) -C \leq t_G( \rho( \g)) \leq K t_\G( \g) +C.\]
\end{defi}
\begin{remark}
  Note that, since $\G$ is finitely generated, the upper bound is automatically 
  satisfied. Furthermore, from the classical equality $t_{G/K}( g) =
  \lim d_{G/K}( K, g^n K)/n$, where $G/K$ is the symmetric space associated
  with $G$, endowed with a left invariant Riemannian metric, it follows that any representation $\rho$ which is a quasi-isometric embedding is also
  well displacing.
\end{remark}

\begin{thm}\label{thm:Ano_QIE}
  Let $\rho$ be an Anosov representation, then $\rho:\G \to G$ is a
  quasi-isometric embedding. In particular: 

  \begin{inparaenum}
  \item $\ker \rho$ is finite,
  \item $ \rho(\Gamma) < G$ is discrete, and 
  \item $\rho$ is well displacing.
  \end{inparaenum} 
\end{thm}

\begin{proof}
  We consider $\hat{\sigma}$, $\hat{ \beta}$, $\mu_{+,\Theta}$, $\mu_{-,\Theta}$, which were introduced 
  in Section~\ref{sec:lifting-section-gm}. Since $\G$ and $\flow$ are
  quasi-isometric, it is enough to show that $\hat{\beta} : \flow \to
  G/M$ is a quasi-isometric embedding. As a matter of fact
  Proposition~\ref{prop:mu_contrac} already shows that there are
  constants $(K,C)$ such that the restriction of $\hat{ \beta}$ to any
  $\RR$-orbit is a $(K,C)$-quasi-geodesic. To conclude, one has to
  note the following property of $\flow$, which is a consequence of its 
  hyperbolicity: there exists $D \geq 0$ such that for any $m,p$ in
  $\flow$ there exist $m_0, p_0\in \flow$ that are on the same
  $\RR$-orbit and such that $d(m,m_0) \leq D$ and $d(p,p_0) \leq D$.
\end{proof}

In the case when $\G = \pi_1(\Sigma)$ is the fundamental group of a
closed connected oriented surface of genus $\geq 2$, following
arguments of \cite[Section 6.3]{Labourie_energy}, or when $\G$ is a free group, following the arguments of \cite[Theorem 3.3]{Minsky} 
one can deduce from
Theorem~\ref{thm:Ano_QIE} and from Theorem~\ref{thm:Ano_QIEUnif} that
the action of the outer automorphism group of $\G$ on the set of Anosov 
representations is proper. We expect that the arguments of \cite[Theorem 3.3]{Minsky} can be generalized to arbitrary word hyperbolic groups. 

 For this let $\hom_{\textup{Anosov}}(\G, G)$ denote the set
of Anosov representations and denote by 
$(\hom_{\textup{Anosov}}(\G, G)/G)^{red}$ its Hausdorff quotient
(i.e.\ two elements $x, y\in \hom_{\textup{Anosov}}(\G, G)/G$ are
identified if every neighborhood of $x$ meets any neighborhood of
$y$).

\begin{cor}
Let $\Sigma$ be a connected orientable surface of negative Euler characteristic and $\G = \pi_1(\Sigma)$. Then the outer automorphism group of $\G$ 
acts properly on $(\hom_{\textup{Anosov}}(\G, G)/G)^{red}$. 
\end{cor} 

\subsection{Proximality}
\label{sec:proximality}
In this section we show that images of Anosov representations have
strong proximality properties.

Recall that  $(P^+, P^-)$ is a fixed 
pair of opposite parabolic subgroups in $G$ and 
$\mathcal{F}^\pm=G/P^\pm$ denote the corresponding homogeneous spaces.

For $x^-\in \mathcal{F}^-$,
set $V^-(x^-)= \{x \in \mathcal{F}^+ \mid x \text{ and }x^- \text{ are
 not transverse} \}$.

\begin{defi}
  An element $g \in G$ is  said to be \emph{proximal} relative to $\Ff^+$ (or
  $\mathcal{F}^+$-proximal) if $g$ has two fixed points, $x^+ \in
  \mathcal{F}^+$ and $x^- \in \mathcal{F}^-$ with $x^+ \notin
  V^-(x^-)$ and such that for all $x \notin V^-(x^-)$, $\lim_{n \to +\infty} g^n 
  \cdot x = x^+$.

A subgroup $\Lambda< G$ is proximal if it contains at least one proximal element.
\end{defi}
When $g$ is proximal, the fixed points $x^+$ and $x^-$ are uniquely determined, we denote them 
by $x^{+}_{g}$ and $x^{-}_{g}$.

When $\Lambda < G$ is a subgroup which is proximal with
respect to $\mathcal{F}^\pm$, then there exists a well
defined closed $\G$-invariant minimal set $\mathcal{L}^\pm_\Lambda \subset
\mathcal{F}^\pm$, which is called the {\em limit set} of $\Gamma$, see 
\cite{Benoist}; it is the closure of the set of attracting fixed
points of proximal elements in $ \Lambda$.

Let $d$ be a (continuous) distance on $\mathcal{F}^+$ and define
\begin{itemize}
\item  
for
$x^+\in \mathcal{F}^+$, $b_\epsilon( x^+) = \{ x \in \mathcal{F}^+
\mid d(x,x^+) \leq \epsilon \}$, and
\item for
$x^-\in \mathcal{F}^-$, $B_\epsilon( x^-) = \{ x \in \mathcal{F}^+
\mid d(x,V^-(x^-)) \geq \epsilon \}$.
\end{itemize}
\begin{defi}
  An element $g$ is $(r,\epsilon)$-\emph{proximal} (or
  $(r,\epsilon)$-$\mathcal{F}^+$-proximal) if $g$ has two fixed points
  $x^+ \in \mathcal{F}^+$ and $x^- \in \mathcal{F}^-$ such that $d(x^+
  , 
  V^-(x^-))\geq r$, $g\cdot B_\epsilon(x^-)\subset b_\epsilon(x^+)$
  and $g|_{B_\epsilon(x^-)}$ is $\epsilon$-contracting.
\end{defi}

In \cite{Abels_Margulis_Soifer_Prox} Abels, Margulis and So\u{\i}fer investigated proximality properties of strongly irreducible subgroups of $\GL(V)$. 
In order to restate their result let us make the following 
\begin{defi}\label{defi:AMS}
  A subgroup $\Lambda <G$ is said to be \emph{(AMS)-proximal}
  (or proximal in the sense of Abels, Margulis ans So\u{\i}fer) 
  relative to $\mathcal{F}^+$ if there exist constants $r>0$ and
  $\epsilon_0>0$ such that, for any $\epsilon < \epsilon_0$ the
  following holds:
  \begin{itemize}
  \item there exists a finite subset $S\subset \Lambda$ with the property that, for any
    $\delta \in \Lambda$, there is $s\in S$ such that $s\delta$ is
    $(r, \epsilon)$-proximal.
  \end{itemize}

  A representation $\rho:\G \to G$ is said to be \emph{(AMS)-proximal} if $\ker \rho$ is finite
  and $\rho(\G)$ is (AMS)-proximal.
\end{defi}

With this, the result of Abels, Margulis and So{\u\i}fer can be reformulated as follows. 
\begin{thm}\label{thm:AMS}
\cite[Theorem~4.1]{Abels_Margulis_Soifer_Prox}

  If $\Lambda < \SL(V)$ is strongly irreducible (i.e.\ any
  finite index subgroup acts irreducibly on $V$) then $\Lambda$ is (AMS)-proximal relative to $\PP(V)$.
\end{thm}

For Anosov representations we have the following: 
\begin{thm}\label{thm:Ano_AMS}
  If $\rho: \G \to G$ is $P^+$-Anosov, then $\rho$ is  (AMS)-proximal 
  relative to $\mathcal{F}^+$.
\end{thm}

\begin{proof}
  Lemma~\ref{lem:ano_proxi} already shows that $\rho(\g)$ is proximal relative to $\mathcal{F}^\pm$, 
  when $\g \neq 1$.
The following theorem thus implies the statement. 
\end{proof}  

\begin{thm}\label{thm:proxi_AMS}
Let $\Lambda<G$ be a
subgroup. If $\Lambda$ is proximal relative to $\mathcal{F}^\pm$, then $\Lambda$ is (AMS)-proximal relative to $\mathcal{F}^\pm$.
\end{thm} 
\begin{proof}
The strategy is to reduce to the situation of a strongly irreducible subgroup of $\SL(V)$ and then to apply Theorem~\ref{thm:AMS}.

Let $V$  be an irreducible $G$-module with a decomposition 
  $V=D\oplus H$ into a line and a hyperplane, such that 
  $P^+ = \stab_G(D)$ and $P^-= \stab_G(H)$, and denote by $\phi:G \to\SL(V)$ the embedding. The induced maps 
  $\phi^+ : \mathcal{F}^+ \to \PP(V)$ and $\phi^- : \mathcal{F}^- \to
  \PP(V^*)$ satisfy the following property: for $x^-\in \mathcal{F}^-$ we have 
  $V^-(x^-) = (\phi^+)^{-1}( \PP( \phi^-(x^-)))$, where one
  considers $\phi^-(x^-)$ as a hyperplane in $V$. 
  Thus, if $g\in G$ is proximal with respect to $\mathcal{F}^\pm$,
  then $\phi(g)$ is proximal with respect to $\PP(V)$  and
  $\PP(V^*)$. 
  Moreover, if $\phi(g)$ is
  $(r,\epsilon)$-proximal on $\PP(V)$ then $g$ is $(r',
  \epsilon')$-proximal on $\mathcal{F}^+$ for some functions
  $r'=r'(r)$ and $\epsilon' = \epsilon'( r,\epsilon)$ satisfying
  $\epsilon'( r,\epsilon) \xrightarrow{\epsilon \to 0}
  0$.  Thus from now on we suppose that $G= \SL(V)$.
  
We consider the $\Lambda$-invariant subspace $W = \bigcap_{x\in
  \mathcal{L}^-_{\Lambda}}  x$, where $x$ is regarded as hyperplane in
$V$, and assume that $W$ is non-empty. Using the minimality of the
action of $\Lambda$ on $\mathcal{L}^+_\Lambda$ we have
$\mathcal{L}^+_\Lambda \cap \PP(W) = \emptyset$. Moreover  $\Lambda<
\stab(W)$ and thus, $\Lambda$ defines a subgroup $\bar{\Lambda} <
\SL(V/W)$ which is $\PP(V/W)$ and $\PP(V/W^*)$ 
proximal. Lemma~\ref{lem:prox_quot} below (with $\delta =
d(\mathcal{L}^+_\Lambda, \PP(W))>0$)
  implies that if $\bar{ \Lambda}$ is (AMS)-proximal then so is
  $\Lambda$. 
      
  Therefore,  we 
  can assume $\bigcap_{x \in \mathcal{L}^-_\Lambda}  x = \{0\}$. Analogously we
  can assume $\sum_{x \in \mathcal{L}^+_\Lambda} x = V$. By Lemma~\ref{lem:irred_maps} below this implies that $\Lambda$ is strongly irreducible. 
  
  Thus Theorem~\ref{thm:AMS} implies that 
  $\Lambda$ is (AMS)-proximal. 
\end{proof}

\begin{lem}
  \label{lem:prox_quot}
  Let $\delta>0$ be a real number and $W\subset V$ two vector
  spaces. Then there are functions $r': \RR_{>0} \to \RR_{>0}$ and
  $\epsilon': \RR^{2}_{>0} \to \RR_{>0}$ satisfying $\epsilon'(
  r,\epsilon) \xrightarrow{\epsilon \to 0} 0$ and such that:
  
  \begin{asparaitem}
  \item If $g\in \stab(W)$ is proximal with $W \subset x^{-}_{g}$ and
    $d(W, x^{+}_{g} ) \geq \delta$ and $\pi(g) \in \SL(V/W)$ is $(r,
    \epsilon)$-proximal, then $g$ is $(r',
    \epsilon')$-proximal.
  \end{asparaitem}
\end{lem}
\begin{proof}
 The statement follows from a direct compactness argument.
\end{proof}

\begin{lem}
  \label{lem:irred_maps}
  Let $\Lambda< \SL(V)$ be a $\mathcal{F}^\pm$-proximal subgroup. 
  Suppose that 
   \[ \bigcap_{x \in \mathcal{L}^-_\Lambda}  x = \{0\} \quad \text{and} \quad 
   \sum_{x \in \mathcal{L}^+_\Lambda} x  = V.\]  
   Then $\Lambda$ is strongly irreducible. 
\end{lem}
\begin{proof}
  Let $\Lambda'$ be a finite index subgroup of $\Lambda$. Any closed
  $\Lambda'$-invariant subset of $\mathcal{L}^\pm_\Lambda$ is either $\emptyset$
  or $\mathcal{L}^\pm_\Lambda$. Suppose that $W \subset V$ is
  $\Lambda'$-invariant. For any proximal element $g\in \Lambda'$ one has
  \[W = x^{+}_{g} \cap W \oplus x^{-}_{g} \cap W,\]
  where $x^\pm_g$ are the attractive fixed points of $g$ in $\PP(V)$ and $\PP(V)^*$ respectively. 
  Hence $x^{+}_{g} \subset W $ or $x^{-}_{g} \supset
  W$,  and consequently one of the following two  closed $\Lambda'$-invariant subset is
  nonempty:
  \[\{ x \in \mathcal{L}^+_\Lambda \mid x\subset W  \}\quad
  \text{or} \quad \{ x \in \mathcal{L}^-_\Lambda  \mid x \supset W
  \}.\]
  If the first is nonempty, one concludes that  $W \supset  \sum_{x \in \mathcal{L}^+_\Lambda} x  = V$; if  the second set is nonempty, $W
  \subset  \bigcap_{x \in \mathcal{L}^-_\Lambda} x = \{0\}$. In either
  case $W$ is trivial, proving the strong irreducibility of $\Lambda$.
\end{proof}

\subsection{Openness}
\label{sec:openness}

In this section we prove that the set of Anosov representations is
open. In the case of fundamental groups of
negatively curved Riemannian manifolds, this is proven in
\cite[Proposition~2.1]{Labourie_anosov}.

Let $P$ be a parabolic subgroup of $G$. Denote by $\homPano( \G, G)$
the set of $P$-Anosov representations.

\begin{thm}\label{thm:Ano_open}
  The set $\homPano( \G,G)$ is open in $\hom( \G, G)$.

  Furthermore the map $\homPano(\G,G) \to C^0( \partial_\infty \G,
  \mathcal{F})$ associating to a representation $\rho$ its Anosov map
  is continuous.
\end{thm}

As a corollary a small adaptation of the proof of
Theorem~\ref{thm:Ano_QIE} gives the following uniformity statement:

\begin{thm}\label{thm:Ano_QIEUnif}
  Let $\rho: \G\to G$ be a $P$-Anosov representation. Then there exist
  constants $K,C >0$ and an open neighborhood $U$ of $\rho$ in $\hom(
  \G, G)$ such that every representation $\rho'\in U$ is a
  $(K,C)$-quasi-isometric embedding.
\end{thm}

\begin{proof}[Proof of Theorem~\ref{thm:Ano_open}]

By Proposition~\ref{prop:Anosov_Rep} we can reduce to the case when $G=
\SL(V)$ and $P= Q^{+}_{0}$. Let $\mathcal{F}^+ = \PP(V)$ and
$\mathcal{F}^- = \PP(V^*)$. 
Given a representation $\rho: \Gamma \to G$ consider the
bundles $\mathcal{F}^{+}_{\rho} = \flow \times_\rho 
\mathcal{F}^+$ and $\mathcal{F}^{-}_{\rho} = \flow \times_\rho
\mathcal{F}^-$. We denote by $(d_m)_{m \in \G \backslash \flow}$ a
(continuous) family of distances on the fibers of
$\mathcal{F}^{+}_{\rho}$, i.e.\ $d_m$ is a distance on
$(\mathcal{F}^{+}_{\rho})_m \cong \PP(V)$. We will regard an element
in $(\mathcal{F}^{-}_{\rho})_m \cong \PP(V^*)$ as a hyperplane in
$(\mathcal{F}^{+}_{\rho})_m$. 

Suppose now that $\rho$ is Anosov. Let $\xi^+$ and
$\xi^-$ be the Anosov maps, and $\sigma: \G\backslash \flow \to \mathcal{X}_\rho \subset \mathcal{F}^{+}_{\rho} \times
\mathcal{F}^{-}_{\rho}$ the section defined by $(\xi^+, \xi^-)$. 
For all $\epsilon >0$ consider (topological)
subbundles $B_\epsilon$ and $b_\epsilon$ defined fiberwise, i.e.\ for all $m\in \G
\backslash \flow$,
\begin{align*}
  (B_\epsilon)_m & = \{ l \in (\mathcal{F}^{+}_{\rho})_m \mid d_m( l,
  \xi^-(m)) > \epsilon \}, \\
  (b_\epsilon)_m & = \{ l \in (\mathcal{F}^{+}_{\rho})_m \mid d_m( l,
  \xi^+(m)) < \epsilon \}.
\end{align*}
The space of continuous sections of those bundles are denoted by
$\Gamma( \mathcal{F}^{+}_{\rho})$, $\Gamma( B_\epsilon)$ and $\Gamma(
b_\epsilon)$. Note that $\Gamma( \mathcal{F}^{+}_{\rho})$ is a
complete metric space. Furthermore there exists $\epsilon_0>0$ such
that for all $\epsilon < \epsilon_0$, $b_\epsilon \subset B_\epsilon$.

 The flow $\phi_t$ acts naturally on
$\mathcal{F}^{+}_{\rho}$ and on the space of sections $\Gamma(
\mathcal{F}^{+}_{\rho})$.  
The contraction property implies: 
\begin{itemize}
\item for all $\epsilon < \epsilon_0$
there exists $T_\epsilon$ such that for all $t \geq T_\epsilon$,  and for all
$f\in \Gamma( B_\epsilon)$ one has  $\phi_{-t} \cdot f \in \Gamma(b_\epsilon)$. Moreover, for all $t \geq T_\epsilon$, the
map $\phi_{-t}: \Gamma( B_\epsilon) \to \G( B_\epsilon)$ is
$\epsilon$-contracting.
\end{itemize}
Now let $U$ be a neighborhood of $\rho$ in $\hom( \G, G)$. Consider the bundles 
$\mathcal{F}^{\pm}_{U}$  over $U \times \G \backslash \flow$: 
\[ \mathcal{F}^{\pm}_{U} = ( U \times \flow) \times_\rho
\mathcal{F}^{\pm},\] 
where the action of $\G$ on $U \times \flow$ is trivial on the first factor, i.e.\  $\g \cdot ( \rho', \hat{m}) = (\rho', \g
\cdot \hat{m}).$

For $U$ small enough the bundles $( \mathcal{F}^{\pm}_{\rho'})_{\rho'
  \in U}$ are all isomorphic, i.e.\ there exists a bundle isomorphism $\psi: \mathcal{F}^{+}_{U} \to U \times \mathcal{F}^{+}_{\rho}$ with $\psi
|_{\mathcal{F}^{+}_{\rho}} = \id$ (\cite[p.~53]{Steenrod}). Note that the flow $\phi_t$ acts on
$\mathcal{F}^{+}_{U}$ and hence on the space of sections $\G(\mathcal{F}^{+}_{U})$.

By continuity (and again for $U$ small enough) there exists $\epsilon_1 > 0$  such that for all $\epsilon< \epsilon_1$ there exists $T_0$ such that for all $t\geq T_0$ the following holds: 
\begin{itemize}
\item for any section
$f_U \in \G (\psi^{-1}(U \times B_\epsilon))$, its image $\phi_{-t} \cdot f_U$ by the flow
belongs to $\G (\psi^{-1}(U \times b_\epsilon))$; moreover the map
$\phi_{-t}: \G (\psi^{-1}(U \times B_\epsilon)) \to \G (\psi^{-1}(U
\times B_\epsilon))$ is $2\epsilon$-contracting.
\end{itemize}

Since $\{ \phi_{-t} \}_{t \geq T_0}$ is a commuting family of contracting
maps, they have a unique common fixed point $\xi^{+}_{U}$ in $\G
(\psi^{-1}(U \times B_\epsilon))$. Certainly $\xi^{+}_{U}$ is also
fixed by $\phi_t$ for any $t\in\RR$. Furthermore, the contraction
property implies that $\xi^{+}_{U} |_{\{ \rho\} \times \G \backslash
  \flow} = \xi^+$. Similarly one finds $\xi^{-}_{U}$ extending
$\xi^-$. 

Since $\xi^{+}_{U}$ and $\xi^{-}_{U}$ are transverse in restriction to
$\{ \rho\} \times \G \backslash \flow$, for $U$ small enough, $\xi^{+}_{U}$ and $\xi^{-}_{U}$ are transverse on $U
\times \G \backslash \flow$. Therefore $\mathcal{X}_U = (
U \times \flow) \times_\rho \mathcal{X} \subset \mathcal{F}^{+}_{U} \times
\mathcal{F}^{-}_{U}$ admits a section $\sigma_U = ( \xi^{+}_{U},
\xi^{-}_{U})$ that is flat along flow lines.

The action of $\phi_t$ on $(\sigma^{*}_U E^+)|_{\{ \rho\} \times \G
  \backslash \flow}$ (resp.\ $(\sigma^{*}_U E^-)|_{\{ \rho\} \times \G
  \backslash \flow}$) is dilating (resp.\ contracting) (see
Section~\ref{sec:defiRiem} for the definition of $E^\pm$). Hence,
again for $U$ small enough, this implies that the action of $\phi_t$ is
 dilating on $\sigma^{*}_U E^+$ (resp.\ contracting on $\sigma^{*}_U
E^-$). 

This shows that there exists a neighborhood $U$ of $\rho$ in $\hom(\G,G)$ such that any $\rho'\in U$ is $Q^{+}_{0}$-Anosov, and moreover the Anosov map varies continuously with $\rho'$.
\end{proof}

\subsection{Groups of rank one}
\label{sec:rank_one}
When $G$ is of rank one, there is only one conjugacy class of
parabolic subgroups. Henceforth we can talk of Anosov representations
unambiguously. Furthermore two points in $\mathcal{F}= G/P$ are
transverse if and only if there are distinct.

A subgroup $\Lambda < G$ is said to be \emph{convex cocompact} if it acts
properly discontinuously and cocompactly on a convex subset
$\mathcal{C}$ of the symmetric space $G/K$. In that case $\Lambda$ is
hyperbolic, 
 and $\partial_\infty \Lambda \cong \partial_\infty
\mathcal{C}$ injects into $\partial_\infty (G/K) \cong G/P$ and the
injection $\Lambda \to G$ is a quasi-isometric embedding. Conversely, 
\cite[Corollaire~1.8.4, Proposition~1.8.6]{Bourdon_conforme} if the injection
$\Lambda \to G$ is a 
quasi-isometric embedding then $\Lambda$ is
convex cocompact. Thus from \cite{Bourdon_conforme} and  the
characterizations of Anosov representations one has:

\begin{thm}\label{thm:Ano_rk_one}
Let $G$ be a Lie group of real rank one. 
  Let $\rho: \G \to G$ be a representation. Then
  the following are equivalent:
  \begin{enumerate}
  \item $\rho$ is Anosov.
  \item There exists $\xi: \partial_\infty \G \to G/P$ a continuous,
    injective and equivariant map.
  \item $\rho$ is a quasi-isometric embedding.
  \item $\ker \rho$ is finite and $\Lambda = \rho( \G)$ is convex cocompact.
  \end{enumerate}
\end{thm}

\section{Examples}
\label{sec:examples}

In this section we give various examples of Anosov representations. 

\subsection{Groups of rank one}
\label{sec:ex_rk_one}

If $G$ is a semisimple Lie group of rank one and $\G < G$ is
a convex cocompact subgroup, then the injection $\iota: \G \to G$ is Anosov by
Theorem~\ref{thm:Ano_rk_one}.
This gives the following examples: 

\begin{enumerate}
\item Inclusion of uniform lattices. 
\item Embeddings of free groups as Schottky groups. 
\item Embeddings of Fuchsian groups into $\PSL(2,\RR)$. 
\item Embeddings of quasi-Fuchsian groups into $\PSL(2,\CC)$. 
\end{enumerate}

Composing the representation $\iota: \Gamma \to G $ with an
embedding $\phi: G \to G'$ of $G$ into a Lie group of higher rank $G'$,
we obtain an Anosov representation $\phi\circ\iota: \Gamma \to G'$
(Proposition~\ref{prop:inj_rk_one}).  By Theorem~\ref{thm:Ano_open}
any small enough deformation of $\phi\circ \iota$ is also an Anosov
representation. In many cases there exist small deformations with
Zariski dense image. In some particular cases all deformations of
$\phi\circ \iota$ remain Anosov representations.

We list some examples, several of which will be discussed in more detail below: 
\begin{enumerate}
\item Holonomies of convex real projective structures:
  Let $\iota :\Gamma \to
  \SO(1,n)$ be the embedding of a uniform lattice. Consider $\phi\circ
  \iota: \Gamma \to \PGL(n+1,\RR)$, where $\phi:\SO(1,n) \to
  \PGL(n+1,\RR)$ is the standard embedding; this is a $Q_0$-Anosov
  representation, where $Q_0< \PGL(n+1,\RR)$ is the stabilizer of a
  line.  Moreover, $\iota(\Gamma)$ preserves the quadric in $\PP^n(\RR)$
  (the Klein model for hyperbolic space) and acts on it properly
  discontinuously with compact quotient. In particular, $\phi
  \circ\iota$ is an example of a holonomy representation of a convex real projective structure. 
  In
  \cite{Benoist_convex_III} Benoist shows that the entire connected
  component of $\hom(\Gamma, \PGL(n+1,\RR))$ containing $\phi\circ
  \iota$ consists of holonomies of convex real projective structures, hence of Anosov
  representations. More details
  are given in Section~\ref{sec:divi_conv}.
\item Hitchin component:
  Let $\G$ be the fundamental group of a
  closed connected oriented surface $\Sigma$ of genus $\geq 2$, and $\iota: \Gamma \to
  \PSL(2,\RR)$ a discrete embedding. Let $\phi: \PSL(2,\RR) \to
  \PSL(n,\RR) $ be the $n$-dimensional irreducible
  representation. Then $\phi\circ \iota$ is $B$-Anosov, where $B<
  \PSL(n,\RR)$ is the Borel subgroup. The connected
  component of $\hom(\Gamma, \PSL(n,\RR))/\PSL(n,\RR)$ containing $\phi\circ
  \iota$ is called the Hitchin component, it is known that every
  representation in the Hitchin component is  $B$-Anosov. More details
  are given in Section~\ref{sec:ex_hit}. 
  
By  a theorem of Choi and Goldman \cite{Goldman_Choi}, for $n=3$ representations in the Hitchin component are precisely the 
  holonomy representations of convex real projective structures on $\Sigma$.
  \item Quasi-Fuchsian groups in $\SO(2,n)$:  
In \cite{Barbot_Merigot_fusion} Barbot and M\'erigot introduced the
notion of quasi-Fuchsian representations $\rho: \G \to \SO(2,n)$ of a
uniform lattice $\Gamma < \SO(1,n)$, the basic example being the
injection of a lattice $\Gamma < \SO(1,n)$ composed with the natural
embedding $\SO(1,n) < \SO(2,n)$. They showed that  
quasi-Fuchsian representations are  precisely $Q_0$-Anosov, where
$Q_0<\SO(2,n)$ is the stabilizer of an isotropic line. In unpublished
work, Barbot shows furthermore that the entire 
  connected component
  in $\hom( \G, \SO(2,n))$ of the injection $\G \to \SO(1,n) \to \SO(2,n)$ consists of
  quasi-Fuchsian representations
  \cite{Barbot_component}. 
\end{enumerate}

\subsection{Holonomies of convex projective structures}
\label{sec:divi_conv}

A discrete group $\G <\SL(V)$ is said to \emph{divide} an open convex set
$\mathcal{C}$ in  $\PP(V)$ (i.e.\ the projectivization of a convex cone
in $V$) if $\G$ acts properly discontinuously on $\mathcal{C}$ with compact
quotient (see \cite{Benoist_Survey, Quint_Bourbaki} for surveys on this
subject). The cone $\mathcal{C}$ is said to be \emph{strictly convex} if 
$\partial \mathcal{C}$ intersects every projective line in at most two
points.

A discrete group $\G <\SL(V)$ dividing a convex set
$\mathcal{C}$, is hyperbolic if and only if $\mathcal{C}$ is strictly
convex \cite[Théorème~1.1]{Benoist_CD1}.
 In that case $\partial_\infty \Gamma$ is naturally homeomorphic to
$\partial \mathcal{C} \subset \PP(V) $. This identification gives an equivariant map $\xi^+:
\partial_\infty \G \to \PP(V)$. Since $\G$ also divides the dual cone
$\mathcal{C}^*$ in $\PP(V^*)$ (see \cite[Lemme~2.10]{Quint_Bourbaki})
one gets a second equivariant map $\xi^-:
\partial_\infty \G \to \PP(V^*)$. Strict convexity of $\mathcal{C}$ easily implies
that $(\xi^+, \xi^-)$ is compatible
(Definition~\ref{defi:compat}). Furthermore by \cite{Vey_convexes} the
action of 
$\G$ on $V$ is irreducible. Hence Proposition~\ref{prop:Comp_Maps}
applies and we have

\begin{prop}\label{prop:divconv}
Let $\G<\SL(V)$ be a discrete subgroup dividing a strictly convex set
$\mathcal{C} \subset \PP(V)$. Then the inclusion $\G \to \SL(V)$ is a
$Q_0$-Anosov representation, where $Q_0$ is the
stabilizer of a point in $\PP(V)$.
\end{prop}

The quotient $\G \backslash\mathcal{C}$ is an orbifold with a convex real projective structure. A representation $\rho: \G \to \SL(V)$ whose image divides a strictly convex set is thus the holonomy of a convex real projective structure.
Koszul \cite{Koszul} showed that for any finitely generated group
$\Gamma$, the set of such holonomy representations in $\hom(\Gamma,
\PGL(n,\RR))$ is open. Benoist \cite{Benoist_convex_III} showed that
this set is a connected component if and only if the virtual center of
$\Gamma$ is trivial.

Examples of convex real projective structures on manifolds whose fundamental groups are not isomorphic to a lattice in a Lie group have been constructed by Benoist \cite{Benoist_convex_IV, Benoist_quasi}. Kapovich \cite{Kapovich_Convex} provides several examples of convex
real projective structures on 
Gromov-Thurston manifolds, i.e.\ compact manifolds which carry a metric of
negative curvature pinched arbitrarily close to $-1$ but which do not
admit a metric of constant negative curvature.  
The corresponding holonomy
representations thus give examples of
Anosov representations of hyperbolic groups $\G$ into $\SL(V)$,  with
$\G$ not being isomorphic to a lattice in a Lie group. 

\subsection{Hitchin components}
\label{sec:ex_hit}
Let  $G$ be the split real form of an adjoint simple algebraic group, i.e.\ $G =
\PSL(n, \RR)$, $\PSO(n,n)$, $\PSO(n,n+1)$, $\PSp(2n, \RR)$ or a split real form of an 
exceptional group.
Any such group admits a principal three dimensional subgroup, that is 
an (up to conjugation) well defined homomorphism  $\phi_p : \PSL(2, \RR) \to G$ that
generalizes the $n$-dimensional irreducible representation $\PSL(2,
\RR) \to \PSL(n, \RR)$.   
The principal $\PSL(2, \RR)$ (or, more accurately, its Lie algebra) was discovered simultaneously by Dynkin and de
Siebenthal, later Kostant studied its connection with the
representation theory of $G$. The relevant results, as well as
references to Kostant's papers, are summarized in \cite[Sections~4 and
6]{Hitchin}.

Let $\G$ be the fundamental group of a closed connected oriented surface of genus $\geq 2$.
The Hitchin components 
 are 
the connected components of $\hom( 
\G, G)/G$ containing representation of the form $\phi_p \circ \iota$
where $\iota: \G \to \PSL(2, \RR)$ is a discrete embedding.

\begin{thm}
  \cite[Theorems~4.1, 4.2]{Labourie_anosov},
  \cite[Theorem~1.15]{Fock_Goncharov}
  Every representation $\rho$ in the Hitchin component is $(B^+,
  B^-)$-Anosov, where $(B^+, B^-)$ is a pair of opposite Borel
subgroups of $G$.
\end{thm}

\begin{remark}
  In fact, Fock and Goncharov in  \cite{Fock_Goncharov} provide a  continuous equivariant map
  $\xi: 
  \partial_\infty \G \to G/B^+$ that satisfies the transversality property
  of Definition~\ref{defi:transmap}. Thus Corollary~\ref{cor:Map_Ano}, together with an analysis of the potential Zariski closures of Hitchin representations, implies the above
  theorem. One can also use the positivity property of
  the equivariant curve, see \cite[Definition~1.10]{Fock_Goncharov}, to
  obtain directly the control on $A_+(m,t)$ required in
  Proposition~\ref{prop:mu_contrac}.
\end{remark}

\begin{remark}
When $G$ is a split real simple Lie group, which is not adjoint (e.g. $\SL(n,\RR)$ of $\Sp(2n,\RR)$), we call a connected component of $\hom(\G, G)/G$ a Hitchin component if and only it its image in $\hom(\G, G^{ad})G^{ad}$ is the Hitchin component, where $G^{ad}$ is the adjoint group of $G$. 
\end{remark}

\subsection{Maximal representations}
\label{sec:ex_max}
Let $G$ be a Lie group of Hermitian type, i.e.\ $G$ is connected,
semisimple with finite center and has no compact factors and the
symmetric space $\mathcal{H}=G/K$ admits a $G$-invariant complex
structure. 
Let $\G$ be the fundamental group of a closed connected oriented
surface of genus $\geq 2$. 
There is a characteristic number, often called the Toledo invariant, 
$\tau_G: \hom(\G, G) \to \ZZ$, which satisfies 
a Milnor-Wood type inequality
 \cite[Section~3]{Burger_Iozzi_Wienhard_toledo}: $| \tau_G(
\rho) | \leq (2g-2) c(G)$,  
where $c(G)$ is a constant depending only on $G$.

\begin{defi}
  A representation is said to be \emph{maximal} if $\tau_G( \rho) =
  (2g-2) c(G)$.
\end{defi}

Let $\check{S}=G/\check{P}$ be the Shilov boundary of $G$; it is the closed
$G$-orbit in the boundary of the bounded symmetric domain realization of $\mathcal{H}$.
 (see \cite[Chapter~4]{Wienhard_thesis} and
\cite{Satake_book}).

\begin{thm}\label{thm:maximal_Anosov}
  \cite{Burger_Iozzi_Wienhard_anosov}
  Any maximal representation $\rho: \Gamma \to G$ is $\check{P}$-Anosov.
\end{thm}

The symplectic group $\Sp(2n, \RR)$ is of Hermitian type. Its Shilov
boundary $\check{S}$ is the space $\mathcal{L}$ of Lagrangian (i.e.\
maximal isotropic) subspaces of $\RR^{2n}$. In that case the map
$\xi: \partial_\infty \G \to \mathcal{L}$ associated to a
maximal representation was constructed in
\cite{Burger_Iozzi_Labourie_Wienhard}; it satisfies the following
transversality condition: for all $t^+ \neq t^-$ in $\partial_\infty
\G$, $\xi( t^+) \oplus \xi(t^-) = \RR^{2n}$.

\begin{remark}\label{rem:redmaxtosymp}
In many cases Theorem~\ref{thm:maximal_Anosov} can be deduced from the
corresponding result for symplectic groups.

Indeed note that, up to passing to a finite index subgroup, the image
of a maximal representation
 $\rho: \pi_1( \Sigma) \to G$
 is contained in a Lie group $H$ of Hermitian
type, which is of tube type
\cite[Theorem~5]{Burger_Iozzi_Wienhard_toledo}.
Furthermore if $G$ is a
classical Lie group, then $H$ is also a classical
 group 
and therefore (up 
to taking finite covers) admits a tight embedding $\phi: H \to
\Sp(2n,\RR)$, 
which extends to an equivariant map of the Shilov boundary $\check{S}$
of $H$ into the space of Lagrangians $\mathcal{L}$, (this was already 
used in \cite{Wienhard_mapping}, see definitions and references therein, in particular \cite{Burger_Iozzi_Wienhard_tight} for the notion of tight embeddings). Up to passing to a finite index
subgroup, the composition $\phi\circ \rho: \pi_1( \Sigma) \to \Sp(2n,\RR)$ is
a maximal representation into the 
symplectic group. Thus, as a consequence of
Corollary~\ref{cor:finite_index} and 
Proposition~\ref{prop:embedding} one can deduce
Theorem~\ref{thm:maximal_Anosov} in the case when $G$ is a classical
Lie group from the case of the symplectic group.
\end{remark}

\subsection{Projective Schottky groups}\label{sec:schottky}
In \cite{Nori} Nori constructed Schottky groups $\G \subset \PGL(2n,\CC)$, which act properly discontinuously and cocompactly on an open subset $\Omega \subset \PP(\CC^{2n})$. 
These examples have been generalized by Seade and Verjovsky in \cite{Seade_Verjovsky_Schottky}.  Their construction also gives discrete free subgroups of $\PGL(2n,\RR)$. 
The embeddings $\rho: \G \to \PGL(2n,\KK)$, $\KK =\RR, \CC$, of these
projective Schottky groups are $P_n$-Anosov representations, where
$P_n$ is the stabilizer of an $n$-dimensional $\KK$-vector subspace in
$\KK^{2n}$ (see \cite{Guichard_Kapovich_Wienhard}). 
Such Schottky groups do not exist in $\PGL(2n+1, \KK)$, see \cite{Cano}.

%%%%%%%%%%%%%%%%%%PART TWO %%%%%%%%%%%%%%%%%%%%%%

\part{Domains of discontinuity}
\label{part:doma-disc}

\section{Automorphism groups of sesquilinear forms}
\label{sec:orth-sympl-groups}
In this section we construct domains of discontinuity for discrete
subgroups 
 of automorphism
groups of non-degenerate sesquilinear forms, which exhibit special dynamical
properties. 
 We then apply this construction  to Anosov
representations of hyperbolic groups. 

\subsection{Notation}

Let $(V,F)$ be a (right) $\KK$-vector space (with $\KK= \RR, \CC$ or $\HH$)
 with a non-degenerate form $F: V \otimes_\RR V \to \KK$,
linear in the second variable ($F(x,y \lambda) = F(x,y)
\lambda$)\footnote{the order in the equation matters only for the case
$\KK= \HH$.} and
such that:
\begin{itemize}
\item If $\KK=\RR$, $F$ is an indefinite symmetric 
  form or a skew-symmetric 
  form.
\item If $\KK = \CC$, $F$ is a symmetric form, a skew-symmetric form
  or an indefinite Hermitian 
  form.
\item If $\KK= \HH$, $F$ is an indefinite Hermitian form or a skew-Hermitian form.
\end{itemize}

Let $G_F < \GL(V)$ be the automorphism group of $(V,F)$. Then $G_F$ is $\O(p,q)$
($0<p\leq q$), $\Sp(2n, \RR)$; $\O(n, \CC)$, $\Sp(2n, \CC)$, 
$\U(p,q)$, ($0<p\leq q$);  $\Sp(p,q)$ ($0< p \leq q$) or $\SO^*(2n)$ respectively.

We denote by 
\[\mathcal{F}_0 = G_F / Q_0 = \PP( \{ x \in V \mid F(x,x)=0\}) \subset \PP(V)\]  
the space of isotropic lines and by  
\[\mathcal{F}_1 = G_F/Q_1 = \{ P \in \Gr_l(V) \mid F|_P = 0 \}\]  
the space of maximal isotropic
subspaces of $V$.
Here $\Gr_l(V)$ denotes the Grassmannian of $l$-planes with $l=p$ if $G_F
= \O(p,q)$ ($0<p \leq q$), $l=n$ if $G_F = \Sp(2n, \RR)$, etc.
When we explicitly want to refer to the vector space we will use the
notation $\mathcal{F}_0(V)$ and $\mathcal{F}_1(V)$.
In the case of $G_F$, transversality of points in $\mathcal{F}_i$ can be put concretely: 
\begin{lem}
A pair $(P_1, P_2) \in \mathcal{F}_i \times \mathcal{F}_i$ is
\emph{transverse} if and only if 
$P_1 + P_{2}^{\perp_F} = V$ 
(or equivalently $P_1 \cap P_{2}^{\perp_F} = \{ 0\}$, or $P_{1}^{\perp_F}
+P_2 = V$, etc.). 
\end{lem}

The closed $G$-orbit in $\mathcal{F}_0 \times
\mathcal{F}_1$ is
\[ \mathcal{F}_{01} = \{ (D,P) \in \mathcal{F}_0 \times
\mathcal{F}_1 \mid D \subset P \}. \]
There are two projections $\pi_i: \mathcal{F}_{01} \to
\mathcal{F}_i$, $i=0,1$.

For any subset $A \subset \mathcal{F}_i$ we define a subset $K_A$ in
$\mathcal{F}_{1-i}$ by
\begin{align*}
  K_A = \pi_{1-i} ( \pi_{i}^{-1} (A))
  &= \{ D \in \mathcal{F}_0 \mid \exists P \in A, D \subset P \}
  \textrm{ if } i=1,\\
  &= \{ P \in \mathcal{F}_1 \mid \exists D \in A, D \subset P \}
  \textrm{ if } i=0.
\end{align*}
If $A$ is closed, $K_A$ is closed.

\subsection{Subgroups with special dynamical properties }
We denote the Lie algebra of $G_F$ by $\mathfrak{g}$ and use the notation 
introduced in Section~\ref{sec:para_opp}. 
One can choose a basis of $V$ such that $\mathfrak{a}
\subset \mathfrak{g} \subset \mathfrak{gl}(V)$ is the
set of diagonal matrices $\diag(t_1, \dots, t_l, 0, \dots, 0, -t_l,
\dots, -t_1)$ with $t_i \in\RR$ for all $i$ (here $l= \rank_\RR G_F$), and such that $\mathfrak{a}^+$ are those
matrices satisfying $t_1 > t_2 > \cdots > t_l >0$. 
Let 
$\alpha_1$ denote the simple root of $\mathfrak{a}$ such that $Q_1=P_{\Delta
  \moins \{ \alpha_1\}}$ (see Section~\ref{sec:para_opp}). 
Then $\alpha_1$
is $\diag(t_1, \dots, t_l, \dots)\mapsto t_l$ (or $\diag(t_1,
\dots, t_l, \dots)\mapsto 2 t_l$ if there are no ``zeroes''). The
root $\alpha_0$ such that $Q_0 = P_{\Delta \moins \{ \alpha_0\}}$ is 
$\diag(t_1, \dots, t_l, \dots)\mapsto t_1 -t_2$.

\begin{defi}\label{defi:kind}
A discrete subgroup $\G< G_F $ is \emph{$\alpha_i$-divergent}, $i= 0,1$, if:
\begin{itemize}
\item any sequence $(g_n)_{n\in \NN}$ in $G$ diverging to infinity  has a subsequence $(g_{\phi(n)})_{n\in \NN}$ such that $\lim_{n\to \infty} \alpha_i(\mu(g_{\phi(n)})) = \infty$. 
\end{itemize}
\end{defi}

\begin{lem}\label{lem:divergent_proximal}
Let $\G < G_F$ be a discrete $\alpha_i$-divergent subgroup. Then $\G$ is proximal with respect to $\mathcal{F}_i$.
\end{lem}
\begin{proof}
This is a direct consequence of \cite[Section~3.2, Lemme]{Benoist}.
\end{proof}
\begin{thm}\label{thm:dod_prox}
Let $\Gamma < G_F$ be a discrete subgroup, and let $i$ be $ 0$ or $1$. Assume that $\Gamma$ is $\alpha_i$-divergent. 
Let $\mathcal{L}_\Gamma < \mathcal{F}_i$ denote the limit set of $\Gamma$. 
Set 
\[ 
\Omega_\Gamma = \mathcal{F}_{1-i}  \moins K_{\mathcal{L}_\Gamma}.\]
Then $\Omega_\G \subset  \mathcal{F}_{1-i} $ is  a $\G$-invariant open subset. Moreover, $\Gamma$ acts properly discontinuously on $\Omega_\Gamma$. 
\end{thm}
  
 \begin{proof}
We consider the case when $ \G < G_F$ is $\alpha_1$-divergent. 
The proof for the other case is entirely
analogous. We consider a point $x \in \mathcal{F}_1$ as a subspace
of $V$, in particular $\PP(x)\subset \mathcal{F}_0\subset \PP(V)$. 

Let $\mathcal{L}_\Gamma < \mathcal{F}_1$  be the limit set.  
We want to show the properness of the action of $\G$ on $\Omega_\G$, where 
\[\Omega_\G  = \mathcal{F}_0 \moins K_{\mathcal{L}_\Gamma}  = \Ff_0 \moins  \bigcup_{x \in \mathcal{L}_\Gamma} \PP(
x)\subset \PP(V).\]
We argue by contradiction. 

Suppose that there exist compact subsets $A$ and
$B$ of $\Omega_\G$ and a sequence $(\g_n)_{n \in \NN}$ in $\G$
such that: $\g_n \xrightarrow{n \to +\infty} \infty$, and for all $n$,
$\g_n A \cap B \neq \emptyset$. 

\begin{asparaenum}
\item By Theorem~\ref{thm:proxi_AMS} 
there is a finite set $S\subset\G$ such that, for all $n$, there is $s_n \in S$ such that $s_n \g_n$
is $(r,\epsilon)$-proximal relative to $\mathcal{F}_1$. Also
$s_n \g_n A \cap s_n B$ is nonempty. Hence, up to
replacing $(\g_n)$ by $(s_n \g_n)$ and $B$ by $\bigcup_{s \in S} s
B$, we can suppose that $\g_n$ is
$(r,\epsilon)$-proximal for all $n$.

\item Let $x^+_n, x^-_n \in \mathcal{L}_\Gamma$ be the attracting and repelling fixed points of $\g_n \in \G$. 
Up to extracting a subsequence we
can suppose that $x^\pm_n$ converge to $x^\pm\in \mathcal{L}_\Gamma$.
Since, for all $n$ the element $\g_n$ is
$(r,\epsilon)$-proximal we have, for all $n$, that $d(x^+_n,
V^-(x^-_n )) \geq r$, hence also $d(x^+, V^-(x^-)) \geq r$ (see
Section~\ref{sec:proximality} for notation). This shows that $x^+$ and $x^-$ are transverse.

\item Without loss of generality, we can assume
that $L=\stab( x^+, x^-)$  is the Levi component of $Q_1$, and that $x^+$ is the attracting fixed point of
 $\exp(a)$ when $a$ is in the Weyl chamber $\mathfrak{a}_L^+ \subset \mathfrak{a}$.
\item As $\lim_{n\to \infty} (x^{+}_{n}, x^{-}_{n})= 
(x^+,x^-)$ in $\mathcal{X} = G/L$,
there exists a sequence $(g_n)_{n\in \NN}$ in $G_F$ converging to $1$
and such that, for all
$n$, we have $(x^{+}_{n}, x^{-}_{n})= g_n
(x^+,x^-)$. Hence, for all $n$, $g_n 
\g_n g_{n}^{-1}$ fixes $(x^+,x^-)$ and hence belongs to $L$. 
We can thus write $g_n 
\g_n g_{n}^{-1}= k_n \exp(a_n) l_n$ with $a_n \in
\overline{\mathfrak{a}}_{L}^{+}$ and $k_n, l_n \in M$. Up to passing to a subsequence we
can assume that $(k_n)$ and $(l_n)$ converge to $k$ and $l$. 
Since $\G$ is $\alpha_1$-divergent we have that 
$(\alpha_1( a_n))$ tends to $+\infty$.

\item Now consider the set $\bigcup_{ n\in \NN} l_n g_{n}^{-1} A \cup l A$. It is a compact
subset of $\mathcal{F}_0 \moins \PP( x^-)$. Therefore there exists $\eta >0$ such that 
\[\bigcup_{ n\in \NN} l_n g_{n}^{-1} A \cup l A \subset B_\eta
:= \{ y \in \mathcal{F}_0 \mid d( y, \PP( x^-)) \geq \eta \}.\]
 
A simple calculation shows that, for all $\epsilon >0$,
there exists $R$ such that if $a \in \overline{ \mathfrak{a}_L}^+$
satisfies $\alpha_1( a ) \geq R$, then 
\[\exp(a) \cdot C \subset \{ y
\in \mathcal{ F}_0 \mid d(y, \PP(x^+)) \leq \epsilon \}.\] 
This implies 
that, for any sequence $(y_n)_{n\in \NN} $ in $A$, any accumulation point of
$(\exp(a_n) l_n g_{n}^{-1} y_n)$ is contained in $\PP(x^+)$. Since
$\lim_{n\to \infty} g_n =  1$ and since  $k = \lim_{n\to \infty} k_n $
stabilizes $\PP(x^+)$, also any  accumulation point of $\g_n
y_n = g_n k_n \exp(a_n) l_n g_{n}^{-1} y_n$ is contained in $\PP(x^+)$. 
\end{asparaenum}

Now we are ready to conclude. If $ \g_n A
\cap B \neq \emptyset$ for all $n$, then there exists an accumulation point of $\g_n
y_n$ which is contained in $B$. With the above this means in particular that $B \cap \PP(x^+)$  is nonempty. This contradicts the assumption that $B \subset \Omega_\G
\subset \mathcal{F}_0 \moins \PP( x^+)$.
 \end{proof}

  \begin{remarks}
  \noindent
\begin{asparaitem}
\item   In the special case when $G = \SO(2,n)$ Frances
  \cite{Frances_Lorentzian} constructed domains of discontinuity in
  $\mathcal{F}_0$ for discrete subgroups with special dynamical
  properties. In this case, $\alpha_1$-divergent groups in
  the sense of Definition~\ref{defi:kind} are called groups of the
  first type in Frances' paper, see
  \cite[Definition 4 and Proposition~6]{Frances_Lorentzian}. 
  
\item Benoist criterion for the properness of actions on homogeneous
  spaces \cite[Section~1.5, Proposition]{Benoist_properness} implies that a discrete subgroup $\G< \O(p,q)$ is $\alpha_1$-divergent if and only if $\G$ acts cocompactly on $\O(p,q)/\O(p-1,q)$; Proposition~6 in \cite{Frances_Lorentzian} is a special case of this. 
\end{asparaitem}
   \end{remarks}
  
  \subsection{Other Lie groups}\label{sec:first-reduction-step}
  
The following proposition allows to use the above construction of the domain of discontinuity to obtain domains of discontinuity for discrete subgroups of more general Lie groups.
\begin{prop}
  \label{prop:dod_in_smaller}
  Let $\G < G$ be a subgroup and $\phi: G \to G'$ be an injective homomorphism. 
  Suppose that $\mathcal{F}$ is a closed $G$-invariant
  subset of a $G'$-space $\mathcal{F}'$. Let $U' \subset \mathcal{F}'$
  be an open $\phi(\G)$-invariant subset such that $U= U' \cap \mathcal{F}$ is
  nonempty.
  \begin{enumerate}
  \item If $\G$ acts (via $\phi$) properly on $U'$ then $\G$ acts properly on $U$.
  \item If furthermore the quotient of $U'$ by $\G$ is compact, then the quotient of $U$ by $\G$ is also compact.
  \end{enumerate}
\end{prop}

In  Section~\ref{sec:descr-doma-disc} we will use Proposition~\ref{prop:dod_in_smaller}
 to reduce the discussion of a
general Anosov representation to the case of a $Q_0$-Anosov
representation into a symplectic group or an orthogonal
group. We will also use this proposition in the applications
discussed in Sections~\ref{sec:compactify} and  \ref{sec:cliffordklein}.

\section{Anosov representations  into orthogonal or symplectic groups} \label{sec:orth-sympl-groups_anosov}
Here we apply the constructions of Section~\ref{sec:orth-sympl-groups}  in
order to obtain cocompact domains of discontinuity for Anosov
representations. 
We first describe the structure of the domain of discontinuity in more
detail and deduce the properness of the action. Then we introduce some reduction steps, which allow us to
simplify the proof for the compactness of the quotient.

\subsection{Structure of the domain of discontinuity}
\label{sec:defin-doma-results}  
Recall that given a  $Q_i$-Anosov representation $\rho: \G \to  G_F$,
the image $\rho(\G) < G_F$ is a discrete subgroup which is
(AMS)-proximal relative to $\mathcal{F}_i$. In particular, its limit
set $\mathcal{L}_{\rho(\G)}< \mathcal{F}_i$ is well defined and equals the
image of the Anosov map associated to $\rho$,
$\mathcal{L}_{\rho(\G)}=\xi(\partial_\infty \G)$.

\begin{prop}\label{prop:Omega}
Let $\rho: \G \to G_F$ be a $Q_i$-Anosov representation with associated Anosov 
map $\xi: \partial_\infty \G \to \mathcal{F}_i$.  
Set 
\[ \Omega_\rho := \mathcal{F}_{1-i} \moins K_{\xi( \partial_\infty
  \G)}.\]
Then: 
\begin{enumerate}
\item $\Omega_\rho$ is an open $\rho(\G)$-invariant subset of $\Ff_{1-i}$. 
\item The map $\pi_{1-i}: \pi_{i}^{-1}( \xi( \partial_\infty
  \G) ) \to K_{ \xi ( \partial_\infty \G)}$ is a homeomorphism. In particular, 
  \[K_{ \xi ( \partial_\infty \G)} \cong \pi_{i}^{-1}(
  \xi( \partial_\infty \G) ) \xrightarrow{\pi_i} \xi( \partial_\infty
  \G) \cong \partial_\infty \G\] 
  is a locally trivial bundle over
  $\partial_\infty \G$ whose fiber over a point $t$ is $\PP( \xi(t))$ when $i=1$,  and $\mathcal{F}_1( P^{\perp_F} / P)$ when $i=0$.
 \end{enumerate} 
  \end{prop}
\begin{proof}
The first statement is obvious. 

For the second statement note that the map $\xi$ is injective, thus  $\partial_\infty \G \cong \xi( \partial_\infty \G)$. The transversality condition $P_1 \cap P_{2}^{\perp_F} = \{
  0\}$ implies $K_{P_1} \cap K_{P_2} = \emptyset$. Since $\xi$ is transverse this implies that $K_{\xi(\partial_\infty \G)}$
  is the disjoint union $\coprod_{t \in \partial_\infty \G}
  K_{\xi(t)}$.
  \end{proof} 

Recall that the cohomological dimension of a group $\G$ is the smallest $n$ such that
every cohomology group with coefficient in any $\G$-module vanishes in
degree $>n$. The virtual cohomological dimension $\vcd(\G)$ is the infimum of the
cohomological dimensions of finite index subgroups. Dimensions of
topological spaces in the following statements are also cohomological dimensions (for \v{C}ech
cohomology) or, what amounts to the same, covering dimensions.
In the next proposition,We will replace the dimension of
$\partial_\infty \Gamma$ by the virtual cohomological dimension of
$\Gamma$ thanks to the following result of Bestvina and
Mess. Formally, this replacement is not necessary in our proofs,
however it gives a hint why the quotient should be compact.

\begin{lem}\cite[Corollary~1.4]{Bestvina_Mess_Boundary}\label{lem:Bestvina_Mess}
Let $\G$ be a word hyperbolic group, then \[\dim \partial_\infty \G =
\vcd(\G)-1.\] 
\end{lem}

\begin{prop}\label{prop:dod_GF_codim}
 Let $\rho: \G \to G_F$ be a $Q_i$-Anosov representation and $i$ be
 $0$ or $1$. Let $\vcd(\G)$ be the virtual
  cohomological dimension of $\G$.  Set  $\delta = \dim
  \mathcal{F}_{1-i} - \dim K_{\xi(\partial_\infty \G)}$. Then
  \begin{enumerate}
  \item
    \begin{asparaitem}
    \item If $G_F= \O(p,q)$, $\U(p,q)$ or $\Sp(p,q)$ ($0< p\leq q$),
      then $\delta = q-\vcd(\G)$, $2q-\vcd(\G)$ or $4q-\vcd(\G)$ respectively.
    \item If $G_F=\O(2n, \CC)$ or $\O(2n-1, \CC)$ then $\delta = 2n
      -\vcd(\G)$.
    \item If $G_F=\Sp(2n, \RR)$ or $\Sp(2n, \CC)$ then $\delta =
      n+1-\vcd(\G)$ or $2n+2-\vcd(\G)$ respectively.
    \item If $G_F=\SO^*(2n)$ then $\delta =
      4n-2-\vcd(\G)$. 
    \end{asparaitem}
    \item If $\partial_\infty \G$ is a topological manifold and
      $\delta=0$, then $K_{\xi( \partial_\infty \G)} =
    \mathcal{F}_{1-i}$. In particular, in this case, $\Omega_\rho$ is empty.
\end{enumerate}
\end{prop}

\begin{proof}
 By Proposition~\ref{prop:Omega} the dimension of $K_{\xi( \partial_\infty \G)}$ is:
\begin{align*}
  \dim K_{\xi(\partial_\infty \G)} &= \dim \pi_{i}^{-1} (
  \xi( \partial_\infty \G)) = \dim \pi_{i}^{-1}(\{P\}) + \dim
  \xi( \partial_\infty \G) \\
  &= \dim \pi_{i}^{-1}(\{P\}) + \dim \partial_\infty \G,
\end{align*}
and $\pi_{i}^{-1}(\{P\}) \cong \mathcal{F}_1( P^{\perp_F}/P)$ if $i=0$
and $\pi_{i}^{-1}(\{P\}) \cong \PP(P)$ if $i=1$. 
With Lemma~\ref{lem:Bestvina_Mess} we thus have 
\begin{align*}
  \delta &= \dim \mathcal{F}_1 - \dim \mathcal{F}_1( P^{\perp_F}/P) -\vcd(\G)
  +1, \text{ if } i=0,\\
  &= \dim \mathcal{F}_0 - \dim \PP(P) -\vcd(\G)+1, \text{ if } i=1.
\end{align*}
The first statement follows now by
calculating the dimensions
of the homogeneous spaces $\mathcal{F}_0$ and $\mathcal{F}_1$.

When $\partial_\infty \G$ is a topological manifold, then 
$K_{\xi( \partial_\infty \G)}$ is also a manifold (by
Proposition~\ref{prop:Omega}). If $\delta=0$ this implies 
that $K_{\xi( \partial_\infty \G)}$ is an open submanifold of
$\mathcal{F}_{1-i}$, since it is also closed, the equality
$K_{\xi( \partial_\infty \G)} =\mathcal{F}_{1-i}$ follows.
\end{proof}

\begin{remark}
The coincidence for the values of $\delta$ for $Q_0$ and $Q_1$-Anosov
representations is explained by the following 
observation: if $p_1: M \to M_1$ and $p_2: M \to M_2$ are two
submersions such that $p_1 \times p_2 : M \to M_1 \times M_2$ is an
immersion, then, for any $m_1\in M_1$ and $m_2\in M_2$, the
codimensions of $p_1( p_{2}^{-1}( m_2))$ in $M_1$ and of $p_2(
p_{1}^{-1}( m_1))$ in $M_2$ are equal to $\dim M_1 + \dim M_2 - \dim M$.
\end{remark}

\begin{remark}
\label{rem:examples_orth}
The control on the codimension of $K_{\xi(\partial_\infty \G)}$ given by
Proposition~\ref{prop:dod_GF_codim} allows to deduce nonemptyness for
many examples of Anosov representations. Here we give some examples
where the domain of discontinuity is empty; we come back to these examples in Section~\ref{sec:cliffordklein}. 
\begin{enumerate}
\item Let  $\iota:\G \hookrightarrow \SO(1,q)$ be a convex cocompact representation, i.e.\ $\iota$ is 
  $Q_0$-Anosov. 
  The composition of $\iota$ with embedding $\phi: \SO(1,q) \to
  \SO(1+p', q+q')$ gives a $Q_0$-Anosov representation $\phi \circ
  \iota$ (see Lemma~\ref{lem:red_spe}.(\ref{item1:lemredspe})). 
  When $i(\G)$ is a cocompact lattice in $\SO(1,q)$ and $\phi: \SO(1,q) \to \SO(p, q)$, $p
  \leq q$, the equality case of Proposition~\ref{prop:dod_GF_codim} is
  attained and  $\Omega_{ \phi \circ \iota} = \emptyset$.
\item 
 Let $\iota: \G \to G$ be a convex cocompact representation into  $\SU(1,n)$, $\Sp(1,n)$ or $G_\mathcal{O}$ (the isometry group of the
Cayley hyperbolic plane). If $G= \SU(1,n)$ consider the natural injection $\phi: \SU(1,n) \to
  \SO(2,2n)$ be the natural injection; if $G= \Sp(1,n)$ consider $\phi: \Sp(1,n)  \to \SO(4,4n)$; if $G$ is $G_\mathcal{O}$ consider $\phi: G \to \SO(8,8)$. In
  any case $\phi \circ \iota$ is $Q_1$-Anosov and $\Omega_{\phi \circ
  \iota}$ is empty when $\iota(\G)$ is a cocompact lattice. 
\end{enumerate}
\end{remark}

\begin{thm}\label{thm:dod_GF}
  Let $\rho: \G \to G_F$ be a $Q_i$-Anosov representation. 
  If $\Omega_\rho\subset \Ff_{1-i}(V)$ is nonempty (e.g.\ if  $\delta> 0$), then 
  \begin{enumerate}
  \item\label{dod_GF_i} $\G$ acts properly discontinuously on $\Omega_\rho$
  \item\label{dod_GF_ii} The quotient $\G\backslash \Omega_\rho$ is compact.
\end{enumerate}
 \end{thm}

The proof of statement \eqref{dod_GF_i} is a direct consequence of
Theorem~\ref{thm:dod_prox}, Theorem~\ref{thm:Ano_AMS}, and the following lemma. 

\begin{lem}\label{lem:anosov_kind}
Let $\rho: \G \to G_F$ be a $Q_i$-Anosov representation, then $\rho(\G)$ is $\alpha_i$-divergent. 
\end{lem}
\begin{proof}
 This is a direct application of Proposition~\ref{prop:mu_contrac}, which gives control on the contraction rate  in terms of the $L$-Cartan projections. 
\end{proof}

 The proof of statement \eqref{dod_GF_ii} is deferred to
Section~\ref{sec:compacity}.

\subsection{Reduction steps}\label{sec:reduction}\label{sec:spec-reduct-steps}
Before we turn to the proof of compactness of the quotient
$\G \backslash \Omega_\rho$ we introduce some reduction steps, which will
allow us to restrict our attention mainly to $Q_i$-Anosov
representations into $\O(p,q)$.

\begin{lem}
  \label{lem:red_spe}
  Let $(V,F)$ be a vector space with a sesquilinear form that satisfies 
  the conditions of Section~\ref{sec:defin-doma-results}.
  Let $\rho: \G
  \to G_F$ be a representation.
  \begin{enumerate}
  \item\label{item1:lemredspe} Let $(V',F')$ be of the same type as $(V,F)$. Then the
    injection $V\to V \oplus V'$ induces a  homomorphism $\phi: G_F \to
    G_{F+F'}$. If $\rho$ is $Q_0$-Anosov then $\phi \circ \rho$ is
    $Q_0$-Anosov.
    
    Let $(V,F)$ be a complex orthogonal space of dimension $2n-1$, and
    $(V', F')$ a complex orthogonal space of dimension one, then if
    $\rho$ is $Q_1$-Anosov, then $\phi \circ \rho$ is
    $Q_1$-Anosov.

  \item\label{item2:lemredspe} Let $k$ a positive integer. Consider the vector space $V^k$, which is endowed
    with the form $F^k$, and let $\phi: G_F
    \to G_{F^k}$ be the diagonal embedding. If $\rho$ is $Q_1$-Anosov then $\phi \circ \rho$ is
    $Q_1$-Anosov.
  \item\label{item3:lemredspe} Suppose $F$ is skew-symmetric (i.e.\ $G_F = \Sp(2n,\RR)$ or
    $\Sp(2n,\CC)$) then the form $F\otimes F$ on $V\otimes V$ is
    symmetric. Let $\phi: G_F \to G_{F\otimes F}$ the corresponding
    homomorphism. If $\rho$ is $Q_i$-Anosov then $\phi \circ \rho$ is
    $Q_i$-Anosov. 
      \item\label{item4:lemredspe} If $F$ is Hermitian, $\Re F$ is a non-degenerate 
    symmetric 
    bilinear form on $V_\RR$ the real space underlying
    $V$. Hence there is a homomorphism $\phi: G_F \to G_{\Re F}$. If
    $\rho$ is $Q_1$-Anosov then $\phi \circ \rho$ is $Q_1$-Anosov.
  \end{enumerate}
  
  In all of the above cases,  if $\rho$ is $Q_i$-Anosov, one has 
  \[\Omega_\rho =
  \mathcal{F}_{1-i}(V) \cap \Omega_{ \phi \circ \rho}.\]
\end{lem}

\begin{proof}
  The cases \eqref{item1:lemredspe} through \eqref{item4:lemredspe} follow 
  from the general proposition~\ref{prop:injLieGr}.  In order to illustrate the ideas we give a direct proof 
  for \eqref{item3:lemredspe} and $\rho$ being a $Q_0$-Anosov representation. Consider the $\phi$-equivariant maps 
  \[\phi_0: \mathcal{F}_0(V) \to \mathcal{F}_0(V \otimes V), \ D \mapsto D \otimes D\] 
  and 
  \[\phi_1: \mathcal{F}_1(V) \to
  \mathcal{F}_1(V \otimes V), \ L \mapsto L\otimes V.\] 
 
 If $D, D'\in \mathcal{F}_0(V)$ are transverse then $\phi_0(D)$ and
  $\phi_0( D')$ are also transverse: indeed $D^\perp \oplus D' = V$
  implies   $V\otimes V = (D
  \otimes D)^\perp \oplus D' \otimes D'$. 
  So if $\xi: \partial_\infty \G \to
  \mathcal{F}_0(V)$ is transverse,  then
  $\phi_0 \circ \xi$ is transverse (see Definition~\ref{defi:transmap}). Concerning the contraction
  property required for an Anosov representation (see
  Definition~\ref{defi:ARhyp}) it is
  enough to remark that if a diagonal element $g$ in $G_F$ fixes $D$ in
  $\mathcal{F}_0(V)$ and contracts $T_D \mathcal{F}_0(V)$ then $\phi(g)$
  contracts $T_{\phi_0(D)} \mathcal{F}_0(V \otimes V)$. 

  The equality $\Omega_\rho = \mathcal{F}_1(V) \cap \Omega_{\phi \circ
    \rho}$ results from the fact that, for $D\in \mathcal{F}_0(V)$ and
  $L\in \mathcal{F}_1(V)$, $D \subset L$ if and only if $D\otimes D
  \subset L \otimes V$.
\end{proof}

\begin{remarks}
\noindent
\begin{asparaenum}[(a)]
\item Due to this lemma the proof of Proposition~\ref{prop:dod_GF_codim} reduces in many cases 
to the case of orthogonal groups $\O(p,q)$. 
However, this reduction does not seem to work for $Q_0$-Anosov representations into $G_F(V)= \O(n, \CC)$,
$\U(p,q)$ or $\Sp(p,q)$. Indeed in these cases one would consider the
orthogonal space $W = \bigwedge^{2}_{\RR} V$ (or $\bigwedge^{4}_{\RR} V$)
and the corresponding embedding $\phi: G_F \to G_F(W)$ sends
$Q_0$-Anosov representations to $Q_0$-Anosov representations. Yet
there is no corresponding embedding of $\mathcal{F}_1(V)$ into
$\mathcal{F}_1(W)$, so that the final conclusion of the lemma does not
hold.

\item Cases \eqref{item1:lemredspe} and \eqref{item2:lemredspe} of the above
lemma together with the formula for the codimension given in
Theorem~\ref{thm:dod_GF} will allow us to assume that the open
$\Omega_\rho$ has high connectedness properties when proving the
compactness of $\Gamma \backslash \Omega_\rho$.
\end{asparaenum}
\end{remarks}

\subsection{Compactness}
\label{sec:compacity}

In this section we prove the compactness of $\G \backslash \Omega_\rho$, claimed in 
Theorem~\ref{thm:dod_GF}.\eqref{dod_GF_ii}.

For clarity of the exposition we will suppose in the following that $G_F =
\O(p,q)$. Applying Lemma~\ref{lem:red_spe} and Proposition~\ref{prop:dod_in_smaller} this proves Theorem~\ref{thm:dod_GF} in all cases
except for $Q_0$-Anosov representations into $\O(n,\CC)$, $\U(p,q)$ or
$\Sp(p,q)$. In the remaining cases the proof is a straightforward adaptation of the arguments presented here. 

\subsubsection{Homological formulation}
\label{sec:some-preparation-1}

Let $\rho: \G \to G_F =
\O(p,q)$ be a $Q_i$-Anosov representation. 
In view of Lemma~\ref{lem:red_spe} we can assume without loss of
generality that  $\min(p-2,q-p)> \max(m+1, \vcd(\G) )$,
for some $m\in \NN$ which will be fixed later. 

 Furthermore, up to passing to a finite index subgroup, we can assume
 that $\G$ is torsion-free, and that $\rho$ is injective.

To simplify some of the cohomological arguments we consider the following 
$2$-fold covers of $\mathcal{F}_0$ and $\mathcal{F}_1$:
\begin{align*}
  \mathcal{F}_{0}^{or} &= \{D \in \Sphr(V) \mid D \text{ oriented line
    with } F|_D = 0 \} \cong \Sphr^{p-1}\times \Sphr^{q-1}\\
  \mathcal{F}_{1}^{or} &= \{P \in \Gr_{p}^{or}(V) \mid P \text{
    oriented }p\text{-plane with } F|_P = 0 \} \cong \SO(q)/ \SO(q-p).
\end{align*}

The homogeneous spaces $\mathcal{F}_{i}^{or}$ are
$\min(p-2,q-p)$-connected. The lift
$\Omega^{or}_{\rho}$ of $\Omega_\rho$ to $\mathcal{F}^{or}_{1-i}$ is a
$2$-fold cover of $\Omega_\rho$ and its complement
$K^{or}_{\xi (\partial_\infty \G)}$ fibers over $\partial_\infty \G$
with fibers isomorphic to $\Sphr^{p-1}$ if $i=0$, and to
$\SO(q-1)/\SO(q-p)$ if $i=1$; in both case the fibers are also 
$\min(p-2,q-p)$-connected.

Up to passing to a finite index subgroup we
suppose that $\G$ preserves the orientation on $\Omega^{or}_{\rho}$
hence $\G \backslash \Omega^{or}_{\rho}$ is an oriented manifold.

For the rest of this section we  denote 
$\Omega^{or}_{\rho}$, $K^{or}_{\xi (\partial_\infty
  \G)}$ and $\mathcal{F}_{1-i}^{or}$ by $\Omega$, $K$ and $\mathcal{F}$ respectively. 

Let $l= \dim \Omega =
\dim \mathcal{F}$. By our assumption $q-p>\vcd(\Gamma)$, and hence  $\delta= \mathrm{codim} K =
q-\vcd(\G) \geq 2$, and  $\Omega$ is connected. 
Therefore $\G \backslash \Omega$ is a $l$-dimensional connected
oriented manifold. Thus $\G \backslash \Omega$ is compact if and only if 
 $H^{0}_{c}( \G\backslash \Omega)$ (cohomology with compact
support and coefficients in $\RR$) is nonzero.
By Poincar\'e
duality this is equivalent to the top-dimensional homology group $H_l( \G \backslash \Omega) =
H^{0}_{c}( \G\backslash \Omega)^*$ being nonzero. 

Therefore Theorem~\ref{thm:dod_GF}.\eqref{dod_GF_ii} follows from the following 
\begin{prop}\label{prop:hom_nonzero}
With the notation above 
\[H_l( \G \backslash \Omega)\cong \RR.\]
\end{prop}

We will now first prove Proposition~\ref{prop:hom_nonzero} in the case when $\G$ is the  fundamental
group of a negatively curved closed manifold, then we consider the case when $\G$ is an arbitrary word hyperbolic group. 

\subsubsection{Fundamental groups}
\label{sec:fundamental-groups}
Let $\G = \pi_1(N)$ be the fundamental
group of a negatively curved closed manifold $N$ of dimension $m$. 

The fibration $\G\backslash (\tilde{N}\times \Omega) \to \G \backslash
\Omega$ has contractible fibers (isomorphic to $\tilde{N}$), hence
induces an isomorphism in homology 
\[H_l( \G \backslash \Omega) = H_l(
\G \backslash (\tilde{N} \times \Omega)).\] 
Applying
Poincar\'e duality to this $(l+m)$-dimensional manifold gives 
\[H_l(
\G \backslash (\tilde{N} \times \Omega))^* \cong H^{m}_{c}(
\G \backslash (\tilde{N} \times \Omega)).\] By definition, this last
group is the direct limit
\[ H^{m}_{c}(
\G \backslash (\tilde{N} \times \Omega)) = \varinjlim_{\mathcal{O}
  \supset \G \backslash (\tilde{N} \times K) } H^{m}(
\G \backslash (\tilde{N} \times \mathcal{F}), \mathcal{O}),\]
where the limit is taken over the open neighborhoods of $\G \backslash
(\tilde{N} \times K)$ in $\G \backslash (\tilde{N} \times
\mathcal{F})$. The long exact sequence of the pair $( 
\G \backslash (\tilde{N} \times \mathcal{F}), \mathcal{O})$ reads as 
\begin{multline*}
 H^m( \mathcal{O}) \longto H^m( \G \backslash (\tilde{N} \times
\mathcal{F})) \longto H^m( \G \backslash (\tilde{N} \times
\mathcal{F}), \mathcal{O}) \longto\\ H^{m+1}( \mathcal{O}) \longto H^{m+1}( \G
\backslash (\tilde{N} \times \mathcal{F})).
\end{multline*}
Passing to the limit one gets 
\begin{multline*}
   \check{H}^m( \G \backslash (\tilde{N} \times K) ) \longto H^m( \G
\backslash (\tilde{N} \times
\mathcal{F})) \longto H^{m}_{c}(
\G \backslash (\tilde{N} \times \Omega)) \longto \\
\check{H}^{m+1}(\G \backslash
(\tilde{N} \times K)) \longto H^{m+1}( \G
\backslash (\tilde{N} \times \mathcal{F})),
\end{multline*}
where $\check{H}^*( \G \backslash (\tilde{N} \times K) )$ is the \v{C}ech
cohomology of $\G \backslash (\tilde{N} \times K)$. Furthermore 
 the fibrations $\G \backslash (\tilde{N}
\times K) \to \G \backslash (\tilde{N} \times \partial_\infty \G)$ and
$\G \backslash (\tilde{N} \times \mathcal{F}) \to \G \backslash
\tilde{N}$ induce isomorphisms in cohomology up to degree $m+1$, since the 
fibers are $\min(p-2, q-p)$-connected and  
$\min(p-2,
q-p)> m+1$. Therefore the last long exact sequence reads as:
\begin{multline*}
   \check{H}^m( \G \backslash (\tilde{N} \times \partial_\infty \G) ) \longto H^m( \G
\backslash \tilde{N} ) \longto H^{m}_{c}(
\G \backslash (\tilde{N} \times \Omega)) \longto \\
\check{H}^{m+1}(\G \backslash
(\tilde{N} \times \partial_\infty \G)) \longto H^{m+1}( \G
\backslash \tilde{N}).
\end{multline*}

When we replace $\Omega$, $K$ and $\mathcal{F}$  by
$\tilde{N}$, $\partial_\infty \tilde{N} \cong \partial_\infty \G$ and
$\overline{ \tilde{N}} = \tilde{N} \cup \partial_\infty \tilde{N}$ respectively, the same argument 
leads to the long exact sequence 
\begin{multline*}
   \check{H}^m( \G \backslash (\tilde{N} \times \partial_\infty \G) ) \longto H^m( \G
\backslash \tilde{N} ) \longto H^{m}_{c}(
\G \backslash (\tilde{N} \times \tilde{N})) \longto \\
\check{H}^{m+1}(\G \backslash
(\tilde{N} \times \partial_\infty \G)) \longto H^{m+1}( \G
\backslash \tilde{N}).
\end{multline*}
Comparing these two exact sequences shows that 
\[H^{m}_{c}(
\G \backslash (\tilde{N} \times \Omega)) \cong H^{m}_{c}(
\G \backslash (\tilde{N} \times \tilde{N})).\] 
Poincar\'e duality implies $H^{m}_{c}(
\G \backslash (\tilde{N} \times \tilde{N}))^* \cong H_{m}(
\G \backslash (\tilde{N} \times \tilde{N}))$; the fibers $\G
\backslash (\tilde{N} \times \tilde{N}) \to \G \backslash \tilde{N}$
being contractible, one has $H_{m}(
\G \backslash (\tilde{N} \times \tilde{N})) \cong H_{m}(
\G \backslash \tilde{N})$.
Recapitulating, we obtain the following chain of isomorphisms:
\begin{multline}
  \label{eq:iso}
H_l( \G \backslash \Omega) \cong H_l(
\G \backslash (\tilde{N} \times \Omega))
\cong H^{m}_{c}(
\G \backslash (\tilde{N} \times \Omega))^* \cong \\
 H^{m}_{c}(
\G \backslash (\tilde{N} \times \tilde{N}))^*  \cong H_{m}(
\G \backslash (\tilde{N} \times \tilde{N})) \cong H_{m}(
\G \backslash \tilde{N}) \cong H_{m}(N) \cong  \RR.
\end{multline}

\subsubsection{Hyperbolic groups}
\label{sec:replacement-manifold}
The previous proof is deeply based on the fact that we are calculating
(co)homology groups of manifolds. In order to adapt the proof to the case of 
a general finitely generated word hyperbolic group $\G$ we need a replacement for $N$
and $\tilde{N}$.

Let $R_d( \G)$ be a Rips complex for $\G$. This is the simplicial
complex whose $k$-simplices are given by $(k+1)$-tuples $(\g_0,\dots, \g_k)$
of $\G$ satisfying $d_\G( \g_i, \g_j) \leq d$ for all $i,j$. For $d$
big enough $R_d( \G)$ is contractible \cite[Chapitre~5]{Coornaert_Delzant_Papadopoulos}. Let $\tilde{R}$ denote such a contractible Rips complex $R_d(\G)$ and let $R= \G \backslash \tilde{R}$ be its
quotient. Then $R$ is a finite simplicial complex and as such
admits an embedding $R \hookrightarrow \RR^m$ into  Euclidean space
\cite[Corollary~A.10]{Hatcher_AT} \cite[II.9]{Eilenberg_Steenrod}. A
small (and regular) neighborhood of $R$ in $\RR^m$ gives  an $m$-dimensional manifold with boundary 
$(U, \partial U)$ such that
$R$ is a retract of $U$. 

In particular $\tilde{R}$ is a retract of $\tilde{U}$ which is hence a
contractible manifold. Note also that $\partial_\infty \tilde{U}
\cong \partial_\infty \tilde{R} \cong \partial_\infty \G$ and that
$\tilde{U} \cup \partial_\infty \G$ retracts to $\tilde{R}
\cup \partial_\infty \G$ which is a contractible space
(\cite[Theorem~1.2]{Bestvina_Mess_Boundary}),
therefore $\tilde{U} \cup \partial_\infty \G$ is also contractible.

The same argument as in Section~\ref{sec:fundamental-groups} (working with manifolds with boundary)  gives the
following sequences of isomorphisms, with $m = \dim U$ and $l = \dim \Omega$,
\begin{multline}
  \label{eq:2}
  H_l(\G \backslash \Omega) \cong H_l( \G \backslash( \Omega \times
  \tilde{U})) \cong H^{m}_{c}( \G \backslash( \Omega \times
  \tilde{U}))^* \cong H^{m}_{c}( \G \backslash( \mathring{
    \tilde{U}} \times \tilde{U}))^* \cong\\
 H_m(\G \backslash( \mathring{
    \tilde{U}} \times \tilde{U}), \G \backslash( \mathring{
    \tilde{U}} \times \partial \tilde{U})) \cong
 H_m(U, \partial U) \cong \RR,
\end{multline}
where the last two isomorphisms are given by Poincar\'e duality for manifolds 
with boundary and by considering the fibration $\G \backslash( \mathring{
  \tilde{U}} \times \tilde{U}) \to U$ of manifolds with boundary
(the fibers are isomorphic to $ \mathring{
  \tilde{U}}$, thus contractible). 
  
\section{General groups}
\label{sec:descr-doma-disc}
We now turn to the case of Anosov representations into general semisimple Lie groups. 
We explicitly construct an open subset of $G/AN$, on which $\G$ acts properly discontinuously  and with compact quotient. The set we construct will depend on the choice of an 
 irreducible $G$-module
$(V,F)$, 
hence a 
homomorphism $\phi: G \to G_F$, 
such that the composition 
 $\phi\circ \rho: \G \to G_F$ is $Q_0$-Anosov. 
 We consider then the domain of discontinuity $\Omega_{\phi\circ \rho} \subset \Ff_1(V)$ (constructed in Section~\ref{sec:defin-doma-results}) and apply Proposition~\ref{prop:dod_in_smaller} in order to obtain a  domain of discontinuity in $G/AN$ on which $\G$ acts with compact quotient.
The dependence of the domain of discontinuity  on the choice of $G$-module $(V,F)$ is illustrated by several examples in Section~\ref{sec:desclim}.

We give an explicit description of the domain of discontinuity in terms of the Bruhat decomposition of $G$. This allows us to describe sufficient conditions for the domain of discontinuity to be nonempty. In the case when $\G$ is a free group or the fundamental group of a closed surface of genus $\geq 2$ this leads to the statements of Theorem~\ref{thm_intro:free} and Theorem~\ref{thm_intro:surface}.

\subsection{$G$-Modules}
\label{sec:g-module}
We use  the notation introduced in
Section~\ref{sec:parab-subgr-lie}.

Let  $\Theta \subset \Delta$ with $\iota(\Theta) = \Theta$, and 
$P = P^+_{\Theta} < G$ the corresponding parabolic subgroup.  Let $(V,F)$ be an irreducible $G$-module,
where $F$ is a non-degenerate bilinear form on $V$, (indefinite)
symmetric or skew-symmetric,  and such that there is an $F$-isotropic
line $D\subset V$ with $ P = \stab_G(D)$. In this section, we use
standard theory for
decomposition of a $G$-module $V$ into weight spaces $V_\chi$.
The reader who is not familiar with basic representation theory is referred to \cite{Fulton_Harris} and \cite[Chapter~IV]{Guivarch_Ji_Taylor}.

The $G$-module $V$ decomposes under the action of $\mathfrak{a}$ into weight spaces:
\[ V = \bigoplus_{ \mu \in \mathcal{C}} V_\mu, \ V_\mu = \{ v
\mid \forall a \in \mathfrak{a},\, a \cdot v = \mu(a)v \}, \ \mathcal{C} = \{ \mu \in
\mathfrak{a}^* \mid V_\mu \neq \{ 0 \}\}.\]
The set $\mathcal{C}$ is a finite subset of $\mathfrak{a}^*$ that may
contain $0$; it is invariant by the action of the Weyl group $W$; it 
is also invariant by $\mu \mapsto -\mu$ since $V$ is isomorphic to
$V^*$. 

We set  $V_+ = \bigoplus_{ \mu > 0} V_\mu$ and $V_- = \bigoplus_{
  \mu < 0} V_\mu$.

Let $\lambda$ be the highest weight of $V$, i.e.\ the highest element
of $\mathcal{C}$ with respect to the order $<$ on $\mathfrak{a}^*$. 
Then $V_\lambda$ is in the kernel of every
element $n\in \mathfrak{n}^+$ and is in fact equal to the intersection
of the kernels $\ker(n)$, $n\in \mathfrak{n}^+$,
(this follows from the fact that, for $n\in \mathfrak{g}_\alpha$ and
for $v \in V_\mu$, $n\cdot v \in V_{\mu +
  \alpha}$ and is nonzero when $\mu,\mu+\alpha\in \mathcal{C}$ and
$v\neq 0$). 

\begin{lem}\label{lem:exist_iso}
\noindent
\begin{enumerate}
\item The highest weight space $V_\lambda$ is one-dimensional, $V_\lambda = D$.
\item The spaces $V_+$ and $V_-$ are $F$-isotropic subspaces of $V$.
\item    There exists a maximal $F$-isotropic $AN$-invariant
  subspace $T$ containing $V_+$. For any such $T$, $T \cap V_- = \{ 0
  \}$.
\end{enumerate}
\end{lem}

\begin{proof}
  The line $D$ is $\mathfrak{p}^{+}_{\Theta}$-invariant, hence  $D\subset V_\lambda$. 
The  $\mathfrak{n}^-$-module generated by $D$ is
  $\mathfrak{g}$-invariant. By irreducibility of $V$ this implies $D=
  V_\lambda$, in particular $\dim V_\lambda =1$.

  The map given by the bilinear form 
  $F: V \otimes V \to \RR$ is a morphism of $G$-modules, where $\RR$
  is the trivial $G$-module.  Hence, for any $\mu$ in
  $\mathfrak{a}^*$, $F( (V \otimes V)_\mu) \subset \RR_\mu$ and
  $\RR_\mu = \{ 0 \}$ unless $\mu =0$. It follows that $F( V_\mu
  \otimes V_{\mu'}) = \{ 0\}$ whenever $\mu + \mu' \neq 0$. Therefore
  the decomposition
  \[ V = V_0 \oplus \bigoplus_{\mu>0} (V_\mu \oplus V_{-\mu})\]
  is $F$-orthogonal and $V_+$ and $V_-$ are isotropic. This proves
  that the restriction of $F$ to $V_0$ and to $V_+ \oplus V_-$ is
  non-degenerate and that $(V_+)^\perp = V_+ \oplus V_0$.

  As a consequence any isotropic space containing $V_+$ is contained
  in $V_+ \oplus V_0$ and henceforth intersects $V_-$ trivially.

  Note now that $V_+$ and $V_0$ are $\mathfrak{a}$-invariant and that
  $\mathfrak{n}^+ \cdot V_0 \subset V_+$ and that $\mathfrak{a}$ acts
  trivially on $V_0$. Thus for any maximal isotropic space $W$ of $(V_0,
  F|_{V_0})$, the space $T= V_+ \oplus W$ is maximal isotropic in $V$
  and is $(\mathfrak{a} + \mathfrak{n}^+)$-invariant.
\end{proof}

\begin{remark}
  In the announcement \cite{Guichard_Wienhard_CRAS} we claimed that
  $V_+$ is itself maximal isotropic, this is false.  A counter-example
  is the $\O(p,p+k)$-module $V = \wdg{2} \RR^{2p+k}$, where the
  restriction of $F$ to $V_0$ has signature $(k(k-1)/2, p)$. 

We do not know if it is always possible to choose some $G$-module $V$ that has a $B$-invariant maximal isotropic subspace. 
\end{remark}

\begin{cor}
Let $(V,F)$ be an irreducible $G$-module as in Section~\ref{sec:g-module} and let $\phi: G \to G_F = G(V,F)$ be the corresponding homomorphism. It induces the following $\phi$-equivariant maps
\[\begin{array}{rlcrl}
  \phi_0:  G/P&  \longto \mathcal{F}_0(V) & {}\quad{} & \phi_1: G/ AN &
  \longto \mathcal{F}_1(V)\\
  gP & \longmapsto \phi(g) V_\lambda & & gH & \longmapsto \phi(g)T,
\end{array}
\]
where $T\subset V$ is a maximal $F$-isotropic subspace, whose existence is  guaranteed by  Lemma~\ref{lem:exist_iso}. 
\end{cor}

\subsection{Domains of discontinuity in $G/AN$}
\label{sec:doma-disc-gh}
Let now $\rho: \G \to G$ be a $P$-Anosov representation, and $\phi: G \to G_F$ the homomorphism considered in Section~\ref{sec:g-module}. Then $\phi \circ \rho: \G \to G_F$ is $Q_0$-Anosov, and hence admits a cocompact domain of discontinuity $\Omega_{\phi \circ \rho}\subset  \mathcal{F}_1(V)$ (see 
Section~\ref{sec:defin-doma-results}).

Applying Proposition~\ref{prop:dod_in_smaller} to the $G$-orbit 
$\phi_1(G/AN) \subset \mathcal{F}_1(V)$ we obtain 

\begin{thm}\label{thm:dod_general_case}
Let $G$ be a semisimple Lie group and $P<G$ a proper parabolic
subgroup. Let $\rho: \G \to G$ be a $P$-Anosov representation. 
Let $\phi: G \to G_F(V)$ such that $\phi \circ \rho$ is
$Q_0$-Anosov and let $T \subset V$ be a maximal isotropic
$AN$-invariant subspace.

Then $\G$ acts properly discontinuously and with compact quotient
on  the $\rho(\G)$-invariant subset 
\[\Omega = \Omega_{\rho,V,T} =
\phi_{1}^{-1} ( \Omega_{\phi \circ \rho}) \subset G/AN,\]
where $ \Omega_{\phi \circ \rho}$ is the domain of discontinuity constructed in Section~\ref{sec:defin-doma-results}.
\end{thm}

\begin{remark}
Examples in Section~\ref{sec:desclim} illustrate that different choices of irreducible representations $\phi$ can lead to different domains of discontinuity $\Omega \subset G/AN$. 
\end{remark}

The main problem is that the domain of discontinuity $\Omega$ might
be empty (see the examples in Remark~\ref{rem:examples_orth}). 
In order to get criteria for $\Omega$ to be nonempty, we describe its
complement $K = (G/AN \moins \Omega)  \subset G/AN$.
Note that 
\[K =
K_{\rho,V,T} = \phi_{1}^{-1} ( K_{\phi_0 \circ \xi(\partial_\infty
  \G)}),\] 
where $\xi: \partial_\infty \G \to G/P$ is the Anosov map
associated with $\rho$. 

The compact $K$ is the union: 
\begin{equation}
  \label{eq:3}
   K = \bigcup_{t \in \partial_\infty \G}\,  \phi_{1}^{-1} ( K_{\phi_0 \circ \xi(t)}).
\end{equation}

The  $\phi$-equivariance of $\phi_1$ and $\phi_0$ implies that for any element $gP \in G/P$ 
\[
\phi_{1}^{-1} ( K_{\phi_0
  (gP)}) = \phi_{1}^{-1} ( \phi(g) K_{\phi_0(P)}) = g\phi_{1}^{-1} (
K_{V_\lambda}).\] 

We now describe in more detail the set $\phi_{1}^{-1} (K_{V_\lambda})$ which 
is $P$-invariant and in particular
$B$-invariant where $B = Z_K( \mathfrak{a}) AN$ is a minimal parabolic subgroup
contained in $P$.
Let $W$ be the Weyl group, it is  generated by reflections $(s_\alpha)_{\alpha \in
  \Delta}$ corresponding to the simple roots.

\begin{lem}\label{lem:bruhat_dec}
  Let $T \subset V$ be a maximal isotropic subspace given by 
  Lemma~\ref{lem:exist_iso}.  Consider the subset $S_\phi =\{ w \in W
  \mid V_{w\cdot \lambda} \subset T \}$ of the Weyl group.
Then: 
  \begin{enumerate}
  \item The set $\phi_{1}^{-1} (K_{V_\lambda})$ is the disjoint union:
  \[ \phi_{1}^{-1} (K_{V_\lambda}) = \bigcup_{w \in S_\phi} Bw AN \subset
  G/AN. \]
  \item\label{it2bruhat} The set  $\{ w \in W \mid w\cdot \lambda <0\}$ is contained in
$W \moins S_\phi$. 
\item The subset $S_\phi$ is right invariant
  under the action of the subgroup $W_P<W$, which is generated by the $s_\alpha$, $\alpha \in \Theta$, if
  $P= P_\Theta$.
\end{enumerate}
\end{lem}

\begin{proof}
  The set $\phi_{1}^{-1} (K_{V_\lambda})$ is a union of $B$-orbits. The
  Bruhat decomposition for $G$ \cite[Theorem~7.40]{Knapp_LieGrp} reads as $G =
  \bigcup_{ w \in W} BwB$ (disjoint union). Since $W= N_K( \mathfrak{a})
  / Z_K( \mathfrak{a})$ one has, for $w$ in $W$, $Z_K( \mathfrak{a}) w
  Z_K( \mathfrak{a}) = Z_K( \mathfrak{a}) w$ and since $B = Z_K(
  \mathfrak{a}) A N$ one has also $BwB= B w AN$. The space $G/AN$ is
  therefore the disjoint union of finitely many $B$-orbits: 
$BwAN$ ($w \in  W$).

  It is hence enough to understand when $wAN$ belongs to
  $\phi_{1}^{-1} (K_{V_\lambda})$. This happens precisely when 
  $\phi_1( w AN) \in
  K_{V_\lambda}$ or, equivalently, when $V_\lambda \subset \phi(w) T$. This is equivalent to  $\phi(w)^{-1} V_\lambda = V_{w \cdot \lambda} \subset T$, which precisely means that $w \in S_\phi$. This proves the first claim. The second follows from the fact that $V_- \cap T = 0$; the third claim from the fact that 
  $\phi_{1}^{-1} (K_{V_\lambda})$ is not only $B$-invariant but
  $P$-invariant.
\end{proof}

\begin{cor}\label{cor:longest}
Let $w_0 \in W$ be the longest element with respect to the generating set $(s_\alpha)_{\alpha \in
  \Delta}$. 
  Then $w_0$ is
not in $S_\phi$. 
 \end{cor}

\begin{proof} 
 The highest weight of $V^*$ is $-w_0 \cdot \lambda$, thus $-w_0 \cdot \lambda = \lambda$ in view of the isomorphism
$V^* \cong V$.  Hence $w_0$ is
not in $S_\phi$. 
\end{proof}

\begin{remark}
The orbit $Bw_0 B$ is the
unique open orbit;  for $w$ in $W$ one has 
$\mathrm{codim} (B w_0 w AN) \geq \ell(w)$, where  $\ell(w)$ is 
 the length of $w$.
\end{remark}

\subsection{Groups of small virtual cohomological dimension}
\label{sec:doma-disc-free}

The following theorem implies Theorems~\ref{thm_intro:dod_freegrp} and \ref{thm:dod_surf} in the introduction.

\begin{thm}\label{thm:dofFnGammag}
  Let $G$ be a semisimple Lie group, $\rho: \G \to G$ be a $P$-Anosov
  representation of  a finitely generated word hyperbolic group and
  let $\phi: G \to G_F(V)$ be such that $\phi \circ 
  \rho$ is $Q_0$-Anosov. Let $\Omega = \Omega_{\rho,V,T}$ be the
  domain of discontinuity constructed above
  (Section~\ref{sec:doma-disc-gh}).
  \begin{enumerate}
  \item If $\vcd(\G) = 1$, then $\Omega$ is nonempty. 
  \item If $\vcd(\G) = 2$ and $P<G$ is a proper parabolic subgroup
    which contains every factor of $G$ that is 
       locally isomorphic to $\SL(2, \RR)$, then $\Omega$ is nonempty.
  \end{enumerate}
  In any case $\G$ acts properly discontinuously and cocompactly on $\Omega$.
\end{thm}

\begin{proof}
  We only have to prove that $\Omega$ is nonempty. In the cases under consideration $\partial_\infty \G$ is
  $0$ or $1$ dimensional. Using the description of $K=
  G/AN \moins \Omega$ given above (Equation~\ref{eq:3}) it is thus enough to
  establish that:
  \begin{itemize}
  \item $\mathrm{codim}\, \phi_{1}^{-1} (K_{V_\lambda}) \geq 1$ if $\vcd(\G) = 1$, and
  \item $\mathrm{codim}\, \phi_{1}^{-1} (K_{V_\lambda}) \geq 2$ if  $\vcd(\G) = 2$.
  \end{itemize}
  If $G$ is of rank one, this is immediate.  
  If $G$ is of rank $\geq 2$ let $S_\phi\subset W$ be the set defined in Lemma~\ref{lem:bruhat_dec}.
  Using the fact that $\mathrm{codim}\, B w_0 w
  AN$ is at least $\ell(w)$, Corollary~\ref{cor:longest} implies the
  claim in the case when $\vcd(\G) = 1$. 

When  $\vcd(\G) = 2$, we need to prove furthermore that, for all
$\alpha \in \Delta$, either $\mathrm{codim}\, B w_0 s_\alpha 
  AN \geq 2$ or $w_0 s_\alpha \notin S_\phi$. If $\alpha$ is a root
  corresponding to a rank one factor of $G$, the bound for the codimension is satisfied by the hypothesis on $P$. If $\alpha$
  corresponds to a higher rank factor, the following lemma and the
  point \eqref{it2bruhat} of Lemma~\ref{lem:bruhat_dec} insure that
  $w_0 s_\alpha \notin S_\phi$.
\end{proof}

\begin{lem}\label{lem:orderA}
  If $\rank_\RR G \geq 2$ and $G$ is simple and $\lambda \in
  \mathfrak{a}^*$ is the 
  highest weight of a nontrivial $G$-module, then there exist a total
  order $<_{\mathfrak{a}^*}$ on 
  $\mathfrak{a}^*$ such that $\Sigma^+ = \{ \alpha \in \Sigma \mid
  \alpha >_{\mathfrak{a}^*} 0 \}$ and such that, for all $\alpha$ in $\Delta$,
  $s_\alpha \cdot \lambda >_{\mathfrak{a}^*}0$.
\end{lem}

In other words, we can modify the order on $\mathfrak{a}^*$ without
changing the set of positive roots and such that, every $s_\alpha
\cdot \lambda$ becomes positive. 

\begin{proof}
  Let $\rho$ be the half sum of positive roots, $\rho = \sum_{\alpha
    \in \Delta} m_\alpha \alpha$ with $m_\alpha \in \frac{1}{2}
  \NN^*$. Let $\langle \cdot, \cdot  \rangle$ be a $W$-invariant scalar
  product on $\mathfrak{a}^*$. Hence $\langle \lambda, \alpha
  \rangle \geq 0$ for all $\alpha$ and there exists $\alpha$ such that
  $\langle \lambda, \alpha \rangle > 0$. Using the fact that for
  any positive roots $\alpha_1$, $\alpha_2$ such that $\langle
  \alpha_1, \alpha_2 \rangle < 0$, $\alpha_1 + \alpha_2$ is a positive
  root, one easily gets that $m_\alpha \geq 3/2$ for all $\alpha$
  unless $\mathfrak{g} \cong \mathfrak{sl}(3, \RR)$.

  We see that $\langle \rho, \lambda \rangle > 0$ and that for any
  $\alpha'$ 
  \[ \langle \rho, s_{\alpha'}\cdot \lambda \rangle = \langle
  s_{\alpha'}\cdot \rho, \lambda \rangle = \langle \rho - \alpha',
  \lambda \rangle = \sum_\alpha (m_\alpha - \delta_{\alpha, \alpha'})
  \langle \alpha, \lambda \rangle>0.\]
  Hence a lexicographic order on $\mathfrak{a}^* \cong \RR^k$ where
  the first coordinate is $\langle \rho, \cdot \rangle$ gives the desired order.
The case of $\mathfrak{sl}(3, \RR)$ can be treated
  directly.
\end{proof}

\subsection{Homotopy invariance}
\label{sec:homo_inv}
In general it is not easy to determine the topology of the quotient manifolds $\G \backslash\Omega_{
    \rho, V,T}$.  The following theorem shows that the homeomorphism type of  $\G \backslash\Omega_{
    \rho, V,T}$ only depends on the connected component of
  $\hom_{P\text{-Anosov}}(\G, G)$ the representation $\rho$ lies in. 
   This allows to restrict the computation to representations $\rho: \G \to G$ of a particularly nice form. 

\begin{thm}\label{thm:DoD_loc_cst}
  Let $\rho: \G \to G$ be a $P$-Anosov representation. Let $V,T$ be as
  in Section~\ref{sec:g-module}, and let $\Omega_{\rho, V,T} \subset
  G/AN$ be
  the domain of discontinuity constructed in Section~\ref {sec:doma-disc-free}.

  Suppose that $\Omega_{ \rho, V,T}$ is nonempty.

  Then there exists  a neighborhood $U$ of $\rho$ in $\homPano( \G,G)$
  such that, for any $\rho' \in U$, $\Omega_{ \rho', V,T}$ is
  nonempty. Furthermore there exists a trivialization:
  \[ \G \backslash \bigcup_{\rho' \in U} \Omega_{ \rho', V,T} \cong U
  \times \G \backslash \Omega_{ \rho, V,T} \]
  as a bundle over $U$. In particular, the homotopy type of $\G \backslash\Omega_{
    \rho, V,T}$ is locally constant.
\end{thm}

\begin{proof}
  The proof is a variation of the openness of the holonomy
  map for geometric structures \cite{Bergeron_Gelander,
    Goldman_geometric}. 

  We can suppose without loss of generality that $\G$ is torsion-free. Hence
  the quotient $X = \G \backslash \Omega_{ \rho, V,T}$ is a compact manifold. We
  shall denote  by $\widehat{X} = \Omega_{ \rho, V,T}$ the corresponding
  $\G$-cover of $X$.

  Denote by $\mathcal{B}_\rho =  \widehat{X} \times_\rho
  G/AN$ the flat $G/AN$-bundle over $X$ associated
  with the representation $\rho: \G \to G$. It comes with a canonical
  section $\sigma: X \to \mathcal{B}_\rho$ that is transverse to the
  horizontal distribution so that the corresponding
  $\rho$-equivariant map $\hat{\sigma}: \widehat{X} \to G/AN$ is a
  local diffeomorphism (in fact $\hat{\sigma}$ is the injection
  $\Omega_{ \rho, V,T} \hookrightarrow G/AN$).

  The complement $K_\rho$ of $\Omega_{ \rho, V,T}$ in $G/AN$ defines a
  closed subset $\mathcal{K}_\rho = \widehat{X} \times_\rho
  K_\rho\subset \mathcal{B}_\rho$. By construction $\sigma(X) \cap
  \mathcal{K}_\rho = \emptyset$, thus there exists $\epsilon >0$
  such that $d( \sigma(X), \mathcal{K}_\rho) > 2 \epsilon$, where $d$
  is a fixed continuous distance on $\mathcal{B}_\rho$.

  Let $U$ be a neighborhood of $\rho$ contained in the space of
  $P$-Anosov representations. For $U$ sufficiently small, we can suppose that
  the (topological) $G/AN$-bundle $\mathcal{B}_U = \coprod_{\rho' \in
    U} \mathcal{B}_{\rho'}$ is trivial, i.e.\ there exists $\psi: \mathcal{B}_U \cong U
  \times \mathcal{B}_\rho$ with $\psi|_{\mathcal{B}_\rho} = \id$. Note
  that $\mathcal{K}_U = \coprod \mathcal{K}_{ \rho'}$ is a closed
  subset of $\mathcal{B}_U$ (this follows from the fact that 
  $\xi_{\rho'}$ varies continuously with $\rho'$, see Theorem~\ref{thm:Ano_open}). Hence, for $U$ sufficiently small, the 
  section $\sigma_U = \psi^{-1}( \sigma)$ is transverse to the flat
  horizontal distribution and $d( \sigma_U( U\times X), \mathcal{K}_U)
  > \epsilon$. This means that for any $\rho' \in U$ there is a
  $\rho'$-equivariant local diffeomorphism $\hat{\sigma}_{\rho'}:
  \widehat{X} \to G/AN$ whose image is contained in $G/AN \moins
  K_{\rho'} =\Omega_{ \rho', V,T}$; in particular this last set is
  nonempty. Furthermore, passing to the quotient, this gives a local
  diffeomorphism $\beta_{\rho'} : X \to \G \backslash \Omega_{ \rho',
    V,T}$ that varies continuously with $\rho'$ and such that
  $\beta_\rho = \id$. From this we get that $\beta_{\rho'}$ is a
  diffeomorphism for any $\rho'\in U$ and that 
\[\coprod \beta_{\rho'} : U
  \times X \to \coprod \G \backslash \Omega_{ \rho', V,T}\] 
gives a  trivialization.
\end{proof}

\section{Explicit descriptions of some domains of discontinuity}
\label{sec:desclim}

In this section we describe in more detail some 
domains of
discontinuity which are obtained by applying the construction
of Section~\ref{sec:doma-disc-gh}.

\subsection{Lie groups of rank one}
\label{sec:cocoDesc}

Let $\G < G$ be  a convex cocompact subgroup in a rank one Lie
group. The group $\G$ is word hyperbolic and by
Theorem~\ref{thm:Ano_rk_one} the injection $\Gamma \hookrightarrow G$
is an Anosov representation; the corresponding $\rho$-equivariant map is the identification
of $\partial_\infty \G$ with the limit set $\mathcal{L}_\Gamma \subset G/P$ of $\G$ (Section~\ref{sec:ex_rk_one}).

Theorem~\ref{thm:dod_GF} implies that $\Gamma$ acts properly discontinuously and with compact quotient on the complement of the limit set $\Omega_\Gamma = G/P  \moins \mathcal{L}_\Gamma$. 

Note that $\Omega_\Gamma = \emptyset$ if and only if $\Gamma$ is a uniform lattice in $G$. 

\subsection{Representations into $\SL(n, \KK)$}
\label{sec:descSLN}
We now describe domains of discontinuity for representations $\rho: \Gamma \to \SL(n,\KK)$, $\KK = \RR, \CC$ of an arbitrary word hyperbolic group, which are $P$-Anosov with 
$P$ being  
\begin{itemize}
\item the minimal parabolic $B$ (or Borel subgroup), i.e.\ $B$ is the
  stabilizer of a complete flag in $\KK^n$.
\item a maximal parabolic subgroup $P_k$ of $\SL(n, \KK)$,  $k = 1, \dots, n-1$, i.e.\ $P_k$ is the stabilizer of a
  $k$-plane in $\KK^n$.
\end{itemize}

Note that the parabolic subgroup opposite to $P_k$ is (conjugate to) the maximal parabolic subgroup $P_{n-k}$. By Lemma~\ref{lem:ThetaAn} the study of $P_k$-Anosov representations thus reduces to the study of $P_{k,n-k}$-Anosov representations, $k = 1, \dots, \lfloor n/2
\rfloor$, 
 where $P_{k,n-k}$ is the stabilizer of a partial flag consisting of a $k$-plane and an incident $(n-k)$-plane.

In order to fix notation we consider 
\begin{itemize}
\item $\mathcal{F} = G/B = \{(F_1, \dots, F_{n-1}) \mid F_i \subset
  F_{i+1}, \dim F_i =i \}$,
\item $\mathcal{F}_{k,n-k} = G/P_{k,n-k} = \{(F_k, F_{n-k}) \mid F_k \subset
  F_{n-k}, \dim F_i =i \}$,
\end{itemize}
and we set $F_0 = \{0\}$ and $F_n = \KK^n$.

\subsubsection{The modules}
\label{sec:sln_modules}
We introduce now the $\SL(n,\KK)$-modules we use to apply the construction of Section~\ref{sec:doma-disc-gh}.
In this section, $\perp$ will be used
for the duality between a vector space and its dual: for $F \subset
V$, $F^\perp$ is the space of linear forms canceling on $F$, $F^\perp
= \{ \psi \in V^* \mid \psi(f) =0, \forall f \in F \}$.

\subsubsection*{Adjoint Representation}

The adjoint representation provides  a 
homomorphism $\SL(n, \KK) \to \O( \mathfrak{sl}(n, \KK),
q_K)$ where $q_K$ is the Killing form. 
Note that $\mathfrak{sl}(n,
\KK) \subset \mathrm{End}( \KK^n)$ and $ \O( \mathfrak{sl}(n, \KK),
q_K) \subset \O( \mathrm{End}( \KK^n), \tr)$ is a natural injection 
which, by
Lemma~\ref{lem:red_spe}, gives the same domain of
discontinuity. Therefore we can use the latter $\SL(n, \KK)$-module and denote by 
\begin{equation*}
  \phi^\mathrm{Ad}: \SL( n, \KK) \longto \O( \mathrm{End}( \KK^n), \tr)
\end{equation*}
the corresponding homomorphism.
The maps:
\begin{align*}
  \phi^{Ad}_{0} :  \mathcal{F}_{1, n-1} \longto & \mathcal{F}_0( \mathrm{End}(
  \KK^n) ) \\
   (F_1, F_{n-1}) \longmapsto & \{ h \in \mathrm{End}( \KK^n) \mid h(F_{n-1})
  = \{0\}, h( \KK^n) \subset F_1\} \\
  &= \{h \mid h( \KK^n) \subset F_1, h^t( \KK^{n*}) \subset F_{n-1}^\perp\} \\
  \phi^{Ad}_{1} :  \mathcal{F} \longto & \mathcal{F}_1( \mathrm{End}(
  \KK^n) ) \\
   (F_1,\dots, F_{n-1}) \longmapsto & \{ h \in \mathrm{End}( \KK^n) \mid h(F_{i+1})
   \subset F_i, i = 0, \dots, n-1\}
\end{align*}
are $\phi^\mathrm{Ad}$-equivariant. 

\subsubsection*{Endomorphisms of $\wdg{k} \KK^n$}

There is a natural homo\-morphism 
\[\phi^k: \SL(n, \KK) \to \O( \mathrm{End}( \wdg{k}
\KK^n) , \tr ).\] 
 
The map 
\begin{align*}
  \phi^{k}_{0} :  \mathcal{F}_{k, n-k} \longto & \mathcal{F}_0( \mathrm{End}(
  \wdg{k} \KK^n) ) \\
   (F_k, F_{n-k}) \longmapsto 
  &\{h \in \mathrm{End}({\textstyle \bigwedge^{k}} \KK^n) \mid \mathrm{Im}(h) \subset
  \wdg{k} F_k, \mathrm{Im}(h^t)  \subset \wdg{k} F_{n-k}^\perp\}
\end{align*}
is $\phi^k$-equivariant. 

In order to define the map $\phi^{k}_{1} :  \mathcal{F} \to \mathcal{F}_1( \mathrm{End}(
  \wdg{k} \KK^n) )$, we need to introduce some notation. 

Let $(F_1,\dots, F_{n-1})$ be a complete flag of $\KK^n$ and let
$(e_i)_{i \in \{1, \dots, n\}}$ be an adapted basis, i.e.\ $(e_i)_{i
  \in \{1, \dots, l\}}$ is a basis of $F_l$, for each $l = 1,
\dots, n$. For an ordered $k$-tuple $I=(i_1 < \cdots < i_k)$ of
integers between $1$ and $n$, we set $e_I = e_{i_1} \wedge \cdots
\wedge e_{i_k}$; then $(e_I)_{I}$ is a basis of $\wdg{k} \KK^n$. 
The flag $(F_I)_I$, where $F_I = \langle e_J \mid J \leq_{lex} I
\rangle$ and $\leq_{lex}$ is the lexicographic order on the
$k$-tuples, depends only on the initial flag  $(F_i)_i$.
We define:
\begin{align*}
  \phi^{k}_{1} :  \mathcal{F}  \longto\, & \mathcal{F}_1( \mathrm{End}(
  \wdg{k} \KK^n) ) \\
   (F_i) \longmapsto\, & 
  \phi^{k}_{1}(F_i)  = \{h \in
  \mathrm{End}(\wdg{k} \KK^n ) \mid \forall I,\, h(F_I) \subset
  F_I, \text{ and } h \text{ is nilpotent}\} \\
  & \hphantom{\phi^{k}_{1}(F_i) } = \{h \in
  \mathrm{End}(\wdg{k} \KK^n ) \mid  \forall I, \, h(F_I) \subset
  \bigcup_{J<_{lex} I} F_J\}.
\end{align*}
\begin{remark}
Other self-dual $\SL(n, \KK)$-modules that could be used are $V \oplus
V^*$ with its natural orthogonal structure or with its natural
symplectic structure and $\bigwedge V = \bigoplus_k \wdg{k} V$
where $V$ is any $\SL(n, \KK)$-module.
\end{remark}

\subsubsection{Anosov representation with respect to the minimal parabolic}
\label{sec:sln_ano_min}

Let $\Gamma$ be a finitely generated word hyperbolic group and  $\rho: \G \to \SL(n,
\KK)$ be a $B$-Anosov representation.
Let $\xi = (\xi_1, \dots,
\xi_{n-1}) : \partial_\infty \G \to \mathcal{F}$ be the corresponding
Anosov map.

\subsubsection*{Adjoint representation}
The representation $\phi^\mathrm{Ad} \circ \rho : \G \to \O(\mathrm{
  End}( \KK^n), \tr)$ is a $Q_1$-Anosov representation with Anosov map
$\phi^{Ad}_{1} \circ \xi: \partial_\infty \G \to \mathcal{F}_1 (
\mathrm{ End}( \KK^n))$. 

Let $\Omega_{\mathrm{Ad} \circ \rho}\subset\mathcal{F}_0 ( \mathrm{
  End}( \KK^n))$ be the domain of discontinuity given by
Theorem~\ref{thm:dod_GF}.  Intersecting $\Omega_{\mathrm{Ad} \circ
  \rho}$ with the image of $\phi^{Ad}_{0}$ gives a domain of
discontinuity $\Omega^{Ad}_{\rho}$ in $\mathcal{F}_{1,n-1}$.  Then
\[
\Omega^{Ad}_{\rho}=
\mathcal{F}_{1,n-1} \moins K^{Ad}_{\xi}
\] 
with $K^{Ad}_{\xi}
=\bigcup_{t \in \partial_\infty \G} K^{Ad}_{\xi(t)}$ and, for $t\in \partial_\infty \G$,
\begin{equation*}
  K^{Ad}_{\xi(t)} = \{ (F_1, F_{n-1}) \in \mathcal{F}_{1, n-1} \mid
  \exists k \in \{1, \dots, n-1\}, \textup{ with } F_1 \subset \xi_k(t)
  \subset F_{n-1}\}.
\end{equation*}

\subsubsection*{Endomorphisms of $\wdg{k} \KK^n$}

Analogously, for any $k = 1,\dots, \lfloor n/2
\rfloor$ the representation $\phi^k
\circ \rho: \G \to  \O( \mathrm{End}( \wdg{k}
\KK^n) , \tr )$ is $Q_1$-Anosov and 
we obtain a domain of discontinuity in $\mathcal{F}_{k, n-k}$:
  \[
  \Omega^{k}_{\rho}  =
  \mathcal{F}_{k,n-k} \moins K^{k}_{\xi}
  \] 
  with 
 $K^{k}_{\xi} =\bigcup_{t \in \partial_\infty \G} K^{k}_{\xi(t)}$, where  
  \[
  K^{k}_{\xi(t)}  = \{ (F_k, F_{n-k}) \in \mathcal{F}_{k, n-k} \mid
  \exists I  \textup{ with } \wdg{k} F_k
  \subset \xi_I(t)  \subset (\wdg{k} F_{n-k}^\perp )^\perp \}.
  \]

The representations $\phi^k \circ \rho: \G \to  \O( \mathrm{End}( \wdg{k}
\KK^n) , \tr )$ are also $Q_0$-Anosov; the corresponding domains
of discontinuity are described in the next paragraph.

\subsubsection{Anosov representation with respect to maximal parabolics}
\label{sec:sln_ano_max}

\subsubsection*{Adjoint representation}
\label{sec:adjo-repr}

Let $\rho: \G \to \SL(n, \KK)$ be a $P_{1,n-1}$-Anosov representation and 
$\xi = (\xi_1, \xi_{n-1}): \partial_\infty \G \to
\mathcal{F}_{1,n-1}$ the corresponding Anosov map. 
The composition $\phi^\mathrm{Ad} \circ \rho$ is a $Q_0$-Anosov
representation, and $\phi^{Ad}_{0}
\circ \xi : \partial_\infty \G \to \mathcal{F}_0( \mathrm{End} (\KK^n)
)$ is the corresponding Anosov map.
The intersection of the domain of discontinuity in $\mathcal{F}_1( \mathrm{End} (
\KK^n))$ with the image of $\phi^{Ad}_{1}$ gives a
domain of discontinuity $\Omega^{\prime Ad}_{\rho}$ in $\mathcal{F}$;
\[
\Omega^{\prime Ad}_{\rho} = \mathcal{F} \moins K^{\prime
  Ad}_{\xi}\text{ with } K^{\prime Ad}_{\xi} =  \bigcup_{t
  \in \partial_\infty \G} K^{ \prime Ad}_{\xi(t)}, 
\]
   and for $t$ in
$\partial_\infty \G$, 
\[K^{ \prime Ad}_{\xi(t)} = \{ (F_1, \dots,
F_{n-1}) \in \mathcal{F} \mid \exists k \in \{1, \dots, n-1\},\
\xi_1(t) \subset F_k \subset \xi_{n-1}(t) \} .\]

\subsubsection*{Endomorphisms of $\wdg{k} \KK^n$}

Let $\rho: \G \to \SL(n, \KK)$ be  a $P_{k, n-k}$-Anosov
representation (or, what amounts to the same, a $P_k$-Anosov
representation), and $\xi
= (\xi_k, \xi_{n-k}): \partial_\infty \G \to \mathcal{F}_{k, n-k}$
the corresponding Anosov map. 
Then the composition $\phi^k \circ \rho$ is
$Q_0$-Anosov with Anosov map $\phi^{k}_{0}\circ \xi: \partial_\infty \Gamma \to \mathcal{F}_0( \mathrm{End}(
  \wdg{k} \KK^n) )$.  
We obtain a domain of
discontinuity $\Omega^{\prime k}_{\rho}$ in $\mathcal{F}$ which is the
complement of $K^{\prime k}_{\xi} =  \bigcup_{t
  \in \partial_\infty \G} K^{\prime k}_{\xi(t)}$ where
\begin{equation*}
  K^{\prime k}_{\xi(t)} = \{ (F_1, \dots,
  F_{n-1}) \in \mathcal{F} \mid \exists I, \wdg{k} \xi_k(t)
  \subset F_I  \subset (\wdg{k} \xi_{n-k}(t)^\perp )^\perp \}
\end{equation*}

A flag $(F_1, \dots,  F_{n-1})$ belongs to $K^{\prime 
  k}_{\xi(t)}$ if and only if the following holds:
\begin{equation*}
  (i_1, \dots , i_k) \leq_{lex} (j_1, \dots , j_k)
\end{equation*}
where the two sequences $(i_l)$ and $(j_l)$ are defined by:
\begin{align}\label{eq_sequences}
 \forall l \in \{1, \dots, k\}, \, 
&i_l = \min \{ i \mid \dim F_i \cap
  \xi_k(t) = l \} \\
\notag & j_l = \max \{ j \mid \dim F_j +
  \xi_{n-k}(t) = n-k-1+l \}. \label{eq_sequences1}
\end{align}
The sequence $I = (i_l)$ satisfies in fact $I = \min \{
I^\prime \mid \wdg{k} \xi_k(t) \subset F_{I^\prime} \}$ where the
$\min$ is taken with respect to the lexicographic order, and similarly 
$J = \max \{
J^\prime \mid   F_{J^\prime}^{\perp} \subset \wdg{k} \xi_{n-k}(t)^{\perp} \}$.

\begin{remark}
  The explicit descriptions of the domains of discontinuity given here
  show that different choices of $G$-modules in the construction of
  Section~\ref{sec:descr-doma-disc} can lead to different domains of
  discontinuity in the same flag variety.  For example, let $\rho:
  \Gamma \to \SL(n,\KK)$ be a $B$-Anosov representation. Then we can
  consider $\rho$ as $P_k$-Anosov representation for any $k =
  1,\dots, \lfloor n/2 \rfloor$. The domains of discontinuity
  $\Omega_\rho^k \subset \Ff$, if nonempty, are different.
\end{remark}

\subsubsection{Codimension}
\label{sec:sln_codim}
We give bounds for the codimensions of the sets 
$K^{\mathrm{Ad}}_\xi$, $K^k_\xi$, $K^{\prime Ad}_\xi$ and $K^{\prime k}_\xi$. 

\begin{prop}\label{prop:codim}
  Let $k \leq n/2$.
  \begin{asparaenum} 
  \item Let $E = (E_1, \dots, E_{n-1})$ be a complete flag in $\KK^n$ and
    let $(E_I)_I$ be the flag of $\wdg{k} \KK^n$ constructed in Section~\ref{sec:sln_modules}.  Then the codimension of
    \begin{equation*}
      K^{k}_{E} = \{ (F_k, F_{n-k}) \in \mathcal{F}_{k, n-k} \mid
      \exists I, \textup{ with } \wdg{k} F_k \subset E_I \subset
      (\wdg{k} F_{n-k}^\perp)^\perp \}
    \end{equation*}
    in $\Ff_{k,n-k}$ is a least $n-k$.
  \item Let $F = (F_k, F_{n-k}) \in \mathcal{F}_{k, n-k}$, then the
    codimension of 
   \begin{equation*}
      K^{\prime k}_{F} = \{ (E_1,\dots,  E_{n-1}) \in \mathcal{F} \mid
      \exists I, \textup{ with } \wdg{k} F_k \subset E_I \subset
      (\wdg{k} F_{n-k}^\perp)^\perp \}
    \end{equation*}
    in $\Ff$ is a least $n-k$.
  \end{asparaenum}
\end{prop}

\begin{proof}
We discuss here only the
  last  case. The treatment of the other case is similar. 
  If $(E_1,\dots,  E_{n-1}) \in K^{\prime k}_{F}$, then $i_1 \leq
  j_1$ 
  where $i_1 = \min \{ i \mid \dim E_i \cap
  F_k = 1 \}$ and $  j_1 = \max \{ j \mid E_j \subset F_{n-k}(t)  \}$
  (see Equation~\eqref{eq_sequences}).
  Hence
  $K^{\prime k}_{F}$ is the union of the $L_{s,u}$, for $1 \leq s \leq u $,
  with 
  \[L_{s,u} = \{E \in K^{\prime k}_{F} \mid E_s \cap F_k \text{ is a line
    and } E_u \subset F_{n-k}\}.\] 
    In particular it is enough to
  calculate the codimension of the projection  $\bar{L}_{s,u}$ of
  $L_{s,u}$ to the partial flag manifold 
  \[\mathcal{F}_{s,u} =
  \{E_s \subset E_u \mid \dim E_s = s, \dim E_u = u\}.\] 
  The dimension
  of $\mathcal{F}_{s,u}$ is $s(n-s) + (u-s)(n-u)$ whereas the
  dimension of $\bar{L}_{s,u} = \{E_s \subset E_u \mid E_s \cap F_k
  \text{ is a line and } E_u \subset F_{n-k} \}$ is $k-1 +
  (s-1)(n-k-s) + (u-s)(n-k-u)$. Thus, the codimension is
  $n-k + (k-1)(s-1) + k(u-s)\geq n-k$. 
\end{proof}

\begin{cor} \label{cor:nonemptyVCD} Let $k\leq n/2$. Let $\G$ be a finitely generated 
  word hyperbolic group and let $\rho : \G \to \SL(n, \KK)$ be a
  $P_k$-Anosov representation (respectively a $B$-Anosov
  representation). Let $\Omega$ be the domain of discontinuity
  constructed in Section~\ref{sec:sln_ano_max} (respectively
  Section~\ref{sec:sln_ano_min}). Set $\epsilon_\RR =1$ and $\epsilon_\CC = 2$. 
  If the virtual cohomological
  dimension $\vcd(\G)$ is less than or equal to  $\epsilon_\KK (n-k)$, then $\Omega$ is nonempty.
\end{cor}

\begin{proof}
  Indeed 
  by Proposition~\ref{prop:codim} 
the real codimension of the complement of $\Omega$ is at least $\epsilon_\KK (n-k)-
  \dim \partial_\infty \G = \epsilon_\KK(n-k) - \vcd(\G) + 1$ (Lemma~\ref{lem:Bestvina_Mess}).
\end{proof}

\subsubsection{The case of $\SL(2n,\KK)$}
\label{sec:coin}

The $\SL(2n, \KK)$-module $V = \wdg{n} \KK^{2n}$ has a natural 
invariant non-degenerate bilinear form $F: V \otimes V \to \wdg{2n} \KK^{2n}
\cong \KK \sep v \otimes w \mapsto v \wedge w$ that is symmetric when
$n$ is even and symplectic when $n$ is odd. Let us denote by  $\phi^{\wedge}: \SL(2n, \KK) \to G_F( V)$ the corresponding homomorphism. 

\begin{align*}
  \phi^{\wedge}_1:& \, \PP( \KK^{2n}) \longto \mathcal{F}_1(V),\
  [v] \longmapsto \ker(v \wedge \cdot : \wdg{n} \KK^{2n} \to
  \wdg{n+1} \KK^{2n}) \\
  \phi^{\wedge}_0:& \, \Gr_n( \KK^{2n}) \longto \mathcal{F}_0(V),\ 
  P \longmapsto \wdg{n} P.
\end{align*}
Given a $P_1$-Anosov representation $\rho: \G \to \SL(2n, \KK)$ with
Anosov map $\xi_1: \partial_\infty \G \to \PP( \KK^{2n})$, the
composition $\phi^{\wedge} \circ \rho$ is $Q_1$-Anosov with Anosov map
$\phi^{\wedge}_1 \circ \xi_1$.  Let $\Omega_{\phi^{\wedge} \circ \rho} \subset
\mathcal{F}_0(V) $ be the domain of discontinuity constructed in
Section~\ref{sec:orth-sympl-groups_anosov}. The preimage under $\phi^{\wedge}_0$ is
a domain of discontinuity $\Omega^{ \phi^{\wedge}, 1}_{\rho} =
{\phi^{\wedge}_0}^{-1} (\Omega_{\phi^{\wedge} \circ \rho} ) \subset \Gr_n(
\KK^{2n})$ for $\Gamma$. It can be described more explicitly by
setting
\[
K_{\xi_1} = \bigcup_{t\in \partial_\infty \Gamma} K_{\xi_1(t) } =
\bigcup_{t\in \partial_\infty \Gamma}\{ P \in \Gr_n( \KK^{2n})\, |\,
\xi_1(t) \subset P\}.
\]
Then  
\[ \Omega^{ \phi^{\wedge}, 1}_{\rho} = \Gr_n( \KK^{2n}) \moins K_{\xi_1}.\]
Similarly, given a $P_n$-Anosov representation $\rho: \G \to
\SL(2n,\KK)$ with Anosov map $\xi_n: \partial_\infty \G \to \Gr_n(
\KK^{2n}) $ one constructs a domain of discontinuity $\Omega^{ \phi^{\wedge},
  n}_{\rho} \subset \PP( \KK^{2n})$. It satisfies  
    \[
  \Omega^{ \phi^{\wedge},
  n}_{\rho}  = \PP( \KK^{2n}) \moins K_{\xi_n}, 
\ 
\textrm{where }
K_{\xi_n} = \bigcup_{t\in \partial_\infty \Gamma} \PP( \xi_n(t)).
\]
In order to compare those open sets with the domains of discontinuity
$\Omega^{1}_{\rho}, \Omega^{n}_{\rho}$ constructed in
Section~\ref{sec:sln_ano_max} we  denote by $\pi_1: \mathcal{F}(
\KK^{2n}) \to \PP( \KK^{2n})$ and by  $\pi_n: \mathcal{F}(
\KK^{2n}) \to \Gr_n( \KK^{2n})$ the natural projections.

\begin{prop}\label{prop:coinSLDN}
  With the notations introduced above:
  \begin{asparaenum}
  \item Let $\rho: \G \to \SL(2n, \KK)$ be $P_1$-Anosov, then
    $\pi_{n}^{-1}( \Omega^{ \phi^{\wedge}, 1}_{\rho}) = \Omega^{1}_{\rho}$.
  \item \label{item2propcoin} Let $\rho: \G \to \SL(2n, \KK)$ be
    $P_n$-Anosov, then $\pi_{1}^{-1}( \Omega^{ \phi^{\wedge},
      n}_{\rho}) = \Omega^{n}_{\rho}$.
  \end{asparaenum}
\end{prop}

\begin{proof}
  We prove (\ref{item2propcoin}). One has $\Omega^{n}_{\rho} =
  \mathcal{F} \moins K^{n}_{\xi} $ with $K^{n}_{\xi} =
  \bigcup_{t\in \partial_\infty \G} K_{\xi_n(t)}$ and
  $K_{\xi_n(t)} = \{(E_1, \dots, E_{2n-1}) \in \mathcal{F} \mid
  \exists I, \wdg{n} \xi_n(t) \subset E_I
  \subset ( \wdg{n} \xi_n(t)^\perp )^\perp\}$. Also $\pi_{1}^{-1}(
  \Omega^{ \phi^{\wedge}, n}_{\rho}) = \mathcal{F} \moins K^{\prime}_{\xi} $
  with $K^{\prime}_{\xi} = \bigcup_{t\in \partial_\infty \G}
  K^{\prime}_{\xi_n(t)}$ and $K^{\prime}_{\xi_n(t)}=
  \pi_{1}^{-1}(K^{\phi^{\wedge},n}_{\xi_n(t)})$ and
  $K^{\phi^{\wedge},n}_{\xi_n(t)} = \{D \mid D \subset \xi_n(t)\}$. It is
  easy to see that (using e.g.\ Equation~\ref{eq_sequences}) that $K_{\xi_n(t)} \subset
  K^{\prime}_{\xi_n(t)}$. Hence $K^{n}_{\xi} \subset K^{\prime}_{\xi}$
  and $\pi_{1}^{-1}(
  \Omega^{ \phi^{\wedge}, n}_{\rho}) \subset \Omega^{n}_{\rho}$. Since the
  action of $\G$ is proper and cocompact on both these open sets, this
  last inclusion implies the equality $\pi_{1}^{-1}(
  \Omega^{ \phi^{\wedge}, n}_{\rho}) = \Omega^{n}_{\rho}$.
\end{proof}

\begin{remark}\label{rem:complex_1}
For the projective Schottky
groups, discussed in Section~\ref{sec:schottky}, the domain of
discontinuity constructed in $\PP(\KK)$, $ \KK = \RR, \CC$, is
precisely the domain of discontinuity given 
in \cite{Nori, Seade_Verjovsky_Schottky}.
\end{remark}

\subsubsection{The Case of  $\SL(3,\RR)$}
\label{sec:descSLt}

By Lemma~\ref{lem:ThetaAn}.(\ref{itemF:lemThetAn}) an Anosov representation $\rho: \G \to \SL(3, \RR)$ is automatically
$B$-Anosov (because $B = P_{1,2}$ 
is the only parabolic subgroup of $\SL(3,\RR)$ conjugate to its
opposite). The corresponding Anosov map is a pair of compatible maps $\xi_1
: \partial_\infty \G \to \PP^2(\RR)$, $\xi_2 : \partial_\infty \G \to
\PP^2(\RR)^*$ and the domain of discontinuity $\Omega_\rho \subset \Ff(\RR^3)$ is the following open
set (see Section~\ref{sec:sln_ano_min}):
\begin{equation*}
  \Omega_\rho = \{ (E_1 , E_2) \in \mathcal{F}( \RR^3) \mid
   E_1 \notin \xi_1(\partial_\infty \G) \text{ and } E_2 \notin
   \xi_2(\partial_\infty \G) \}.
\end{equation*}

When $\G = \pi_1(\Sigma)$ is a surface group, $B$-Anosov
representations and  the above domain of discontinuity 
have been studied by Barbot
\cite{Barbot_anosov}.
He proved the following dichotomy:
\begin{enumerate}
\item If $\rho$ is in the Hitchin component, then $\Omega_\rho$ has three connected components. 
One component $\Omega_1$ is the pull-back of the invariant convex set
$\mathcal{C} \subset \PP^2(\RR)$ (see Goldman \cite{Goldman_convex}),
another component $\Omega_2$ is the pull-back of the invariant convex
set $\mathcal{C}^* \subset \PP^2(\RR)^*$. The third component
$\Omega_3$ is ``de-Sitter like'', it is the set of flags $(D,P)$ with
$D \notin \overline{\mathcal{C}}$ and $P  \notin \overline{\mathcal{C}}^*$. For any $i$, $\pi_1(\Sigma)
\backslash \Omega_i$ is homeomorphic to the projectivized tangent bundle
of $\Sigma$.
\item If $\rho$ is not in the Hitchin component, then $\Omega_\rho$ is connected and  
$ \pi_1(\Sigma) \backslash
  \Omega_\rho$ is diffeomorphic to a circle bundle over $\Sigma$; 
  Barbot asked if it is always homeomorphic to the double cover
  $\pi_1( \Sigma) \backslash \SL(2, \RR)$ of 
   $T^1 \Sigma \cong \pi_1( \Sigma) \backslash \PSL(2, \RR)$. This is
  known to be true in some explicit examples.
\end{enumerate}

\subsubsection{Holonomies of convex projective structures}
\label{sec:descDiv}

If $\G \subset \SL(n+1, \RR)$ divides a strictly convex set
$\mathcal{C}$ in $\PP^n(\RR)$ (see Section~\ref{sec:divi_conv}) then the
injection $\iota : \G \to \SL(n+1, \RR)$ is $P_{1,n}$-Anosov. Thus the construction in Section~\ref{sec:sln_ano_max} provides a  domain of discontinuity $\Omega$ in the full flag variety
$\mathcal{F}( \RR^{n+1})$. The pull-back of
$\mathcal{C}$ to $\mathcal{F}( \RR^{n+1})$ and the pull-back of the
dual convex set $\mathcal{C}^*$ are components of $\Omega$. However, 
$\Omega$ has in general other components, for example
for a lattice in $\SO(1,n)$, $\Omega$ has $n+1$ components.

\subsection{Representations into $\Sp(2n, \KK)$}
\label{sec:descSPN}
Section~\ref{sec:orth-sympl-groups_anosov} gives a direct construction of
domains of discontinuity for $Q_i$-Anosov representations $\rho: \G
\to \Sp(2n,\KK)$, $i=0,1$. By embedding $\Sp(2n,\KK)$ into
$\SL(2n,\KK)$, the construction in Section~\ref{sec:sln_ano_max} can
be applied to representations $\rho: \G \to \Sp(2n,\KK)$ which are
$P_k$-Anosov, where $P_k$ is the stabilizer of an isotropic $k$-plane.
This gives domains of discontinuity in the complete flag variety of
$\Sp(2n, \KK)$:
\begin{equation*}
  \mathcal{F}_\Sp = \{ (E_1, \dots, E_{2n-1}) \in \mathcal{F}(
  \KK^{2n}) \mid \forall i, E_{2n-i} = E_{i}^{\perp,\omega} \}. 
\end{equation*}

For $Q_i$-Anosov representations this is the same as applying the general construction of
Section~\ref{sec:descr-doma-disc} using the representations 
$\phi^1: \Sp(2n, \KK) \to \O( \mathrm{End}( \KK^{2n}), \tr)$ when $i=0$ and
$\phi^n: \Sp(2n, \KK) \to \O( \mathrm{End}( \wdg{n} \KK^{2n}),
\tr)$ when $i=1$. 
Thus, for a $Q_0$-Anosov representation, we constructed three domains
of discontinuity: 
$\Omega^{0}_{\rho}\subset\mathcal{L}$, $\Omega^{0, \Sp}_{\phi^1,
  \rho} \subset \Ff_\Sp$, and  $\Omega^{0, \Sp}_{\phi^n, \rho}\subset \Ff_\Sp$. 
  Similarly, for a $Q_1$-Anosov representation, we obtain 
   $\Omega^{1}_{\rho}\subset \PP( \KK^{2n})$, $\Omega^{1, \Sp}_{\phi^1,
  \rho}\subset \Ff_\Sp$, and $\Omega^{1, \Sp}_{\phi^n, \rho}\subset \mathcal{F}_\Sp$.

\begin{prop}
Let 
$\pi_{1, \Sp} : \mathcal{F}_\Sp \to \PP( \KK^{2n})$ and $\pi_{n, \Sp} :\mathcal{F}_\Sp \to \mathcal{L}$ be the natural projection. 
  \begin{asparaenum}

  \item \label{item:propcoinSP} If $\rho$ is $Q_0$-Anosov, then   $\Omega^{0, \Sp}_{\phi^1,
  \rho} = \Omega^{0, \Sp}_{\phi^n, \rho} = \pi_{n,
  \Sp}^{-1}(\Omega^{0}_{\rho})$.
  \item If $\rho$ is $Q_1$-Anosov, then   $\Omega^{1, \Sp}_{\phi^1,
  \rho} = \Omega^{1, \Sp}_{\phi^n, \rho} = \pi_{1,
  \Sp}^{-1}(\Omega^{1}_{\rho})$. 
  \end{asparaenum}
\end{prop}

\begin{proof}
  The equalities $\Omega^{0, \Sp}_{\phi^n, \rho} = \pi_{n,
  \Sp}^{-1}(\Omega^{0}_{\rho})$ and $\Omega^{1, \Sp}_{\phi^n, \rho} = \pi_{1,
  \Sp}^{-1}(\Omega^{1}_{\rho})$ follow from
Proposition~\ref{prop:coinSLDN}. To prove, for example, the second
equality in (\ref{item:propcoinSP}): $\Omega^{0, \Sp}_{\phi^1, \rho} =
\pi_{n, \Sp}^{-1}(\Omega^{0}_{\rho})$, it is enough to note that, for
$L\in \PP( \KK^{2n})$, $\{(E_1, \dots, E_{2n-1} ) \in
\mathcal{F}_\Sp \mid \exists k, L \subset E_k \subset L^{\perp \omega}
\} = \{(E_1, \dots, E_{2n-1} ) \in
\mathcal{F}_\Sp \mid L \subset E_n\}$.
\end{proof}

\begin{remark}\label{rem:upcoming} 
As mentioned in Section~\ref{sec:ex_max} any maximal representation
$\rho: \pi_1( \Sigma) \to \Sp(2n,\RR)$ is $Q_1$-Anosov. The quotient manifolds $M = \pi_1( \Sigma)
\backslash \Omega^{1}_{\rho}$ will be investigated in more detail in
\cite{Guichard_Wienhard_DoDSymp}, using results on
topological invariants for maximal representations from 
\cite{Guichard_Wienhard_InvaMaxi}. In particular, we deduce that
the manifolds $M$ are homeomorphic to the total space of
$\O(n)/\O(n-2)$-bundles over $\Sigma$. This implies  also that for  
representations $\rho: \pi_1( \Sigma) \to \SL(2n,\RR)$ in the Hitchin component the quotient manifold
$\pi_1( \Sigma) \backslash \Omega^{ \phi^{\wedge}, n}_{\rho}$ is
homeomorphic to the total space of  a $\O(n)/\O(n-2)$-bundle over
$ \Sigma$.
\end{remark}

\subsection{Representations into $\SO(p,q)$}
\label{sec:descSOtn}
The construction in Section~\ref{sec:orth-sympl-groups_anosov} gives an explicit description of domains of discontinuity $\Omega_\rho^{1-i}\subset \Ff_{1-i}$ for $Q_i$-Anosov representations $\rho:\Gamma \to \SO(p,q)$, $i= 0,1$.
\begin{enumerate}
\item 
Let $\Gamma < \SO(1,n)$ be a convex cocompact subgroup. Consider the
$Q_0$-Anosov representation $\rho: \Gamma \to \SO(1,n+1)$, obtained
by naturally embedding $\SO(1,n) < \SO(1,n+1)$. 
Then the quotient
$\G\backslash \Omega_\rho$ is the union of two
  copies of $\Gamma \backslash \HH^n$ when $\Gamma$ is cocompact or the double
  of the compact manifold with boundary whose interior is $\Gamma
  \backslash \HH^n$ otherwise. This is similar to what happens for Fuchsian
  groups embedded in $\PSL(2, \CC)$. 
\item When $\Gamma< \SO(1,n)< \SO(p,q)$ ($p \leq q$) is a  convex cocompact subgroup of 
$\SO(1,n)$, we have $\Omega_{\rho}^{1-i} = \emptyset$ if $q=n$ and
$\Gamma$ is a lattice, but one gets interesting domains of discontinuity
for $q= n+1$, see Section~\ref{sec:cliffordklein}.
\item A representation $\rho: \pi_1(\Sigma) \to \SO(2,3)$ in the
  Hitchin component is $Q_1$-Anosov. In this case we obtain a nonempty
  domain of discontinuity 
$\Omega_\rho^0 \subset \Ff_0$ in the Einstein space. The quotient
$\pi_1(\Sigma)\backslash \Omega_\rho^0$ consists of two connected
components, which are both homeomorphic to the unit tangent bundle of
the surface.  Considering the representation $\rho$ as $Q_0$-Anosov
representation, one obtains a domain of discontinuity in
$\PP^3(\RR)$. 
\end{enumerate}

%%%%%%%%%%%%%%%%%%%%%%%%% Part III: Applications %%%%%%%%%%%%%%%%

\part{Applications}
\label{part:appli}

\section{Higher Teichm\"uller
  spaces}\label{sec:surfteic}

In Part~\ref{part:doma-disc}  we gave a construction of domains of discontinuity for Anosov representations $\rho: \G \to G$. The domains of discontinuity are open subsets $\Omega_\rho$ of a homogeneous space $X = G/H$ for some subgroup $H$ containing $AN$. 
The quotient $W= \G\backslash\Omega_\rho$ is naturally equipped
with a $(G,X)$-structure. Thus one can rephrase the result of 
Theorem~\ref{thm:dod_general_case} 
as associating a $(G,X)$-structure to an Anosov representation.
In this section we make this statement 
 precise for representations
in a Hitchin component as well as for maximal
representations.

\subsection{Geometric structures}
\label{sec:geometric-structures}

Let $G$ be a Lie group and $X$ be a manifold with a smooth $G$-action. 
(The definition below can be adapted to treat more general $G$-spaces.)

A \emph{$(G,X)$-variety} $W$ is a topological space together with a
(maximal) $G$-atlas on $W$, that is 
\begin{enumerate}
\item an open cover $\mathcal{U}$ and,
for each $U$ in $\mathcal{U}$ 
a homeomorphism $\phi_U: U \to \phi_U(U)$ onto an open subset of $X$,
such that 
\item 
for any $U$ and $U'$ in $\mathcal{U}$, $\phi_{U'} \circ
\phi^{-1}_{U} : \phi_U (U \cap U') \to \phi_{ U'} (U \cap U')$ is
(locally) the restriction of an element of $G$. 
\end{enumerate}
The maps $\phi_U$ are
called \emph{charts}. 
A map $\psi: W \to W'$ between $(G,X)$-varieties is a \emph{$G$-map} if $\psi$ is (locally) in the charts an element of $G$.

A \emph{$(G,X)$-structure} on a manifold $M$ is a pair $(W,f)$ of a
$(G,X)$-variety $W$ and a diffeomorphism $f: M \to W$. 
Two $(G,X)$-structures $(W,f)$ and $(W', f')$ on $M$ are said to be \emph{equivalent} if there exists a $G$-map $\psi: W \to W'$. They are said to be 
\emph{isotopic} if there exists a $G$-map $\psi: W \to W'$ such that 
$\psi\circ f'$ is isotopic to $f$. 

The space of isotopy classes of $(G,X)$-structures on $M$ is denoted by
$\mathcal{D}_{G,X}(M)$. The space of equivalence classes is denoted by
$\mathcal{M}_{G,X}(M)$. There is a natural action of
$\mathrm{Diff}(M)$ on $\mathcal{D}_{G,X}(M)$ (by precomposition) that
factors through the mapping class group $\mathrm{Mod}(M) = \pi_0 (\mathrm{Diff}(M))$; 
and $\mathcal{M}_{G,X}(M) = \mathcal{D}_{G,X}(M)/ \mathrm{Mod}(M)$.

\subsection{The holonomy theorem}
\label{sec:holonomy-theorem}
In this section we recall some background on locally homogeneous $(G,X)$-structures, we refer the reader to \cite[Section~3]{Goldman_geometric} for more details.

Every $(G,X)$-structure $(W,f)$ on $M$ induces a
$\pi_1(M)$-invariant $(G,X)$-structure on the universal cover
$\widetilde{M}$. As $\widetilde{M}$ is simply connected, the
$(G,X)$-structure on $\widetilde{M}$ can be encoded in \emph{one} map
$\dev: \widetilde{M} \to X$ that is a local diffeomorphism (the charts
$\phi_U$ can be ``patched'' together). The map $\dev$ is called the
\emph{developing map} and is unique up to postcomposition with
elements of $G$.

 This uniqueness means that there exists a representation $\rho:
\pi_1( M) \to G$ such that $\dev \circ \gamma = \rho( \g) \cdot \dev$
for any $\g\in \pi_1(M)$. The homomorphism $\rho$ is called the
\emph{holonomy representation}. If $\dev$ is changed to $g
\cdot \dev$ then $\rho$ is changed to the conjugate homomorphism 
$\g \mapsto g \rho(\g) g^{-1}$, in particular only the conjugacy
 class of $\rho$ is well defined by the
$(G,X)$-structure $(W,f)$. This conjugacy class is denoted by 
$\hol(W,f) \in \hom(
\pi_1(M), G) / G$.

The space of isotopy classes $\mathcal{D}_{G,X}(M) $ is thus identified
with equivalence classes of pairs $(\dev, \hol)$, of a local diffeomorphism $\dev$
that is  equivariant with respect to a representation $\hol$, and can be topologized using the compact open topology on these spaces of maps.

If $\psi: M \to M$ is a diffeomorphism, then the holonomy representations for $(W,f)$ and $(W,f
\circ \psi)$ are related by $\hol(W,f \circ \psi) = \hol(W,f)  \circ
\psi_*$ where $\psi_* : \pi_1(M) \to \pi_1(M)$ is the induced homomorphism.
Hence the holonomy map descends to a map
$\hol: \mathcal{D}_{G,X}(M) \to \hom(\pi_1(M), G) / G$ that is
$\mathrm{Mod}(M)$-equivariant.

\begin{thm}\cite[\S 5.3.1]{Thurston} 
  (Holonomy Theorem)

  The holonomy map $\hol: \mathcal{D}_{G,X}(M) \to \hom(\pi_1(M), G) /
  G$ is a local homeomorphism.
\end{thm}

\begin{remark}
  In order to avoid having to deal with potential singularities in
    the $G$-quotient  $\hom(\pi_1(M), G) /
  G$, one can work with the space  $\mathcal{D}^{*}_{G,X}(M)$ of \emph{based}
   $(G,X)$-structures. 
    Here, in addition, a germ of 
 $(G,X)$-structure at a point $m \in
    M$ is specified. Then the holonomy map is well
    defined as a map $\hol: \mathcal{D}^{*}_{G,X}(M) \to \hom(
    \pi_1(M,m), G)$ that is a local homeomorphism. 
\end{remark}

\subsection{Hitchin component for $\SL(2n, \RR)$}
\label{sec:hitch-comp-sl2n}

\begin{thm}\label{thm_hitchin_sl2n}
  Let  $\Sigma$ be a closed connected
  orientable surface of genus $ \geq 2$. 
  Let $\mathcal{C}$ be the Hitchin component of $ \hom(\pi_1(\Sigma), \PSL(2n, \RR))/ \PSL(2n,
  \RR)$. Assume that $n\geq 2$. 

  Then there exist
   \begin{itemize}
  \item a compact $(2n-1)$-dimensional manifold $M$,
  \item a homomorphism $\pi: \pi_1(M) \to \pi_1(\Sigma)$,
  \item and a connected component $\mathcal{D}$ of the  space $\mathcal{D}_{\PSL(2n, \RR), \PP^{2n-1}(\RR)}(M)$ of $(\PSL(2n, \RR), \PP^{2n-1}(\RR))$-structures on $M$.
 \end{itemize}
  such that
  \begin{enumerate}
  \item The map
    \begin{align*}
      \mathcal{C} & \longrightarrow \mathcal{C}_M \subset \hom(
      \pi_1(M), \PSL(2n, \RR) ) / \PSL(2n, \RR) \\
      \rho & \longmapsto \rho \circ \pi
    \end{align*}
    is a homeomorphism onto a connected component $\mathcal{C}_M$.
    \item The restriction of the holonomy map $\hol$ 
       to
      $\mathcal{D}$ is a homeomorphism onto $\mathcal{C}_M$, i.e.\ $\hol
      |_{\mathcal{D}}: \mathcal{D} \overset{\sim}{\longto} \mathcal{C}_M$.
  \end{enumerate}

  Furthermore, there exists a homomorphism $\theta: \mathrm{Mod}(\Sigma)
  \to \mathrm{Mod}(M)$ such that the identification $\mathcal{C} \cong
  \mathcal{C}_M$ is $\theta$-equivariant. In other words, 
$ \mathcal {D} \cong \mathcal{C}$ is equivariant with respect to  the action of the  mapping class group.
\end{thm}

\begin{proof}
  For $\rho \in \mathcal{C}$, let $\Omega_\rho \subset \Sphr^{2n-1}$ the
  lift to the sphere of the domain of discontinuity constructed in
  Section~\ref{sec:coin}, considering $\rho: \pi_1(\Sigma) \to
  \PSL(2n,\RR)$ as a $P_n$-Anosov representation. When $n= 2$ the
  domain of discontinuity has two connected  components and
  $\Omega_\rho$ denotes any one of them.
  The quotient space $W_\rho = \G \backslash
  \Omega_\rho$ is a $(\PSL(2n, \RR), \PP^{2n-1}(
  \RR))$-variety. Moreover, by Theorem~\ref{thm:DoD_loc_cst}, the
  total space $\mathcal{W} = \coprod_{\rho \in \mathcal{C}} W_\rho$ is
  a fiber bundle over the base $\mathcal{C}$ (i.e.\ locally over
  $\mathcal{C}$, $\mathcal{W}$ is a product). Since $\mathcal{C}$ is
  simply connected (actually, by   \cite[Theorem~A]{Hitchin},
  $\mathcal{C}$ is a cell), this fiber bundle is trivial, i.e.\ 
  $\mathcal{W} = \mathcal{C} \times M$.

  In particular, for each $\rho \in \mathcal{C}$, there is a
  diffeomorphism $f_\rho: M \to W_\rho$.  Hence there is a continuous map
  $\sigma : \mathcal{C} \to \mathcal{D}_{ \PSL( 2n, \RR), \PP^{2n-1}(
    \RR)}(M)\sep \rho \mapsto (W_\rho, f_\rho)$. Moreover the 
  $\Gamma$-cover $\Omega_\rho \to W_\rho$ gives a homomorphism $\pi:
  \pi_1(W_\rho) \cong \pi_1(M) \to \G$ that does not depend on $\rho$
  (again using that $\mathcal{C}$ is simply connected).

  The map $\sigma$ fits in the diagram
\[ \xymatrix{  & \mathcal{D}_{\PSL( 2n, \RR), \PP^{2n-1}(\RR)}(M) \ar[d]^{\hol} \\
    \mathcal{C} \ar@/^1pc/[ru]^-{\sigma} \ar[r] & \hom( \pi_1(M), \PSL(2n,
    \RR))/ \PSL(2n, \RR)} \] 
where the bottom map is $\rho \mapsto \rho \circ \pi$.
Since $\sigma$ is injective and $\hol$ is a local homeomorphism,
the statements of the theorem follow if the map $\beta: \rho
\mapsto \rho \circ \pi$ is onto a connected component
$\mathcal{C}_M$ of $\hom( \pi_1(M), \PSL(2n, \RR))/ \PSL(2n, \RR)$.
For this, it is enough to show that $\beta$ is open. 

When $n \geq 4$, the codimension of $K_\rho = \Sphr^{2n-1}
\moins \Omega_\rho$ is bigger than $3$, hence $\pi_1(\Omega_\rho) =
\pi_1( \Sphr^{2n-1}) = \{1\}$ and $\pi_1(M) = \pi_1( \Sigma)$ and
$\beta$ is the identity and is obviously open.

For $n=2$, one observes (considering a Fuchsian representation and
using Theorem~\ref{thm:DoD_loc_cst}) that
$\Omega_\rho \cong \SL(2, \RR)$ or $\Omega_\rho \cong \SL(2, \RR) /
(\ZZ/3\ZZ)$, $\ZZ/3\ZZ \subset \SO(2)$, depending on which connected component of the domain of discontinuity one considers.  In particular,  $M$ is the total space
of a circle bundle over $\Sigma$ with Euler number $g-1$ or
$3g-3$. 
The set $\{\rho \in \hom(\pi_1(M), \PSL(2n, \RR)) \mid Z(\rho) \text{
  is finite}\}$ is open in $\hom(\pi_1(M), \PSL(2n, \RR))$ and since
$\pi_1(M) \to \pi_1( \Sigma)$ is a central extension, the image by
$\beta$ of  $\{\rho \in \hom(\pi_1( \Sigma),
\PSL(2n, \RR)) \mid Z(\rho) \text{ is finite}\}$ is open. By
\cite[Lemma~10.1]{Labourie_anosov} $\mathcal{C} \subset \{\rho \in
\hom(\pi_1( \Sigma), 
\PSL(2n, \RR)) \mid Z(\rho) \text{ is finite}\}$, 
hence $\beta$ is open. 

For $n=3$ the complement  $K_\rho$ of $\Omega_\rho$ in $\Sphr^{5}$ is
homeomorphic to $(\Sphr^1 \times
\Sphr^2)/ \{\pm 1\}$. Hence $\pi_1( K_\rho) \cong \ZZ$ and $H^3(
K_\rho, \ZZ) \cong \ZZ/2\ZZ$. Alexander duality implies that $H_1(
\Omega_\rho, \ZZ) \cong \ZZ/2\ZZ$. If $U$ denotes a tubular
neighborhood of $K_\rho$, then $U \moins K_\rho$ has abelian
fundamental group. 
The Van Kampen theorem implies now that $\pi( U \moins K_\rho) \to
\pi_1(\Omega_\rho)$ is onto (otherwise, one would get $\pi_1( \Sphr^5)
\neq 0$) and hence that $\pi_1( \Omega_\rho)$ is
abelian. In conclusion, $\pi_1( \Omega_\rho) \cong H_1( \Omega_\rho,
\ZZ) \cong \ZZ/2\ZZ$ is finite. This is enough to show that  $\beta$
is open.

The mapping class group acts naturally on $\mathcal{C}$
and on $\coprod_{\rho \in \mathcal{C}} \Omega_\rho$ and hence on
$\mathcal{W} = \Gamma \backslash \coprod_{\rho \in \mathcal{C}}
\Omega_\rho$. Thus, for each $\psi\in \mathrm{Mod}( \Sigma)$, we get
a bundle automorphism of $\mathcal{W} \cong \mathcal{C} \times M$,
that is to say a family of diffeomorphisms $\{f_{\psi, \rho}\}_{\rho
  \in \mathcal{C}}$.  Since $\mathcal{C}$ is connected, the class
$\theta(\psi)\in \mathrm{Mod}(M)$ of $f_{\psi, \rho}$ is well
defined independently of $\rho$. 
This defines a homomorphism $\theta: \mathrm{Mod}(\Sigma)
  \to \mathrm{Mod}(M)$
satisfying all the wanted properties.
\end{proof}

\begin{remarks}
\noindent
\begin{asparaenum}
\item 
Theorem~\ref{thm_hitchin_sl2n}  and Theorem~\ref{thm:hitchin_other} below solve the problem of giving a geometric
interpretation of Hitchin components\footnote{Hitchin did ask for such a geometric interpretation in  \cite{Hitchin}.}. 
 We obtain an embedding of the Hitchin component into the deformation
 space of geometric structures (e.g.\ real projective structures when
 $G = \PSL(2n,\RR)$), such that the image is a 
connected component of $\mathcal{D}_{G,X}(M)$. This implies that
the Hitchin component parametrizes specific $(G,X)$-structures on a
manifold $M$.  However, in the general case, we do not 
characterize the image in $\mathcal{D}_{G,X}(M)$ in geometric
terms. A geometric characterization had been obtained for $\PSL(3,\RR)$ by Choi and Goldman
\cite{Goldman_Choi, Goldman_convex}, and for $\PSL(4, \RR)$ by the
authors \cite{Guichard_Wienhard_Duke}.

\item In \cite{Guichard_Wienhard_DoDSymp} we determine the
  homeomorphism type of $M$ and show that $M$ is homeomorphic to the
  total space of an $\O(n)/\O(n-2)$-bundle over $\Sigma$.
\end{asparaenum}
\end{remarks}

\medskip
\subsection{Hitchin components for classical groups}
\label{sec:hitch-comp-other}

\begin{thm}\label{thm:hitchin_other} 
  Let $\Sigma$ be a closed connected orientable surface of genus
  $\geq 2$.
  Assume that $G$ is 
$\PSL(2n, \RR)$ ($n \geq 2$), 
 $\PSp(2n, \RR)$ ($n \geq
  2$), or 
 $\PSO(n,n)$ ($n\geq 3$), 
  and $X = \PP^{2n-1}(\RR)$; 
 or that $G$ is 
  $\PSL(2n+1, \RR)$ ($n \geq 1$), or 
 $\PSO(n,n+1)$ ($n\geq
  2$), and  $X = \mathcal{F}_{1,2n}(\RR^{2n+1}) = \{ (D,H) \in
  \PP^{2n}(\RR) \times \PP^{2n}(\RR)^* \mid D \subset H \}$.

  Let $\mathcal{C} \subset \hom( \pi_1(\Sigma), G)/G$ be the
  Hitchin component.

  Then there exists a compact manifold $M$, a homomorphism $\pi:
  \pi_1(M) \to \pi_1( \Sigma)$, and 
   a connected component $\mathcal{D}$
  of the deformation space $\mathcal{D}_{G,X}(M)$ 
  such that
  \[\mathcal{C} \to \mathcal{C}_M \sep \rho \mapsto \rho \circ \pi\] is
  a homeomorphism onto a connected component $\mathcal{C}_M$ of
  $\hom(\pi_1(M),G)/G$ and such that 
  \[\hol |_{\mathcal{D}}:
  \mathcal{D} \to \mathcal{C}_M\] 
  is a homeomorphism.

  Furthermore there is a homomorphism $\theta: \mathrm{Mod}(\Sigma) \to
  \mathrm{Mod}( M)$ such that the identification $\mathcal{C} \cong
  \mathcal{D}$ is $\theta$-equivariant.
\end{thm}

\begin{proof}
The proof proceeds along the same lines as the proof of Theorem~\ref{thm_hitchin_sl2n},  considering  the following domains of discontinuity $\Omega_\rho \subset X$.   
\begin{enumerate} 
\item When $X = \PP(\RR^{2n})$, $\Omega_\rho$ is defined in Section~\ref{sec:coin}, regarding $\rho:\pi_1(\Sigma) \to G \to \PSL(2n,\RR)$ as a $P_n$-Anosov representation. 
\item When $X= \mathcal{F}_{1,2n}(\RR^{2n+1})$, $\Omega^{Ad}_\rho$ is defined in Section~\ref{sec:sln_ano_min}, regarding $\rho: \pi_1(\Sigma) \to G \to \PSL(2n+1,\RR)$ as a $B$-Anosov representation. 
\end{enumerate}
The central ingredients are that the Hitchin component $\mathcal{C}$ is simply
connected and that the  fundamental group of the domain of discontinuity is finite (or centralized by $\pi_1(\Sigma)$).
\end{proof}
\subsection{Components of the space of maximal representations}
\label{sec:maxim-comp-sympl}
Components of the space of maximal representations might have nontrivial topology, in particular their fundamental groups can be nontrivial.  

For example, work of Gothen \cite[Propositions~5.11, 5.13 and
5.14]{Gothen} implies that the components of the space of maximal
representations of $\pi_1(\Sigma)$ into $\Sp(4,\RR)$, which are not
Hitchin components, have fundamental groups which are isomorphic to
surface groups,  $(\ZZ/2\ZZ)^{2g}$, or to $\ZZ^{2g}$. 

\begin{thm}\label{thm:compspn}
 Let $\Sigma$ be a
  closed connected orientable surface of genus $\geq 2$. 
   Let $\mathcal{C}\subset \hom(\pi_1(\Sigma), \Sp(2n,
  \RR))/\Sp(2n,\RR))$, 
   $n\geq 2$, be a component of the space of maximal representations.

  Then there exists a compact manifold $M$ of dimension $2n-1$ and 
  a homomorphism $\pi: \pi_1(M) \to \pi_1(\Sigma)$,  such that
   $\rho \mapsto \rho \circ \pi$ gives an identification of $\mathcal{C}$
  with a connected component $\mathcal{C}_M$ of $\hom(\pi_1(M),
  \Sp(2n, \RR)) / \Sp(2n, \RR)$. 

Furthermore there exists a connected component $\mathcal{D}$ of the deformation space 
  $\mathcal{D}_{\Sp(2n, \RR), \PP^{2n-1}(\RR)}(M)$ and a homomorphism $\kappa: \pi_1(
  \mathcal{C}) \to \mathrm{Mod}(M)$ such that $\hol : \mathcal{D} \to
  \mathcal{C}_M \cong \mathcal{C}$ is the Galois cover associated with
  $\ker \kappa$. 
  The corresponding isomorphism of universal covers
  induces a local homeomorphism
   $\widetilde{ \mathcal{D}} \cong \widetilde{ \mathcal{C}} \to \mathcal{D}$, that is 
  equivariant
  with respect to the subgroup $\mathrm{Mod}_{\mathcal{C}}$ of
  $\mathrm{Mod}( \Sigma)$ stabilizing $\mathcal{C}$.
\end{thm}

\begin{remark}
Work of Garc\'\i{}a-Prada, Gothen and
Mundet i Riera \cite{GarciaPrada_Gothen_Mundet} seems to imply that components of the space of 
maximal representations of $\pi_1(\Sigma) $ into $\Sp(2n,\RR)$ with $n\geq 3$ are always simply connected. If this holds true the statement of the theorem can be simplified when $n\geq 3$. 
\end{remark}

Of course, there are similar statements for components of the space of  maximal representations of $\pi_1(\Sigma)$ into other Lie groups $G$ of Hermitian type. 
For classical Lie groups we list here the homogeneous space $X$ the geometric structure is modeled on:
\begin{itemize}
\item $G= \SO(2,n)$, $X= \mathcal{F}_1(\RR^{2+n})$ the space of
  isotropic $2$-planes.
\item $G= \SU(p,q)$, $X\subset \PP^{p+q-1}(\CC)$ is the null cone for
  the Hermitian form.
\item $G= \SO^*(2n)$, $X \subset \PP^{n-1}(\HH)$ is the null cone for the skew-Hermitian form.
\end{itemize}

\section{Compactifying quotients of symmetric spaces}
\label{sec:compactify}
Let $\Gamma$ be a word hyperbolic group and $\rho: \Gamma \to G$ an Anosov representation. Then $\rho(\G)<G$ is a discrete subgroup, and the action of $\rho(\G)$ on the symmetric space $\mathcal{H} = G/K$ by isometries is properly discontinuous. In most cases, the quotient $M = \rho(\G) \backslash \mathcal{H}$ will not be compact nor of finite volume. 

In this section we will describe how the construction of domains of discontinuity, together with Proposition~\ref{prop:dod_in_smaller} can be applied in order to describe compactifications of $\G \backslash \mathcal{H}$. More precisely, in a suitable compactification $\overline{\mathcal{H}}$ of
 $\mathcal{H}$ we will describe a $\rho(\G)$-invariant subset
 $\overline{\mathcal{H}}_\rho$ containing $\mathcal{H}$ such that the
 following holds: $\rho(\G)$ acts properly discontinuously on
 $\overline{\mathcal{H}}_\rho$ with compact quotient $\overline{M}= \G
 \backslash \overline{\mathcal{H}}_{\rho}$, and $\overline{M}$ contains $M$ as an open dense set.  

 We will now describe the construction in detail in the case when $\rho:\G \to \Sp(2n,\RR)$ is a $Q_0$-Anosov representation, and $\overline{\mathcal{H}}$ is the bounded symmetric domain compactification of the symmetric space $\mathcal{H}_{\Sp(2n,\RR)}$. We then list other examples, to which an analogous construction applies.

 \subsection{Quotients of the Siegel space}
 \label{sec:siegel}
 Let us recall the geometric realization of the Borel embeddings for the Siegel space $\mathcal{H}_{\Sp(2n,\RR)}$. 
 Let $(\RR^{2n}, \omega)$ be a symplectic vector space and $(\CC^{2n},
 \omega_\CC)$ be its complexification, and $\phi: \Sp(2n,\RR) \to \Sp(2n,\CC)$ the corresponding embedding. 
 Let $h$ be the non-degenerate Hermitian form of signature $(n,n)$ on $\CC^{2n}$ defined by $h(v,w) = i \omega_\CC (\overline{v}, w)$. 
 Then $h$ is preserved by $\phi(\Sp(2n,\RR))$. 
 The symmetric space $\mathcal{H}_{\Sp(2n,\RR)}$ admits a $\phi$-equivariant embedding into the complex Lagrangian Grassmannian $\Ll(\CC^{2n})$, 
 namely 
 \[
 \mathcal{H}_{\Sp(2n,\RR)} \cong \mathcal{H} = \{ W \in \Ll(\CC^{2n} )\,|\, h|_W  > 0\},  
 \] 
 where $h|_W>0$ means that $h$ restricted to $W$ is positive definite. 
The natural compactification 
\[
\overline{\mathcal{H}}_{\Sp(2n,\RR)} =  \overline{\mathcal{H}} = \{ W \in \Ll(\CC^{2n} )\,|\, h|_W  \geq0\}, 
\]
where $h|_W\geq 0$ means that $h$ restricted to $W$ is positive semi-definite, is isomorphic to the bounded symmetric domain compactification  of $\mathcal{H}_{\Sp(2n,\RR)}$.

The compactification $\overline{\mathcal{H}}$ decomposes into
$\Sp(2n,\RR)$-orbits $\mathcal{H}_k$, $k = 0,\dots , n$ with $\mathcal{H}_0 =
\mathcal{H}$ and $\mathcal{H}_n \cong\Ll(\RR^{2n})$. The other $G$-orbits
$\mathcal{H}_k$ have the structure of a fiber bundle over $\Ff_k = \{ V
\subset \RR^{2n}\, |\, \dim(V) = k,\, \omega|_{V} = 0\}$, the space of
isotropic $k$-dimensional subspace in $\RR^{2n}$. The fiber over $V
\in \Ff_k$ is 
\[
 \{ W \in \Ll(\CC^{2n}) \, |\, h|_{W} \geq 0, \, W\cap \overline {W} =
 V \otimes_\RR \CC\}. 
\]

\begin{thm}\label{thm:compactify}
Let $\G \to \Sp(2n,\RR)$ be a $Q_0$-Anosov representation. 
Then there exists a compactification $\overline{M}$ of $M = \G \backslash \mathcal{H}$ such that $\overline{M}$ carries a $(\Sp(2n,\RR), \overline{\mathcal{H}})$-structure and the inclusion $M \subset \overline{M}$ is an $\Sp(2n,\RR)$-map. 

More precisely, there exists an open $\rho(\G)$-invariant subset
$\overline{\mathcal{H}}_\rho$ of $\overline{\mathcal{H}}$ 
 containing $\mathcal{H}$ such that $\G$ acts properly discontinuously and cocompactly on $\overline{\mathcal{H}}_\rho$, and $\overline{M} = \G \backslash \overline{\mathcal{H}}_\rho$.
\end{thm}

In conclusion, $\overline{M}$ is a manifold with corners that is locally modelled on
$\overline{ \mathcal{H}}$. 

\begin{proof}
Let $\rho: \Gamma \to \Sp(2n,\RR)$ be the  $Q_0$-Anosov representation with associated Anosov map $\xi: \partial_\infty \G \to \PP(\RR^{2n})$. 
 Then $\phi\circ \rho: \Gamma \to \Sp(2n,\CC)$ is a $Q_0$-Anosov representation with Anosov map $\xi_\CC: \partial_\infty \G \to \PP(\CC^{2n})$. Let $\Omega_{\phi\circ\rho} \subset \Ll(\CC^{2n})$ be the domain of discontinuity associated to $\phi\circ \rho$ in Section~\ref{sec:orth-sympl-groups_anosov}. 
 
 Then, by Proposition~\ref{prop:dod_in_smaller}, $\rho(\G)$ acts properly discontinuously and cocompactly on 
 \[
 \overline{\mathcal{H}}_\rho:= \overline{\mathcal{H}} \cap \Omega_{\phi\circ\rho}. 
 \]
 
 Recall that 
 \[\Omega_{\phi\circ\rho} = \{ W \in \Ll(\CC^{2n}) \, |\, \forall t \in \partial_\infty \G, W\cap \xi_\CC(t) = 0\}.\]
 Since $\xi(t)$ is a real line, if $\xi_\CC(t) \subset W$, then $\xi_\CC(t) \subset \overline{W}$.  
 This implies that $\Omega_{\phi\circ\rho}$ contains the set $\{ W \in \Ll(\CC^{2n}) \, |\, W\cap \overline{W} = 0\}$, which contains $\mathcal{H}$. 
 \end{proof}
 We can describe $\overline{\mathcal{H}}_\rho$ more explicitly. For this we set 
 \[
 K_k(\xi) := \bigcup_{t\in \partial_\infty\Gamma} \{ V \in \Ff_k\, |\, \xi(t) \subset V\}. 
 \]

 Then 
 \[
\overline{\mathcal{H}}_\rho = \bigcup_{i=0}^{n}( \mathcal{H}_k  \moins \mathcal{H}_k|_{ K_k(\xi)}), 
\] 
where $ \mathcal{H}_k|_{ K_k(\xi)}$ is the restriction of the bundle to $K_k(\xi) \subset \mathcal{F}_k$ . 

\begin{cor}
Let $\rho: \pi_1(\Sigma) \to \Sp(2n,\RR)$ be a Hitchin representation, then there is a natural compactification of $M = \pi_1(\Sigma)\backslash \mathcal{H}$ as  $(\Sp(2n,\RR), \overline{\mathcal{H}})$-manifold.
\end{cor}

An analogous construction applies for example to: 
\begin{enumerate}
\item $Q_0$-Anosov representations $\rho: \G \to \SU(n,n)$, where $G_\CC = \SL(2n,\CC)$. Then $\phi \circ \rho: \G \to \SL(2n,\CC)$ is a $P_1$-Anosov representation and $\Omega_{\phi\circ \rho} \subset \Gr_n(\CC^{2n})$ is the domain of discontinuity described in Section~\ref{sec:coin}. 
\item $Q_0$-Anosov representations $\rho: \G \to \SO(n,n)$ where we
  consider $\phi: \SO(n,n) \to  \SL(2n,\RR)$.
 Then $\phi \circ \rho:
  \G \to \SL(2n,\RR)$ is a $P_1$-Anosov representation and
  $\Omega_{\phi\circ \rho} \subset \Gr_n(\RR^{2n})$ is the domain of
  discontinuity described in Section~\ref{sec:coin}.  
\end{enumerate}

\begin{remark}
In the case of an arbitrary Anosov representation $\rho: \G \to G$ we do not know how to construct a natural compactification of the quotient manifolds $M = \rho(\G) \backslash \mathcal{H}$, where $\mathcal{H}=G/K$ is the symmetric space. 

However, we propose to investigate the following approach. 
Recall that for a general semisimple Lie group $G$ the symmetric space $\mathcal{H} = G/K$ can be embedded into the space
of probability measures $\mathcal{M}(G/Q)$ on $G/Q$, where $Q< G$ is
any proper parabolic subgroup. The closure of the image $\overline{\mathcal{H}}$ is a
Furstenberg-compactification of $\mathcal{H}$.  

Assume that $\rho: \G \to G$ is an Anosov representation which admits a domain of discontinuity $\Omega _\rho \subset G/Q$ with compact quotient. Let $K_\rho = G/Q \moins \Omega_\rho$, and denote by $\mathcal{K}_\rho \subset \mathcal{M}(G/Q)$ the set of probability measures with support on $K_\rho$. Then the action of $\rho(\G)$ on $\mathcal{M}(G/Q) \moins \mathcal{K}_\rho$ is proper. Furthermore $\mathcal{M}(G/Q) \moins \mathcal{K}_\rho$  contains the image of $\mathcal{H}$. 

When is the action of $\rho(\G)$ on $\overline{\mathcal{H}} \cap (\mathcal{M}(G/Q) \moins \mathcal{K}_\rho)$ cocompact ?
\end{remark}

\section{Compact Clifford-Klein forms}\label{sec:cliffordklein}
In this section we apply the construction of domains of discontinuity
to construct compact Clifford-Klein forms of homogeneous
spaces.  The examples indicate that a more systematic treatment would be very
 interesting. For a recent survey on compact Clifford-Klein forms we
 refer the reader to \cite{Kobayashi_Yoshino}.

\subsection{$Q_1$-Anosov representations}\label{sec:CK_Q_1}
Let 
 $G= \SU(1,n)$, $ \Sp(1,n)$ or $G_\mathcal{O}$  (the isometry group of the
Cayley hyperbolic plane) and let $\Gamma < G$ be a convex cocompact subgroup.
Denote by $G \to \SO(p,q)$ the natural embedding, i.e.\ $(p,q) = (2,2n)$, $(4,4n)$, or $(8,8)$ respectively  (see Remark~\ref{rem:examples_orth}). Let $\rho: \Gamma \to \SO(p,q)$ be the corresponding embedding of $\Gamma$. 
Then $\rho$ is a $Q_1$-Anosov representation. The domain of discontinuity $\Omega_\rho \subset \Ff_0(\RR^{p,q})$ is empty if and only if $\G$ is a uniform lattice, see Remark~\ref{rem:examples_orth}. 
Consider $\phi: \SO(p,q) \to \SO(p, q+1)$. 
The composition $\phi \circ \rho: \Gamma \to \SO(p,q+1)$ is again a $Q_1$-Anosov representation, but now with a domain of discontinuity $\Omega_{\phi\circ \rho} \subset \Ff_0(\RR^{p,q+1})$ that is always nonempty. 
By Theorem~\ref{thm:dod_GF}, the action of $\phi\circ \rho (\Gamma)$ on $\Omega_{\phi\circ \rho}$ is properly discontinuous with compact quotient. 
\begin{prop}\label{prop:CK_one} Let $\G$, $\rho$ and $(p,q)$ be as
  above.

Then the embedding $\rho: \Gamma \to \SO(p,q)$ as well as any sufficiently small deformation of $\rho$ leads to a properly discontinuous action on the homogeneous space $\SO(p,q)/\SO(p-1,q)$. 

\begin{enumerate}
\item If $\G$ is a uniform lattice, then $\G\backslash \SO(p,q)/\SO(p-1,q)$ is compact.
\item If $\G$ is not a uniform lattice the quotient $\G\backslash
  \Omega_{\phi\circ \rho}$ is a compactification of $\G\backslash
  \SO(p,q)/\SO(p-1,q)$.
\end{enumerate}
\end{prop}
When $\G$ is torsion free, this compactification is a manifold. In general, $\G\backslash
  \SO(p,q)/\SO(p-1,q)$ is an orbifold. 
  
\begin{proof}
Let $v = (v_1, \dots , v_{p+q+1})$ be an isotropic vector in
$\RR^{p,q+1}$ with $v_{p+q+1} \neq 0$. Then, the stabilizer of $\RR v
\in \Ff_0(\RR^{p,q+1})$ in $\SO(p,q)$ is $\SO(p-1, q)$, and the orbit
$\SO(p,q) (\RR v) \subset  \Omega_{\phi\circ \rho}$, with equality if
and only if $\G$ is a uniform lattice. 
\end{proof}

\begin{remark}
The Clifford-Klein forms given by $\rho$ have been studied by Kobayashi, \cite{TKobayashi_deformation}. When $\G$ is a uniform lattice the only deformations that exist are deformations into the normalizer of $G$.
\end{remark}

\subsection{$Q_0$-Anosov representations}
Let $\Gamma < \SO(1,2n)$ be a convex cocompact subgroup, and $\SO(1,2n) \to \SO(2,2n)$ the standard embedding. The corresponding embedding $\rho: \Gamma \to \SO(2,2n)$ is a $Q_0$-Anosov representation. The domain of discontinuity $\Omega_\rho \subset \Ff_1$ is empty if and only if $\G$ is a uniform lattice (Remark~\ref{rem:examples_orth}).
Let $\phi: \SO(2,2n) \to \SO(2n+2, \CC)$ be the embedding into the complexification. 
Then $\phi\circ \rho$ is $Q_0$-Anosov. The construction of
Section~\ref{sec:orth-sympl-groups_anosov} provides a domain of discontinuity
$\Omega_{\phi\circ \rho} \subset \Ff_1(\CC^{2n+2})$, on which
$\phi\circ\rho(\Gamma)$ acts properly discontinuously with compact
quotient.

Barbot
\cite{Barbot_component}
 shows that the entire connected component of  $\rho$ in 
$\hom(\G , \SO(2,2n))$ consists of $Q_0$-Anosov representations. Using this and Theorem~\ref{thm:DoD_loc_cst} 
we deduce:

\begin{thm}\label{thm:CK_two}
If $\G$ is a uniform lattice, then any representation $\rho'$ in the connected component of $\rho$ in 
$\hom(\G , \SO(2,2n))$ 
leads to a properly discontinuous, and cocompact action on the homogeneous space $\SO(2,2n)/\U(1,n)$.
 
Let $\G < \SO(1,2n) $ be a convex cocompact subgroup and $\rho:\G \to \SO(1,2n) \to \SO(2,2n)$ the embedding. Then $\rho$ and any sufficiently small deformation of $\rho$ leads to a properly discontinuous action of $\G$ on $\SO(2,2n)/\U(1,n)$. The quotient $\G \backslash\Omega_{\phi \circ \rho}$ is a compactification of the quotient $\G \backslash\SO(2,2n)/\U(1,n)$.
\end{thm}
\begin{proof}
Let $L \in \Omega_{\phi\circ \rho} \subset \Ff_1(\CC^{2n+2})$ be a
$n+1$-plane such that $L\cap \overline{L} = 0$. A direct calculation
gives that the stabilizer of $L$ in $\SO(2,2n)$ is $\U(1,n)$, and $
\SO(2,2n)/\U(1,n)\subset \Omega_{\phi \circ \rho}$
with equality if and only if $\G$ is a cocompact lattice.
\end{proof}
Note that Theorem~\ref{thm:CK_two} extends, in the case of
$\SO(2,2n)$, a recent result of Kassel \cite[Theorem
1.1]{Kassel_deformation}, that small deformations of $\rho$ lead to
properly discontinuous 
action on the homogeneous
space $\SO(2,2n)/\U(1,n)$.

In particular, as is noted in \cite{Kassel_deformation},  Johnson and
Millson \cite{Johnson_Millson} constructed explicit bending
deformations with Zariski dense image in $\SO(2,2n)$ when $\Gamma$ is
an arithmetic lattice. This allows to conclude the following
\begin{cor}\cite[Corollary 1.2]{Kassel_deformation}
There exist Zariski dense subgroups $\Lambda< \SO(2,2n)$ acting
properly discontinuously, freely and cocompactly on the homogeneous
space $\SO(2,2n)/\U(1,n)$. 
\end{cor}

\providecommand{\bysame}{\leavevmode\hbox to3em{\hrulefill}\thinspace}
\providecommand{\MR}{\relax\ifhmode\unskip\space\fi MR }


\begin{thebibliography}{10}

\bibitem{Abels_Margulis_Soifer_Prox}
Herbert Abels, Grigori~A. Margulis, and Grigori~A. So{\u\i}fer,
  \emph{Semigroups containing proximal linear maps}, Israel J. Math.
  \textbf{91} (1995), no.~1-3, 1--30. \MR{1348303 (96i:22029)}

\bibitem{Barbot_anosov}
Thierry Barbot, \emph{Three-dimensional {A}nosov flag manifolds}, Geom. Topol.
  \textbf{14} (2010), no.~1, 153--191. \MR{2578303}

\bibitem{Barbot_component}
\bysame, \emph{Deformations of {F}uchsian {A}d{S} representations are
  quasi-{F}uchsian}, in preparation, 2011.

\bibitem{Barbot_Merigot_fusion}
Thierry Barbot and Quentin M{\'e}rigot, \emph{Anosov {A}d{S} representations
  are quasi-{F}uchsian}, arXiv/0710.0969, arXiv/0710.0618, to appear in Groups,
  Geometry and Dynamics, 2007.

\bibitem{Benoist_properness}
Yves Benoist, \emph{Actions propres sur les espaces homog\`enes r\'eductifs},
  Ann. of Math. (2) \textbf{144} (1996), no.~2, 315--347. \MR{1418901
  (97j:22023)}

\bibitem{Benoist}
\bysame, \emph{Propri\'et\'es asymptotiques des groupes lin\'eaires}, Geom.
  Funct. Anal. \textbf{7} (1997), 1--47.

\bibitem{Benoist_CD1}
\bysame, \emph{Convexes divisibles. {I}}, Algebraic groups and arithmetic, Tata
  Inst. Fund. Res., Mumbai, 2004, pp.~339--374. \MR{2094116 (2005h:37073)}

\bibitem{Benoist_convex_III}
\bysame, \emph{Convexes divisibles. {III}}, Ann. Sci. \'Ecole Norm. Sup. (4)
  \textbf{38} (2005), no.~5, 793--832. \MR{2195260 (2007b:22011)}

\bibitem{Benoist_convex_IV}
\bysame, \emph{Convexes divisibles. {IV}. {S}tructure du bord en dimension 3},
  Invent. Math. \textbf{164} (2006), no.~2, 249--278. \MR{2218481
  (2007g:22007)}

\bibitem{Benoist_quasi}
\bysame, \emph{Convexes hyperboliques et quasiisom\'etries}, Geom. Dedicata
  \textbf{122} (2006), 109--134. \MR{2295544 (2007k:20091)}

\bibitem{Benoist_Survey}
\bysame, \emph{A survey on divisible convex sets}, Geometry, analysis and
  topology of discrete groups, Adv. Lect. Math. (ALM), vol.~6, Int. Press,
  Somerville, MA, 2008, pp.~1--18. \MR{2464391}

\bibitem{Bergeron_Gelander}
Nicolas Bergeron and Tsachik Gelander, \emph{A note on local rigidity}, Geom.
  Dedicata \textbf{107} (2004), 111--131. \MR{2110758 (2005k:22015)}

\bibitem{Bestvina_Mess_Boundary}
Mladen Bestvina and Geoffrey Mess, \emph{The boundary of negatively curved
  groups}, J. Amer. Math. Soc. \textbf{4} (1991), no.~3, 469--481. \MR{1096169
  (93j:20076)}

\bibitem{Bourdon_conforme}
Marc Bourdon, \emph{Structure conforme au bord et flot g\'eod\'esique d'un
  {${\rm CAT}(-1)$}-espace}, Enseign. Math. (2) \textbf{41} (1995), no.~1-2,
  63--102. \MR{1341941 (96f:58120)}

\bibitem{Bowditch_convergence}
Brian~H. Bowditch, \emph{Convergence groups and configuration spaces},
  Geometric group theory down under ({C}anberra, 1996), de Gruyter, Berlin,
  1999, pp.~23--54. \MR{1714838 (2001d:20035)}

\bibitem{Bradlow_GarciaPrada_Gothen}
Steven~B. Bradlow, Oscar Garc{\'{\i}}a-Prada, and Peter~B. Gothen,
  \emph{Surface group representations and {${\rm U}(p,q)$}-{H}iggs bundles}, J.
  Differential Geom. \textbf{64} (2003), no.~1, 111--170. \MR{2015045
  (2004k:53142)}

\bibitem{Bradlow_GarciaPrada_Gothen_survey}
\bysame, \emph{Maximal surface group representations in isometry groups of
  classical {H}ermitian symmetric spaces}, Geom. Dedicata \textbf{122} (2006),
  185--213. \MR{2295550 (2008e:14013)}

\bibitem{Burger_Iozzi_Labourie_Wienhard}
Marc Burger, Alessandra Iozzi, Fran{\c{c}}ois Labourie, and Anna Wienhard,
  \emph{Maximal representations of surface groups: {S}ymplectic {A}nosov
  structures}, Pure Appl. Math. Q. \textbf{1} (2005), no.~3, Special Issue: In
  memory of Armand Borel. Part 2, 543--590. \MR{2201327 (2007d:53064)}

\bibitem{Burger_Iozzi_Wienhard_tight}
Marc Burger, Alessandra Iozzi, and Anna Wienhard, \emph{Tight homomorphisms and
  {H}ermitian symmetric spaces}, Geom. Funct. Anal. \textbf{19} (2009), no.~3,
  678--721. \MR{2563767}

\bibitem{Burger_Iozzi_Wienhard_anosov}
\bysame, \emph{Maximal representations and {A}nosov structures}, in
  preparation, 2010.

\bibitem{Burger_Iozzi_Wienhard_toledo}
\bysame, \emph{Surface group representations with maximal {T}oledo invariant},
  Ann. of Math. (2) \textbf{172} (2010), no.~1, 517--566. \MR{2680425}

\bibitem{Burger_Iozzi_Wienhard_survey}
\bysame, \emph{Higher {T}eichm\"uller spaces: From {$\mathrm{SL}(2,
  \mathbf{R})$} to other {L}ie groups}, arXiv/1004.2894, to appear in the
  Handbook of Teichm\"uller Theory vol. IV, 2011.

\bibitem{Cano}
Angel Cano, \emph{Schottky groups can not act on {$\bold P^{2n}_{\bold C}$} as
  subgroups of {${\rm PSL}_{2n+1}(\bold C)$}}, Bull. Braz. Math. Soc. (N.S.)
  \textbf{39} (2008), no.~4, 573--586. \MR{2465265 (2010c:20064)}

\bibitem{Champetier}
Christophe Champetier, \emph{Petite simplification dans les groupes
  hyperboliques}, Ann. Fac. Sci. Toulouse Math. (6) \textbf{3} (1994), no.~2,
  161--221. \MR{1283206 (95e:20050)}

\bibitem{Goldman_Choi}
Suhyoung Choi and William~M. Goldman, \emph{Convex real projective structures
  on closed surfaces are closed}, Proc. Amer. Math. Soc. \textbf{118} (1993),
  no.~2, 657--661. \MR{1145415 (93g:57017)}

\bibitem{Coornaert_Delzant_Papadopoulos}
Michel Coornaert, Thomas Delzant, and Athanase Papadopoulos, \emph{Les groupes
  hyperboliques de {G}romov. [{G}romov hyperbolic groups], with an english
  summary}, Lecture Notes in Mathematics, vol. 1441, Springer-Verlag, Berlin,
  1990, G\'eom\'etrie et th\'eorie des groupes. \MR{1075994 (92f:57003)}

\bibitem{Delzant_Guichard_Labourie_Mozes}
Thomas {Delzant}, Olivier {Guichard}, Fran\c{c}ois {Labourie}, and Shahar
  {Mozes}, \emph{Displacing {R}epresentations and {O}rbit {M}aps}, Geometry,
  Rigidity, and Group Actions (Benson Farb, David Fisher, and Robert~J. Zimmer,
  eds.), University of Chicago Press, 2011, pp.~494--514.

\bibitem{Eilenberg_Steenrod}
Samuel Eilenberg and Norman Steenrod, \emph{Foundations of algebraic topology},
  Princeton University Press, Princeton, New Jersey, 1952. \MR{0050886
  (14,398b)}

\bibitem{Fock_Goncharov}
Vladimir Fock and Alexander Goncharov, \emph{Moduli spaces of local systems and
  higher {T}eichm\"uller theory}, Publ. Math. Inst. Hautes \'Etudes Sci.
  (2006), no.~103, 1--211. \MR{2233852}

\bibitem{Frances_Lorentzian}
Charles Frances, \emph{Lorentzian {K}leinian groups}, Comment. Math. Helv.
  \textbf{80} (2005), no.~4, 883--910. \MR{2182704 (2006h:22009)}

\bibitem{Fulton_Harris}
William Fulton and Joe Harris, \emph{Representation theory, a first course},
  Graduate Texts in Mathematics, vol. 129, Springer-Verlag, New York, 1991,
  Readings in Mathematics. \MR{1153249 (93a:20069)}

\bibitem{GarciaPrada_Gothen_Mundet}
Oscar Garc\'{\i}a-Prada, Peter~B. Gothen, and Ignasi Mundet~i Riera,
  \emph{Higgs bundles and surface group representations in the real symplectic
  group}, arXiv/0809.0576, 2008.

\bibitem{Goldman_geometric}
William~M. Goldman, \emph{Geometric structures on manifolds and varieties of
  representations}, Geometry of group representations ({B}oulder, {CO}, 1987),
  Contemp. Math., vol.~74, Amer. Math. Soc., Providence, RI, 1988,
  pp.~169--198. \MR{957518 (90i:57024)}

\bibitem{Goldman_88}
\bysame, \emph{Topological components of spaces of representations}, Invent.
  Math. \textbf{93} (1988), no.~3, 557--607.

\bibitem{Goldman_convex}
\bysame, \emph{Convex real projective structures on compact surfaces}, J.
  Differential Geom. \textbf{31} (1990), no.~3, 791--845.

\bibitem{Gothen}
Peter~B. Gothen, \emph{Components of spaces of representations and stable
  triples}, Topology \textbf{40} (2001), no.~4, 823--850. \MR{1851565
  (2002k:14017)}

\bibitem{Gromov_hyp}
Misha Gromov, \emph{Hyperbolic groups}, Essays in group theory, Math. Sci. Res.
  Inst. Publ., vol.~8, Springer, New York, 1987, pp.~75--263. \MR{919829
  (89e:20070)}

\bibitem{Guichard_Kapovich_Wienhard}
Olivier Guichard, Misha Kapovich, and Anna Wienhard, \emph{Schottky groups are
  {A}nosov}, in preparation, 2011.

\bibitem{Guichard_Wienhard_Duke}
Olivier Guichard and Anna Wienhard, \emph{Convex foliated projective structures
  and the {H}itchin component for {${\rm PSL}\sb 4({\bf R})$}}, Duke Math. J.
  \textbf{144} (2008), no.~3, 381--445. \MR{2444302}

\bibitem{Guichard_Wienhard_CRAS}
\bysame, \emph{Domains of discontinuity for surface groups}, C. R. Math. Acad.
  Sci. Paris \textbf{347} (2009), no.~17-18, 1057--1060. \MR{2554576}

\bibitem{Guichard_Wienhard_InvaMaxi}
\bysame, \emph{Topological invariants of {A}nosov representations}, J. Topol.
  \textbf{3} (2010), no.~3, 578--642. \MR{2684514}

\bibitem{Guichard_Wienhard_DoDSymp}
\bysame, \emph{Domains of discontinuity for maximal symplectic
  representations}, in preparation, 2011.

\bibitem{Guivarch_Ji_Taylor}
Yves Guivarc'h, Lizhen Ji, and John~C. Taylor, \emph{Compactifications of
  symmetric spaces}, Progress in Mathematics, vol. 156, Birkh\"auser Boston
  Inc., Boston, MA, 1998.

\bibitem{Hatcher_AT}
Allen Hatcher, \emph{Algebraic topology}, Cambridge University Press,
  Cambridge, 2002. \MR{1867354 (2002k:55001)}

\bibitem{Hernandez}
Luis {H}ern\`andez, \emph{Maximal representations of surface groups in bounded
  symmetric domains}, Trans. Amer. Math. Soc. \textbf{324} (1991), 405--420.
  \MR{1033234 (91f:32040)}

\bibitem{Hitchin}
Nigel~J. Hitchin, \emph{Lie groups and {T}eichm\"uller space}, Topology
  \textbf{31} (1992), no.~3, 449--473. \MR{1174252 (93e:32023)}

\bibitem{Johnson_Millson}
Dennis Johnson and John~J. Millson, \emph{Deformation spaces associated to
  compact hyperbolic manifolds}, Discrete groups in geometry and analysis (New
  Haven, Conn., 1984), Progr. Math., vol.~67, Birkh\"auser Boston, Boston, MA,
  1987, pp.~48--106.

\bibitem{Kapovich_Convex}
Michael Kapovich, \emph{Convex projective structures on {G}romov-{T}hurston
  manifolds}, Geom. Topol. \textbf{11} (2007), 1777--1830. \MR{2350468
  (2008h:53045)}

\bibitem{KapovichLeebMillson_triangle}
Michael Kapovich, Bernhard Leeb, and John~J. Millson, \emph{The generalized
  triangle inequalities in symmetric spaces and buildings with applications to
  algebra}, Mem. Amer. Math. Soc. \textbf{192} (2008), no.~896, viii+83.
  \MR{2369545 (2009d:22018)}

\bibitem{Karpelevic53}
Fridrikh~I. Karpelevi{\v{c}}, \emph{Surfaces of transitivity of a semisimple
  subgroup of the group of motions of a symmetric space}, Doklady Akad. Nauk
  SSSR (N.S.) \textbf{93} (1953), 401--404. \MR{0060283 (15,647g)}

\bibitem{Kassel_deformation}
Fanny Kassel, \emph{Deformation of proper actions on reductive homogeneous
  spaces}, to appear in Mathematische Annalen, DOI: 10.1007/s00208-011-0672-1,
  2009.

\bibitem{Kleiner_Leeb}
Bruce Kleiner and Bernhard Leeb, \emph{Rigidity of invariant convex sets in
  symmetric spaces}, Invent. Math. \textbf{163} (2006), no.~3, 657--676.
  \MR{2207236 (2006k:53064)}

\bibitem{Knapp_LieGrp}
Anthony~W. Knapp, \emph{Lie groups beyond an introduction}, second ed.,
  Progress in Mathematics, vol. 140, Birkh\"auser Boston Inc., Boston, MA,
  2002. \MR{1920389 (2003c:22001)}

\bibitem{TKobayashi_deformation}
Toshiyuki Kobayashi, \emph{Deformation of compact {C}lifford-{K}lein forms of
  indefinite-{R}iemannian homogeneous manifolds}, Math. Ann. \textbf{310}
  (1998), no.~3, 395--409. \MR{1612325 (99b:53074)}

\bibitem{Kobayashi_Yoshino}
Toshiyuki Kobayashi and Taro Yoshino, \emph{Compact {C}lifford-{K}lein forms of
  symmetric spaces---revisited}, Pure Appl. Math. Q. \textbf{1} (2005), no.~3,
  part 2, 591--663. \MR{2201328 (2007h:22013)}

\bibitem{Koszul}
Jean-Louis Koszul, \emph{D\'eformations de connexions localement plates}, Ann.
  Inst. Fourier (Grenoble) \textbf{18} (1968), no.~fasc. 1, 103--114.
  \MR{0239529 (39 \#886)}

\bibitem{Labourie_anosov}
Fran{\c{c}}ois Labourie, \emph{Anosov flows, surface groups and curves in
  projective space}, Invent. Math. \textbf{165} (2006), no.~1, 51--114.
  \MR{2221137}

\bibitem{Labourie_crossratio}
\bysame, \emph{Cross ratios, surface groups, {${\rm PSL}(n,{\bf R})$} and
  diffeomorphisms of the circle}, Publ. Math. Inst. Hautes \'Etudes Sci.
  (2007), no.~106, 139--213. \MR{2373231}

\bibitem{Labourie_energy}
\bysame, \emph{Cross ratios, {A}nosov representations and the energy functional
  on {T}eichm\"uller space}, Ann. Sci. \'Ec. Norm. Sup\'er. (4) \textbf{41}
  (2008), no.~3, 437--469. \MR{2482204}

\bibitem{Mineyev_flow}
Igor Mineyev, \emph{Flows and joins of metric spaces}, Geom. Topol. \textbf{9}
  (2005), 403--482 (electronic). \MR{2140987 (2006b:37059)}

\bibitem{Minsky}
Yair Minsky, \emph{On dynamics of ${O}ut({F}_n)$ on ${P}{S}{L}(2,{C})$
  characters}, arXiv:0906.3491, 2009.

\bibitem{Mostow_dec}
George~D. Mostow, \emph{Some new decomposition theorems for semi-simple
  groups}, Mem. Amer. Math. Soc. \textbf{1955} (1955), no.~14, 31--54.
  \MR{0069829 (16,1087g)}

\bibitem{Nori}
Madhav~V. Nori, \emph{The {S}chottky groups in higher dimensions}, The
  {L}efschetz centennial conference, {P}art {I} ({M}exico {C}ity, 1984),
  Contemp. Math., vol.~58, Amer. Math. Soc., Providence, RI, 1986,
  pp.~195--197. \MR{860414 (88c:32017)}

\bibitem{Parreau_distance}
Anne Parreau, \emph{La distance vectorielle dans les immeubles affines et les
  espaces sym\'etriques}, Preprint 2010.

\bibitem{Quint_cc}
Jean-Fran{\c{c}}ois Quint, \emph{Groupes convexes cocompacts en rang
  sup\'erieur}, Geom. Dedicata \textbf{113} (2005), 1--19. \MR{2171296
  (2006h:22010)}

\bibitem{Quint_Bourbaki}
\bysame, \emph{Convexes divisibles (d'apr\`es {Y}ves {B}enoist)}, Ast\'erisque
  (2010), no.~332, 45--73, S{\'e}minaire Bourbaki. Volume 2008/2009.
  Expos{\'e}s 997--1011. \MR{2648674}

\bibitem{Sambarino}
Andr\'es Sambarino, \emph{Quantitative properties of convex representations},
  Preprint, arXiv:1104.4705, 2011.

\bibitem{Satake_book}
Ichir{\^o} Satake, \emph{Algebraic structures of symmetric domains}, Kan\^o
  Memorial Lectures, vol.~4, Iwanami Shoten, Tokyo, 1980. \MR{591460
  (82i:32003)}

\bibitem{Seade_Verjovsky_Schottky}
Jos{\'e} Seade and Alberto Verjovsky, \emph{Complex {S}chottky groups},
  Ast\'erisque (2003), no.~287, xx, 251--272, Geometric methods in dynamics.
  II. \MR{2040008 (2005d:20091)}

\bibitem{Steenrod}
Norman Steenrod, \emph{The {T}opology of fibre bundles}, Princeton Mathematical
  Series, vol. 14, Princeton University Press, Princeton, N. J., 1951.
  \MR{0039258 (12,522b)}

\bibitem{Thurston}
William Thurston, \emph{Geometry and topology of 3-manifolds}, Notes from
  Princeton University, Princeton, NJ, 1978.

\bibitem{Thurston_stretch}
\bysame, \emph{{Minimal stretch maps between hyperbolic surfaces}},
  arXiv:math/9801039, January 1998.

\bibitem{Toledo_89}
Domingo Toledo, \emph{Representations of surface groups in complex hyperbolic
  space}, J. Differential Geom. \textbf{29} (1989), no.~1, 125--133. \MR{978081
  (90a:57016)}

\bibitem{Vey_convexes}
Jacques Vey, \emph{Sur les automorphismes affines des ouverts convexes
  saillants}, Ann. Scuola Norm. Sup. Pisa (3) \textbf{24} (1970), 641--665.
  \MR{0283720 (44 \#950)}

\bibitem{Wienhard_thesis}
Anna Wienhard, \emph{Bounded cohomology and geometry}, Ph.D. thesis, Bonn
  {U}niversit{\"a}t, January 2005, Bonner Mathematische Schriften Nr. 368, Bonn
  2004; arXiv:math/0501258, p.~126 pages.

\bibitem{Wienhard_mapping}
\bysame, \emph{The action of the mapping class group on maximal
  representations}, Geom. Dedicata \textbf{120} (2006), 179--191. \MR{2252900
  (2008g:20112)}

\end{thebibliography}
\end{document}